\RequirePackage{fix-cm}
\documentclass[smallextended,numbook]{svjour3}       
\smartqed  
\usepackage{graphicx}
\usepackage{mathptmx}      
\usepackage{amsmath,amssymb} 
\usepackage{amsthm}
\allowdisplaybreaks

\usepackage{xcolor}

\usepackage{mathrsfs} 
\usepackage{url} 
\usepackage{latexsym} 
\usepackage{bm} 

\usepackage{mathtools}

\usepackage{booktabs,multirow}
\usepackage{array}
\newcolumntype{P}[1]{>{\centering\arraybackslash}p{#1}}
\usepackage{diagbox}

\usepackage[ruled,vlined,linesnumbered]{algorithm2e}
\DontPrintSemicolon
\SetKw{Goto}{go to}
\SetKwProg{Fn}{Function}{:}{}
\SetKwFor{Loop}{loop}{}{EndLoop}
\SetArgSty{textnormal}

\usepackage[sort&compress,numbers]{natbib}

\usepackage{thmtools,thm-restate}

\usepackage{hyperref}
\usepackage{cleveref}
\expandafter\def\csname ver@etex.sty\endcsname{3000/12/31} 
\usepackage{autonum} 

\usepackage{adjustbox}
\usepackage{xparse} 
\usepackage{pbox} 

\usepackage[subrefformat=parens]{subcaption}
\usepackage{enumitem} 
\hypersetup{
  colorlinks,
  linkcolor={green!35!black},
  citecolor={green!35!black},
  urlcolor={blue!50!black}
}

\DeclareMathOperator{\dom}{dom}

\DeclareMathOperator{\prox}{prox}

\DeclareMathOperator{\dist}{dist}

\DeclareMathOperator*{\argmin}{argmin}

\DeclareMathOperator{\EE}{\mathbb{E}}

\DeclarePairedDelimiter\ceil{\lceil}{\rceil}

\DeclarePairedDelimiterX\inner[2]{\langle}{\rangle}{{#1},{#2}}
\DeclarePairedDelimiter\abs{|}{|}
\DeclarePairedDelimiter\norm{\|}{\|}
\DeclarePairedDelimiter\set{\{}{\}}
\DeclarePairedDelimiter\prn{(}{)}
\DeclarePairedDelimiter\bra{[}{]}
\DeclarePairedDelimiterX\Set[2]{\{}{\}}{\mspace{2mu}{#1}\;\delimsize|\;{#2}\mspace{2mu}}
\DeclarePairedDelimiterX\Prn[2]{(}{)}{\mspace{2mu}{#1}\;\delimsize|\;{#2}\mspace{2mu}}
\DeclarePairedDelimiterX\Bra[2]{[}{]}{\mspace{2mu}{#1}\;\delimsize|\;{#2}\mspace{2mu}}

\newcommand{\N}{\mathbb N}

\newcommand{\R}{\mathbb R}

\renewcommand{\epsilon}{\varepsilon}

\NewDocumentCommand{\exsub}{s m O{} m}{%
  \IfBooleanT{#1}{\EE_{#2}\nolimits\bra*{#4}}%
  \IfBooleanF{#1}{\EE_{#2}\nolimits\bra[#3]{#4}}%
}

\newcommand{\mathInd}{\hphantom{{}={}}}
\newcommand{\by}[2][]{\text{\pbox[c]{\textwidth}{(by \pbox[t]{\textwidth}{\,\!#2)#1}}}}

\newcommand{\revise}[1]{#1}

\newcommand\doublecheck{{\checked\kern-0.4em\checked}}
\DeclareMathOperator{\LogSoftmax}{LogSoftmax}
\DeclarePairedDelimiter\normop{\|}{\|_{\mathrm{op}}}

\declaretheoremstyle[
  bodyfont=\normalfont\itshape,
]{thmsty}

\declaretheorem[
  name=Assumption,
  refname={Assumption,Assumptions},
  style=thmsty,
]{assumption}

\crefname{algorithm}{Algorithm}{Algorithms}
\crefname{line}{Line}{Lines}
\crefname{section}{Section}{Sections}
\crefname{appendix}{Appendix}{Appendices}
\crefname{table}{Table}{Tables}
\crefname{lemma}{Lemma}{Lemmas}
\crefname{theorem}{Theorem}{Theorems}
\crefname{proposition}{Proposition}{Propositions}
\crefname{remark}{Remark}{Remarks}
\crefname{figure}{Fig.}{Figs.}
\crefname{equation}{}{}
\Crefname{equation}{Eq.}{Eqs.}

\setlist[itemize]{
  topsep=0.4\baselineskip,
  itemsep=0\baselineskip,
  leftmargin=1.5em,
}

\setlist[enumerate]{
  font=\upshape,
  label=(\alph*),
  ref=(\alph*),
  topsep=0.4\baselineskip,
  itemsep=0\baselineskip,
  leftmargin=2em,
}

\newlist{enuminasm}{enumerate}{1} 
\setlist[enuminasm]{
  font=\upshape,
  label=(\alph*),
  ref=\theassumption(\alph*),
  topsep=0.4\baselineskip,
  itemsep=0\baselineskip,
  leftmargin=2em,
}
\crefalias{enuminasmi}{assumption}

\newlist{enuminthm}{enumerate}{1}
\setlist[enuminthm]{
  font=\upshape,
  label=(\alph*),
  ref=\thetheorem(\alph*),
  topsep=0.4\baselineskip,
  itemsep=0\baselineskip,
  leftmargin=2em,
}
\crefalias{enuminthmi}{theorem}

\newlist{enuminlem}{enumerate}{1}
\setlist[enuminlem]{
  font=\upshape,
  label=(\alph*),
  ref=\thelemma(\alph*),
  topsep=0.4\baselineskip,
  itemsep=0\baselineskip,
  leftmargin=2em,
}
\crefalias{enuminlemi}{lemma}

\usepackage{threeparttable}

\newcommand{\Proposed}{\textsf{Proposed}\xspace}
\newcommand{\PG}{\textsf{PG}\xspace}
\newcommand{\ABO}{\textsf{ABO}\xspace}
\newcommand{\DP}{\textsf{DP}\xspace}

\begin{document}

\title{%
  Accelerated-gradient-based generalized Levenberg--Marquardt method with oracle complexity bound and local quadratic convergence%
  \thanks{%
    This work was supported by JSPS KAKENHI (19H04069, 20H04145, 20K19748) and JST ERATO (JPMJER1903).
  }%
}

\titlerunning{Generalized LM with oracle complexity bound and local quadratic convergence}        

\author{%
  Naoki Marumo \and
  Takayuki Okuno \and
  Akiko Takeda
}


\institute{%
  N. Marumo \at
  Graduate School of Information Science and Technology,
  University of Tokyo,
  Tokyo 113-8656, Japan\\
  \revise{\email{marumo@mist.i.u-tokyo.ac.jp}}
  \and
  T. Okuno \at
  Faculty of Science and Technology,
  Seikei University,
  Tokyo 180-8633, Japan\\
  Center for Advanced Intelligence Project, RIKEN,
  Tokyo 103-0027, Japan
  \and
  A. Takeda \at
  Graduate School of Information Science and Technology,
  University of Tokyo,
  Tokyo 113-8656, Japan\\
  Center for Advanced Intelligence Project, RIKEN,
  Tokyo 103-0027, Japan
}

\date{Received: date / Accepted: date}

\maketitle

\begin{abstract}
  Minimizing the sum of a convex function and a composite function appears in various fields. The generalized Levenberg--Marquardt (LM) method, also known as the prox-linear method, has been developed for such optimization problems. The method iteratively solves strongly convex subproblems with a damping term.

  This study proposes a new generalized LM method for solving the problem with a smooth composite function. The method enjoys three theoretical guarantees: iteration complexity bound, oracle complexity bound, and local convergence under a H\"olderian growth condition. The local convergence results include local quadratic convergence under the quadratic growth condition; this is the first to extend the classical result for least-squares problems to a general smooth composite function. In addition, this is the first LM method with both an oracle complexity bound and local quadratic convergence under standard assumptions. These results are achieved by carefully controlling the damping parameter and solving the subproblems by the accelerated proximal gradient method equipped with a particular termination condition. Experimental results show that the proposed method performs well in practice for several instances, including classification with a neural network and nonnegative matrix factorization.
\end{abstract}

\section{Introduction}
We consider the nonconvex optimization problem of minimizing the sum of a convex function and a composite function:
\begin{align}
  \min_{x \in \R^d}\ 
  \set[\Big]{ F(x) \coloneqq g(x) + h(c(x))}.
  \label{eq:problem_main}
\end{align}
Here, $g: \R^d \to \R \cup \set{+\infty}$ is proper closed convex (but possibly nonsmooth), $h: \R^n \to \R$ is smooth convex, and $c: \R^d \to \R^n$ is smooth.
The function $g$ can also represent convex constraints by setting $g$ to an indicator function.
Optimization problems of this form appear in various fields, including statistics, signal processing, and machine learning. 
A typical example is regression and classification; the task is learning a function that predicts $y \in \R^q$ from $x \in \R^p$, given many data pairs, $(x_1, y_1), \dots, (x_N, y_N)$.
Let $\phi_w: \R^p \to \R^q$ be the function to be learned, and $\ell: \R^q \times \R^q \to \R$ be a function that measures the error.
The function $\phi_w$ is parametrized by $w \in \R^d$, and its modern example is neural networks.
Then, the task is formulated as
\begin{align}
  \min_{w \in \R^d} \ 
  \frac{1}{N} \sum_{i=1}^N \ell(\phi_w(x_i), y_i),
\end{align}
and this optimization problem is an instance of \cref{eq:problem_main}.

The \emph{Levenberg--Marquardt (LM) method} \citep{levenberg1944method,marquardt1963algorithm} was first introduced for least-squares problems, $h(\cdot) = \frac{1}{2}\norm{\cdot}^2$, and has been extended to the general problem \cref{eq:problem_main} \citep{nesterov2007modified,lewis2016proximal,drusvyatskiy2018error}.
The generalized LM method,\footnote{
  In this paper, we refer to the LM method extended to the general case as the generalized LM method or simply the LM method.
  In some literature (e.g., \citep{drusvyatskiy2018error}), it is called the proximal linear method or the prox-linear method.
}
given the $k$-th iterate $x_k \in \R^d$, defines the next one $x_{k+1}$ as an exact or approximate solution to the subproblem
\begin{align}
  \min_{x \in \R^d}\ 
  \set*{
    \bar F_{k, \mu}(x)
    \coloneqq
    g(x)
    +
    h \prn[\big]{
      c(x_k) + \nabla c(x_k)(x - x_k)
    }
    + \frac{\mu}{2} \norm*{x - x_k}^2
  },
  \label{eq:subproblem}
\end{align}
where $\mu > 0$.
This problem is strongly convex and hence has a unique optimum.
The objective function $\bar F_{k, \mu}$ is obtained by approximating $c$ in $F$ with the Jacobian matrix $\nabla c(x_k) \in \R^{n \times d}$ and adding a damping term that enforces the next iterate $x_{k+1}$ to be close to $x_k$.
As we will see in the next section, various LM methods have been proposed, and the difference lies in, for example, the choice of the damping parameter $\mu$.
We propose a new LM method that achieves theoretically fast global and local convergence.

To evaluate the global convergence rate of LM methods, several studies \citep{ueda2010global,zhao2016global,bergou2020convergence,marumo2023majorization,drusvyatskiy2018error} have derived upper bounds on \emph{iteration complexity}, which is the number of iterations required to find an $\epsilon$-stationary point.
The iteration complexity makes sense as a measure for the whole running cost if we explicitly compute $c(x_k)$ and $\nabla c(x_k)$ before solving subproblem~\cref{eq:subproblem} and their computational costs are dominant; then, the total cost of LM is proportional to the iteration complexity.
On the other hand, computing the $n \times d$ matrix $\nabla c(x_k)$ should be avoided from a practical standpoint for large-scale problems where $n$ and $d$ amount to thousands or millions.
For such problems, it is practical to access $c$ through the Jacobian-vector products $\nabla c(x_k) u$ and $\nabla c(x_k)^\top v$ without explicitly computing $\nabla c(x_k)$ \citep{drusvyatskiy2019efficiency}.\footnote{
  Automatic differentiation libraries such as JAX~\citep{jax2018github} compute the Jacobian-vector products at several times the cost of evaluating $c(x)$.
  See, e.g., the JAX documentation \citep{jax2022autodiff}.
}
In this case, the iteration complexity is no longer appropriate as a measure of complexity because it hides how many times the Jacobian-vector products are computed to solve subproblem~\cref{eq:subproblem}, and \emph{oracle complexity} is used instead.
The oracle complexity is the number of oracle calls required to find an $\epsilon$-stationary point.
In this paper, we assume that we have access to oracles that compute $c(x)$, $\nabla c(x) u$, $\nabla c(x)^\top v$, $h(y)$, $\nabla h(y)$, and
\begin{align}
  \prox_{\eta g}(x)
  \coloneqq
  \argmin_{z \in \R^d} \set*{
    g(z) + \frac{1}{2 \eta} \norm*{z - x}^2
  }
\end{align}
for given $x, u \in \R^d$, $y, v \in \R^n$, and $\eta > 0$, as in \citep{drusvyatskiy2019efficiency}.
We analyze the iteration complexity of the proposed LM method for small-scale problems and the oracle complexity for large-scale problems.

The local behavior of the sequence $(x_k)_{k \in \N}$ should also be investigated because some algorithms move faster in the neighborhood of a stationary point.
For example, Newton's method (with regularization) achieves local quadratic convergence but requires the costly computation of Hessian matrices; the gradient descent (GD) method does not use Hessian but fails quadratic convergence.
Some LM methods, including ours, achieve local quadratic convergence without Hessian by exploiting the composite structure of $h(c(x))$ in problem~\cref{eq:problem_main}.
\cref{fig:rosenbrock} shows that our LM method has a better local performance than GD.

\begin{figure}[t]
  \centering
  \includegraphics[width=0.49\linewidth]{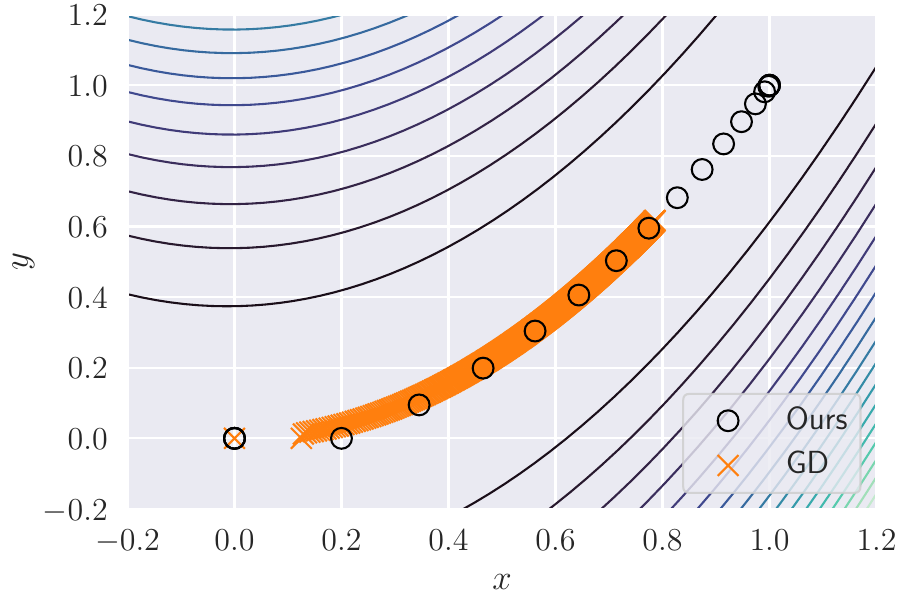}
  \hfill
  \includegraphics[width=0.49\linewidth]{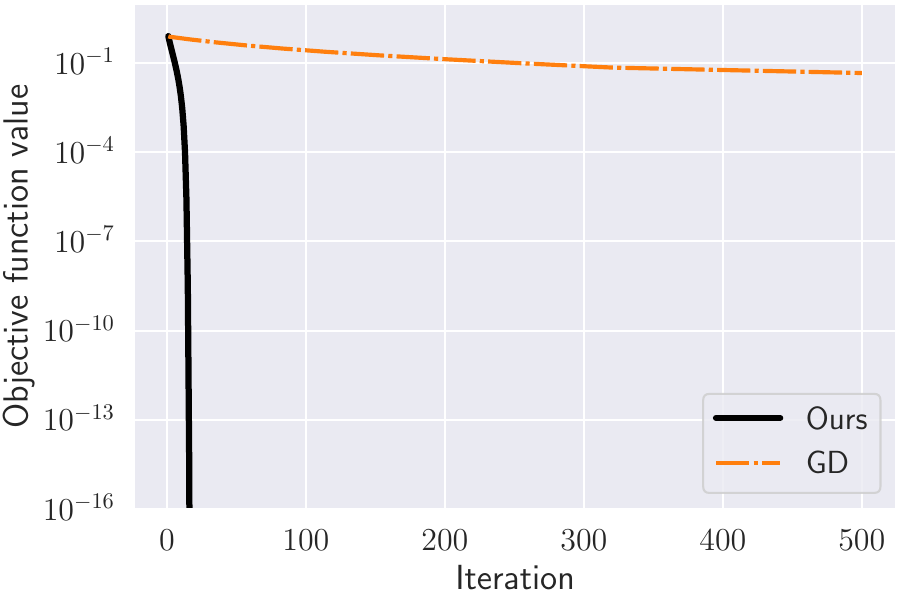}
  \caption{
    Minimization of the Rosenbrock function~\citep{rosenbrock1960automatic}: $\min_{(x, y) \in \R^2}\, (x - 1)^2 + 100 (y - x^2)^2$.
    GD is gradient descent where the step-size is adaptively chosen by standard backtracking with the Armijo condition.
    The starting point of each algorithm is $(0, 0)$ and the optimal solution is $(1, 1)$.
    The left figure shows the sequences generated by each method.
  }
  \label{fig:rosenbrock}
\end{figure}

\subsection*{Our contribution}
We define the set $S_{k, \mu}$ of approximate solutions to subproblem~\cref{eq:subproblem} (see \cref{eq:def_Sk} for the definition of $S_{k, \mu}$) and propose an LM method (\cref{alg:proposed_backtracking}) that computes a point in $S_{k, \mu}$ at each iteration $k \in \N$.
Our theoretical contributions are listed below.
\begin{itemize}
  \item 
  \textbf{Iteration complexity.}
  \revise{
    The proposed method has parameters $g^*$ and $h^*$ such that $g^* \leq \inf_{x \in \R^d} g(x)$ and $h^* \leq \inf_{y \in \R^n} h(y)$.
    Define $\Delta_k$ and $F^*$ by 
    \begin{align}
      \Delta_k
      &\coloneqq
      F(x_k) - (g^* + h^*),\quad
      F^*
      \coloneqq
      \inf_{x \in \R^d} F(x).
      \label{eq:def_Deltak_Fast}
    \end{align}%
    (Note that $\Delta_k \geq 0$ by definition.)
  }%
  Our LM method has the following iteration complexity bound for an $\epsilon$-stationary point:
  \begin{align}
    O \prn*{
      \frac{L_c \sqrt{L_h \Delta_0}}{\epsilon^2}
      (F(x_0) - F^*)
    },
    \label{eq:our_iteration_complexity}
  \end{align}
  where $L_c$ and $L_h$ denote the Lipschitz constants of $\nabla c$ and $\nabla h$.

  \item
  \textbf{Oracle complexity.}
  When a point in $S_{k, \mu}$ is computed by an accelerated method (\cref{alg:apg}), our LM method has the following oracle complexity bound for an $\epsilon$-stationary point:
  \begin{align}
    O \prn*{
      \frac{ L_c \sqrt{L_h \Delta_0}}{\epsilon^2}
      (F(x_0) - F^*)
      \prn*{
        1 + 
        \sqrt{\kappa}\log \kappa
      }
    },
    \label{eq:our_oracle_complexity}
  \end{align}
  where 
  \begin{align}
    \kappa
    \coloneqq
    1 + 
    \frac{\sqrt{L_h} \sigma^2}{L_c \sqrt{\Delta_0}},\quad
    \sigma \coloneqq \sup_{k \in \N} \normop*{\nabla c(x_k)},
    \label{eq:def_kappa_sigma}
  \end{align}
  and $\normop*{\cdot}$ denotes the operator norm.
  
  \item
  \textbf{Local convergence.}
  For our method, the sequences $(F(x_k))_{k \in \N}$ and $(\Delta_k)_{k \in \N}$ are ensured to converge to some values, say $F_\infty$ and $\Delta_\infty$.
  We analyze the order of convergence of $(F(x_k))_{k \in \N}$ assuming the following H\"olderian growth condition:
  for some $1 \leq r < 3$, $\gamma > 0$, and $K \in \N$,
  \begin{align}
    \frac{\gamma}{r} \dist(x_k, X^*)^r \leq F(x_k) - F_\infty,\quad
    \forall k \geq K,
    \label{eq:holderian}
  \end{align}
  where $\dist(x, X^*)$ denotes the distance between the point $x \in \R^d$ and the set $X^* \subseteq \R^d$ of stationary points.
  In particular,
  \begin{itemize}
    \item 
    if $r = 2$ and $\Delta_\infty = 0$, then $(F(x_k))_{k \in \N}$ converges quadratically;
    \item
    if $r = 2$ and $\Delta_\infty$ is small enough, then $(F(x_k))_{k \in \N}$ converges linearly.
  \end{itemize}
\end{itemize}
A significant result is that our LM method simultaneously achieves both an oracle complexity bound and local quadratic convergence under standard assumptions.
As shown in \cref{table:complexity} later, some existing LM methods have iteration and oracle complexity bounds, and other LM methods have an iteration complexity bound and local quadratic convergence.
However, to the best of our knowledge, no methods exist with an oracle complexity bound and local quadratic convergence, even for least squares.

Achieving an oracle complexity bound and local quadratic convergence simultaneously is challenging due to the following two conflicts.
First, the damping parameter $\mu$ needs to be small for local quadratic convergence, while a small $\mu$ increases the cost for the subproblem, hence the oracle complexity.
Second, local quadratic convergence requires subproblem \cref{eq:subproblem} to be solved with high accuracy, meaning that the approximate solution set $S_{k, \mu}$ needs to be small.
However, a smaller set $S_{k, \mu}$ increases the oracle complexity.
To resolve these conflicts, we carefully define $\mu$ and $S_{k, \mu}$; such $\mu$ and $S_{k, \mu}$ are small enough for local convergence, and at the same time, a point in $S_{k, \mu}$ can be found quickly using an accelerated method.
Thus, our method has succeeded in having such theoretical guarantees.

To investigate the practical performance of the proposed method, we conducted several numerical experiments.
The results show that our method converges faster than existing methods for large-scale problems with $10^4$ to $10^6$ variables.

\paragraph{Paper organization}
\cref{sec:related_work} reviews existing methods and compares our results with them.
\cref{sec:algorithm} describes our LM method and presents an accelerated method for the subproblems.
We also show a key lemma for the LM method and provide the complexity result for the accelerated methods.
\cref{sec:global_convergence,sec:local_convergence} show the global complexity and local convergence results, respectively.
\cref{sec:experiments} provides several numerical results, and \cref{sec:conclusion} concludes the paper.
Some lemmas and proofs are deferred to the appendix.

\paragraph{Notation}
Let $\N$ denote the set of nonnegative integers: $\N \coloneqq \set{0,1,\dots}$.
The symbol $\inner{\cdot}{\cdot}$ denotes the dot product, $\norm{\cdot}$ denotes the Euclidean norm, and $\normop{\cdot}$ denotes the operator norm of a matrix with $\norm{\cdot}$, which is equivalent to the largest singular value.
Let $\ceil{\cdot}$ denote the ceiling function.

\section{Comparison with existing methods}
\label{sec:related_work}
We review the existing theoretical results
and compare our work with them.
This section uses the same notations as used in \cref{eq:our_iteration_complexity,eq:our_oracle_complexity,eq:holderian}, and the algorithms discussed in this section are summarized in \cref{table:complexity}.

\begin{table}[t]
  \centering
  \begin{threeparttable}
    \caption{
      Comparison of methods for problem~\cref{eq:problem_main} with smooth $h$ and $c$.
      For readability, the iteration and oracle complexity bounds are divided by a common factor, $L_c \epsilon^{-2} (F(x_0) - F^*)$.
    }
    \label{table:complexity}
    \small
    \setlength{\tabcolsep}{0.4em}
    \def\arraystretch{1.3}
    \begin{tabular}{c|cccccc}\toprule
      Assumption for $h$                     & Algorithm                                                                                                                                                                                                                                                            & \#Iteration                                                                    & \#Oracle                                                                       & Local      \\\midrule
      $h(\cdot) = \frac{1}{2}\norm{\cdot}^2$ & Proximal gradient                                                                                                                                                                                                                                                    & $O \prn[\big]{ \sqrt{\Delta_0}  \kappa }$                                      & $O \prn[\big]{ \sqrt{\Delta_0}  \kappa }$                                      &            \\
                                             & \citep{ueda2010global,zhao2016global,bellavia2018levenberg}                                                                                                                                                                                                          & $O \prn[\big]{ \sqrt{\Delta_0}  \kappa }$                                      &                                                                                &            \\
                                             & \citep{yamashita2001rate,bellavia2014strong,bellavia2010convergence,behling2012unified,dan2002convergence,facchinei2013family,fan2005quadratic,fan2006convergence,fischer2010inexactness,kanzow2004levenberg,ahookhosh2019local,bao2019modified,wang2021convergence} &                                                                                &                                                                                & \checkmark \\
                                             & \citep{bergou2020convergence} A                                                                                                                                                                                                                                      & $\tilde O \prn[\big]{ \frac{\Delta_0}{\Delta_\infty} \sqrt{\Delta_0} \kappa }$ & $\tilde O \prn[\big]{ \frac{\Delta_0}{\Delta_\infty} \sqrt{\Delta_0} \kappa }$ &            \\
                                             & \citep{bergou2020convergence} B                                                                                                                                                                                                                                      & $\tilde O \prn[\big]{ \frac{\Delta_0}{\Delta_\infty} \sqrt{\Delta_0} \kappa }$ &                                                                                & \checkmark \\
                                             & \citep{marumo2023majorization} A                                                                                                                                                                                                                                     & $O \prn[\big]{ \sqrt{\Delta_0} \kappa }$                                       & $O \prn[\big]{ \sqrt{\Delta_0} \kappa }$                                       &            \\
                                             & \citep{marumo2023majorization} B                                                                                                                                                                                                                                     & $O \prn[\big]{ \sqrt{\Delta_0} \kappa }$                                       &                                                                                & \checkmark \\
                                             & \textbf{This work}                                                                                                                                                                                                                                                   & $O \prn[\big]{ \sqrt{ \Delta_0 } }$                                            & $\tilde O \prn[\big]{ \sqrt{ \Delta_0 } \sqrt{\kappa} }$                       & \checkmark \\\cmidrule{1-5}
      convex, smooth                         & Proximal gradient                                                                                                                                                                                                                                                    & $O \prn[\big]{ \sqrt{L_h \Delta_0} \kappa }$                                   & $O \prn[\big]{ \sqrt{L_h \Delta_0} \kappa }$                                   &            \\
                                             & \textbf{This work}                                                                                                                                                                                                                                                   & $O \prn[\big]{\sqrt{L_h \Delta_0} }$                                           & $\tilde O \prn[\big]{\sqrt{L_h \Delta_0 } \sqrt{\kappa} }$                     & \checkmark \\\cmidrule{1-5}
      convex, smooth, Lipschitz              & \citep{drusvyatskiy2019efficiency}                                                                                                                                                                                                                                   & $O \prn[\big]{ K_h }$                                                          & $\tilde O \prn[\big]{ K_h \sqrt{\kappa'} }$                                    &            \\\bottomrule
    \end{tabular}
    \begin{tablenotes}
      \item
      \small
      This table uses the same notations as used in \cref{eq:our_iteration_complexity,eq:our_oracle_complexity,eq:holderian}, and $\kappa'$ is defined in \cref{eq:complexity_drus}.
      \item
      Because $h(\cdot) = \frac{1}{2} \norm{\cdot}^2$ implies $L_h = 1$, the constant $L_h$ does not appear in the complexity bounds dedicated to that case.
      \item
      The notation $\tilde O$ hides logarithmic factors in $L_c$, $L_h$, $K_h$, $\sigma$, $\Delta_0$, $\kappa$, $\kappa'$, and $\epsilon^{-1}$.
      \item 
      The LM methods marked ``B'' solve subproblems more accurately than ``A'' for the sake of fast local convergence.
      \item
      The mark $\checkmark$ in ``Local'' indicates that the local quadratic convergence is ensured under $\Delta_\infty = 0$ and the quadratic growth condition.
    \end{tablenotes}
  \end{threeparttable}
\end{table}

\subsection{Methods for least squares}
For least-squares problems, $h(\cdot) = \frac{1}{2} \norm{\cdot}^2$, some LM methods \citep{ueda2010global,zhao2016global} achieve the iteration complexity bound of
\begin{align}
  O \prn*{
    \frac{L_c \sqrt{\Delta_0} \kappa}{\epsilon^2} (F(x_0) - F^*)
  }
  \label{eq:complexity_ueda}
\end{align}
with exact solutions to subproblem~\cref{eq:subproblem}, where $\kappa$ is defined in \cref{eq:def_kappa_sigma}.
\revise{
The LM method in \citep{bellavia2018levenberg} deals with cases where errors are included in the computation of $c$ and $\nabla c$ or in the subproblem solutions, achieving the same iteration complexity.
}
Solving the subproblem inexactly with first-order methods yields LM methods \citep{bergou2020convergence,marumo2023majorization} that enjoy the iteration and oracle complexity bounds, which are the same order as \cref{eq:complexity_ueda}.

For local quadratic convergence of LM-type methods, standard assumptions are $\Delta_\infty = 0$ and the H\"olderian growth condition~\cref{eq:holderian} with $r = 2$, also known as a \emph{quadratic growth} condition or an error-bound condition.
Pioneered by \citet{yamashita2001rate}, many studies have shown local quadratic convergence under these assumptions or essentially equivalent assumptions \citep{fan2005quadratic,fan2006convergence,dan2002convergence,fischer2010inexactness,bellavia2014strong,bellavia2010convergence,kanzow2004levenberg,behling2012unified,facchinei2013family,bergou2020convergence,marumo2023majorization}.
Recently, local convergence studies with different $\Delta_\infty$ and $r$ have been advanced.
\citet{bergou2020convergence} showed local linear convergence with small $\Delta_\infty$ and $r = 2$.
Superlinear convergence under $r \geq 2$ has been developed \citep{ahookhosh2019local,bao2019modified,wang2021convergence}.


Some studies achieve fast global and local convergence simultaneously.
\citet{marumo2023majorization} proposed an LM method with both an iteration complexity bound and local quadratic convergence guarantee under $\Delta_\infty = 0$ and the quadratic growth.
\footnote{
  \citet{bergou2020convergence} also consider these convergence properties, but the proposed algorithm has either property depending on $\Delta_\infty$; the iteration complexity bound for $\Delta_\infty > 0$ and local quadratic convergence for $\Delta_\infty = 0$.
}

The proximal gradient method can also be applied to problem~\cref{eq:problem_main}, having iteration and oracle complexity bounds.
See \cref{sec:complexity_pg} for the details.
We note that accelerated proximal gradient methods do not improve the theoretical complexity bound for nonconvex problems under standard assumptions \citep{cartis2010complexity,carmon2020lower}.

\subsection{Methods for convex $h$}
There is another research stream for solving \eqref{eq:problem_main} with general convex (but possibly nonsmooth) $h$.
Some studies \citep{cartis2011evaluation,drusvyatskiy2018error,drusvyatskiy2019efficiency,nesterov2007modified} provide complexity results assuming $h$ is globally Lipschitz continuous, while others \citep{drusvyatskiy2018error,burke1995gauss,li2002convergence,nesterov2007modified} do local convergence results assuming $h$ is \emph{sharp}.\footnote{
  See \citep[Eq.~(31)]{drusvyatskiy2018error} for the definition of sharpness.
  The sharpness implies non-differentiability at the optimal solution.
}
Note that these results do not cover the least-squares setting in the previous section because $h(\cdot) = \frac{1}{2} \norm{\cdot}^2$ is neither globally Lipschitz nor sharp.\footnote{
  Suppose that $g(\cdot) = 0$.
  One might think that setting $h(\cdot) = \norm{\cdot}$ could cover the least-squares settings because such $h$ is Lipschitz continuous, but that is not the case.
  That $h$ makes subproblem~\cref{eq:subproblem} nonsmooth and increases the computational cost compared to the case $h(\cdot) = \frac{1}{2} \norm{\cdot}^2$.
}
The most relevant result to our study is that by \citet{drusvyatskiy2019efficiency}; assuming $h$ is smooth in addition to Lipschitz continuity, the article \citep{drusvyatskiy2019efficiency} developed a method with the oracle complexity of
\begin{align}
  O \prn[\bigg]{
    \frac{L_c K_h}{\epsilon^2}
    (F(x_0) - F^*)
    \prn*{
      1 + 
      \sqrt{\kappa'}\log \kappa'
    }
  },
  \quad\text{where}\quad
  \kappa'
  \coloneqq
  1 + 
  \frac{L_h \sigma^2}{L_c K_h}
  \label{eq:complexity_drus}
\end{align}
and $K_h$ and $L_h$ are the Lipschitz constants of $h$ and $\nabla h$, respectively.
The main idea is employing an accelerated method to solve subproblem~\cref{eq:subproblem} approximately.

Removing the Lipschitz assumption for complexity analysis seems non-trivial.
The assumption yields a simple upper bound on the objective function (see \citep[Lemma~3.2]{drusvyatskiy2019efficiency}), and existing complexity results such as \cref{eq:complexity_drus} are based on the upper bound.

\revise{
\subsection{Methods for nonconvex $g$}
All the studies mentioned above assume the convexity of $g$.
Some recent studies \citep{aravkin2022proximal,aravkin2023levenberg} deal with problems where the regularizer $g$ is nonconvex.
\citet{aravkin2023levenberg} developed an LM method for regularized least squares with iteration and oracle complexity of $O(\epsilon^{-2})$; 
\citet{aravkin2022proximal} developed a quasi-Newton method for a more general problem with similar complexity bounds.
However, these algorithms do not guarantee the local quadratic convergence unlike the proposed method and many of the existing LM methods in \cref{table:complexity}.
}

\subsection{Our method \revise{in comparison with} previous work}
\label{sec:related_work_comparison}
Our method solves problem~\cref{eq:problem_main} without the Lipschitz continuity of $h$, and hence generalizes the classical LM method for least squares.
Our iteration complexity bound~\cref{eq:our_iteration_complexity} improves the existing bound~\cref{eq:complexity_ueda} by $\Theta(\kappa)$, and our oracle complexity bound~\cref{eq:our_oracle_complexity} does the existing bound by $\tilde \Theta(\sqrt{\kappa})$.
On the other hand, the oracle complexity result can also be viewed as removing the Lipschitz assumption for $h$ in the existing result in \cref{eq:complexity_drus}.
Technically, this alleviation is accomplished by taking full advantage of the convexity and smoothness of $h$ to construct another upper bound for the objective function (see \cref{eq:FxFbarx_diff_upperbound}) that is different from the existing one, \citep[Lemma~3.2]{drusvyatskiy2019efficiency}.
Our complexity bounds are also better than the bounds of the proximal gradient method.

Our method also achieves local quadratic convergence under $\Delta_\infty = 0$ and the quadratic growth.
To the best of our knowledge, this is the first generalization of the classical quadratic convergence result for $h(\cdot) = \frac{1}{2} \norm{\cdot}^2$ to general smooth functions $h$.
Furthermore, we analyze our method under a more general condition, the H\"olderian growth~\cref{eq:holderian} with $1 \leq r < 3$.
\revise{Our linear convergence result for small $\Delta_\infty > 0$ and $r = 2$ is similar to that of \citep{bergou2020convergence}, but their assumptions are different.
The details will be discussed later in \cref{rem:compare_bergou2020}.}

\revise{
Our complexity bounds \cref{eq:our_iteration_complexity,eq:our_oracle_complexity} are derived with the appropriate setting of the algorithm's parameter based on $L_c$ and $L_h$.
However, even if the parameter is set arbitrarily, we can obtain bounds of the same order with respect to $\epsilon$; see \cref{thm:complexity_global} for the formal statement.
It is a common practice to make assumptions in the parameter setting of algorithms in order to derive simpler or better complexity bounds (e.g., Assumption~4 in \citep{bellavia2018levenberg} and the paragraph below Theorem~6.4 in \citep{drusvyatskiy2019efficiency}).
}

\revise{
Our algorithm needs parameters $g^*$ and $h^*$ satisfying $g^* \leq \inf_{x \in \R^d} g(x)$ and $h^* \leq \inf_{y \in \R^n} h(y)$ under the assumption of $\inf_{x \in \R^d} g(x) > - \infty$ and $\inf_{y \in \R^n} h(y) > - \infty$.
Especially for local quadratic convergence, we have to set $g^* = \inf_{x \in \R^d} g(x)$ and $h^* = \inf_{y \in \R^n} h(y)$, which is implied by  $\Delta_\infty = 0$, and so we need the knowledge of the infimum of $g$ and $h$.
Although this assumption may seem quite restrictive, it is not stronger than that made by existing analyses of LM methods for least squares (or non-Lipschitz $h$).\footnote{
\revise{The method of \citep{drusvyatskiy2019efficiency} does not require the knowledge of the infimum of $g$ and $h$ but assumes that $h$ is Lipschitz continuous.}
}
We know $\inf_{y \in \R^n} h(y) = 0$ if $h(\cdot) = \frac{1}{2} \norm{\cdot}^2$; existing LM methods \citep{yamashita2001rate,bellavia2014strong,bellavia2010convergence,behling2012unified,dan2002convergence,facchinei2013family,fan2005quadratic,fan2006convergence,fischer2010inexactness,kanzow2004levenberg,ahookhosh2019local,bao2019modified,wang2021convergence,ueda2010global,zhao2016global,bergou2020convergence,marumo2023majorization} for least squares assume that $g(\cdot) = 0$ or that $g$ is the indicator function of a convex set, and hence, we know $\inf_{x \in \R^d} g(x) = 0$.
We believe that it is non-trivial and interesting that the knowledge of the infimum of $g$ and $h$ suffices to generalize the results for least-squares to a smooth convex $h$.
}

\section{Our generalized Levenberg--Marquardt method}
\label{sec:algorithm}
This section describes our LM method and show a key lemma called a \emph{descent lemma}, which will be used for our analysis.
We also present an accelerated method for solving subproblem~\cref{eq:subproblem} and provide a complexity result for the method.

\subsection{Notations}
\label{sec:preliminary}
Let us introduce some notations to describe the algorithm.
Recall that the functions $F$ and $\bar F_{k, \mu}$ are defined in \cref{eq:problem_main,eq:subproblem}.
Let $H$ and $\bar H_{k, \mu}$ be the differentiable part of $F$ and $\bar F_{k, \mu}$, respectively:
\begin{align}
  H(x)
  &\coloneqq
  h(c(x)),\quad
  \bar H_{k, \mu}(x)
  \coloneqq
  h \prn[\big]{
    c(x_k) + \nabla c(x_k)(x - x_k)
  }
  + \frac{\mu}{2} \norm*{x - x_k}^2,
  \label{eq:def_H_Hbar}
\end{align}
and we then have
\begin{align}
  F(x)
  = g(x) + H(x),\quad
  \bar F_{k, \mu}(x)
  =
  g(x) + \bar H_{k, \mu}(x).
  \label{eq:def_Fbar}
\end{align}
Let $\dom g$ denote the effective domain of $g$, and let $\partial g(x)$ denote the subdifferential of $g$ at $x \in \dom g$:
\begin{align}
  \dom g
  &\coloneqq
  \Set*{x \in \R^d}{g(x) < + \infty},\\
  \partial g(x)
  &\coloneqq
  \Set*{p \in \R^d}{g(y) - g(x) \geq \inner{p}{y - x}, \ \forall y \in \R^d }.
  \label{eq:def_subdif_g}
\end{align}
We define the first-order optimality measures for problems~\cref{eq:problem_main,eq:subproblem} as follows:
\begin{align}
  \omega(x)
  &\coloneqq
  \min_{p \in \partial g(x)} \norm*{p + \nabla H(x)},\quad
  \bar{\omega}_{k, \mu}(x)
  \coloneqq
  \min_{p \in \partial g(x)} \norm*{p + \nabla \bar H_{k, \mu}(x)}.
  \label{eq:def_error}
\end{align}
Our goal is to find a point $x \in \dom g$ such that $\omega(x) \leq \epsilon$, an $\epsilon$-stationary point of problem~\cref{eq:problem_main}.

\subsection{Assumptions on problem \cref{eq:problem_main}}
We formally state basic assumptions for our LM method.
\begin{assumption}
  \leavevmode
  \label{asm:hc}
  \begin{enuminasm}
    \item
    \label{asm:g_convex}
    $g: \R^d \to \R \cup \set{+\infty}$ is proper closed convex.
    \item
    \label{asm:h_convex}
    $h: \R^n \to \R$ is convex.
    \item
    \label{asm:g_bounded}
    $\inf_{x \in \R^d} g(x) > - \infty$.
    \item
    \label{asm:h_bounded}
    $\inf_{y \in \R^n} h(y) > - \infty$.
    \item
    \label{asm:h_smooth}
    There exists a constant $L_h > 0$ such that $\norm*{\nabla h(y) - \nabla h(x)} \leq L_h \norm{y - x}$ for all $x,y \in \R^n$.
    \item
    \label{asm:nabla-c_lip}
    There exists a constant $L_c > 0$ such that $\normop*{\nabla c(y) - \nabla c(x)} \leq L_c \norm{y - x}$ for all $x,y \in \dom g$.
    \end{enuminasm}
\end{assumption}
\cref{asm:g_convex,asm:h_convex,asm:h_smooth,asm:nabla-c_lip} are standard ones that are also assumed for existing generalized LM method~\citep{drusvyatskiy2019efficiency} for smooth $h$.

\revise{
\Cref{asm:g_bounded,asm:h_bounded} are necessary to set the parameters $g^*, h^* \in \R$ of the proposed algorithm as
\begin{align}
  g^* \leq \inf_{x \in \R^d} g(x),\quad
  h^* \leq \inf_{y \in \R^n} h(y).
  \label{eq:requirement_gast_hast}
\end{align}
Throughout the paper, we use $g^*$ and $h^*$ to denote such parameters.
It may not be common practice to impose \cref{asm:g_bounded,asm:h_bounded} in general nonconvex optimization.
However, it is reasonable for the theoretical guarantee of LM methods, as discussed in \cref{sec:related_work_comparison}.
}
In practical applications, \revise{many functions $g$ and $h$ are bounded below, and their infima are known.}
For example, \cref{asm:g_bounded} holds for the $\ell_1$-norm, the nuclear norm, and the indicator function of a convex set, and \cref{asm:h_bounded} holds for square loss, logistic loss, and Kullback--Leibler divergence.
For the examples of $g$ and $h$ mentioned above, \revise{we know $\inf_{x \in \R^d} g(x) = \inf_{y \in \R^n} h(y) = 0$ and can set $g^* = h^* = 0$.}


\cref{asm:h_smooth,asm:nabla-c_lip} impose the Lipschitz continuity of $\nabla h$ and $\nabla c$ on $\R^n$ and $\dom g$.
These assumptions can be restricted to the sublevel sets as in the existing analysis~\citep{marumo2023majorization}; we can also show the complexity bounds and local convergence shown in \cref{sec:global_convergence,sec:local_convergence} by assuming
\begin{align}
  \norm{\nabla h(y) - \nabla h(x)}
  \leq 
  L_h \norm{y - x},\quad
  \forall x,y \in \Set*{z \in \R^n}{h(z) \leq F(x_0)}
\end{align}
and 
\begin{align}
  \normop*{\nabla c(y) - \nabla c(x)}
  \leq 
  L_c \norm{y - x},\quad
  \forall x,y \in \Set*{z \in \dom g}{F(z) \leq F(x_0)},
\end{align}
where $x_0$ is the starting point of the algorithm.
This allows us to handle a function $h$ whose gradient is not Lipschitz continuous globally, such as exponential functions.

\subsection{Algorithm}
\begin{algorithm}[t]
  \caption{Generalized LM method for problem~\cref{eq:problem_main}}
  \label{alg:proposed_backtracking}
  \KwIn{$x_0 \in \dom g$, $0 < \theta < 1$, $\rho_{\min} > 0$, $\alpha > 1$\revise{, and $g^*, h^* \in \R$ satisfying \cref{eq:requirement_gast_hast}}}
  $\rho \gets \rho_{\min}$\;
  \For{$k=0,1,\dots$}{
    $\mu \gets \rho \sqrt{F(x_k) - (g^* + h^*)}$\;
    \label{alg-line-bt:set_lambda}
    Find $x \in S_{k, \mu}$ by e.g., \cref{alg:apg}
    \tcp*{$S_{k, \mu}$ is defined by \cref{eq:def_Sk}}
    \If{$\displaystyle F(x) \leq F(x_k) - \frac{1 - \theta}{2} \mu \norm*{x - x_k}^2$ does not hold}{
      \label{alg-line-bt:check_cond}
      $\rho \gets \alpha \rho$\;
      \label{alg-line-bt:unsuccessful}
      \Goto{\cref{alg-line-bt:set_lambda}}
    }
    $(x_{k+1}, \mu_k) \gets (x, \mu)$\;
    \label{alg-line-bt:update_xk_muk}
  }
\end{algorithm}

Our LM method is described in \cref{alg:proposed_backtracking}.
This algorithm sets the damping parameter to
\begin{align}
  \mu = \rho \sqrt{F(x_k) - (g^* + h^*)} = \rho \sqrt{g(x_k) + h(c(x_k)) - (g^* + h^*)},
  \label{eq:set_mu}
\end{align}
where $\rho > 0$ is a parameter explained later, and computes an approximate solution $x \in S_{k, \mu}$ to subproblem~\cref{eq:subproblem}, where
\begin{align}
  S_{k, \mu}
  \coloneqq
  \Set[\Big]{x \in \dom g}{ \bar \omega_{k, \mu}(x) \leq \theta \mu \norm{x - x_k} }
  \label{eq:def_Sk}
\end{align}
and $0 < \theta < 1$ is a constant.
The method to find a solution in $S_{k, \mu}$ is described in the next section.

When constructing an LM method to achieve small oracle complexity and local quadratic convergence, we face two types of conflicts.
First, the damping parameter $\mu$ should converge to $0$ for local quadratic convergence, while a small $\mu$ increases the cost for the subproblem, hence the oracle complexity.
Second, subproblem~\cref{eq:subproblem} needs to be solved accurately for local quadratic convergence, while solving the subproblem too accurately increases the oracle complexity.
The damping parameter in \cref{eq:set_mu} and the approximate solution set in \cref{eq:def_Sk} resolve these conflicts and enable us to achieve both small oracle complexity and local quadratic convergence.
We should note that Eq.~\cref{eq:set_mu} is inspired by the existing choice of $\mu = \rho \norm{c(x_k)}$ for the case where $g$ is an indicator function and $h(\cdot) = \frac{1}{2} \norm{\cdot}^2$ \citep{marumo2023majorization}.

\cref{alg:proposed_backtracking} employs backtracking to choose the parameter $\rho$ and ensures that the objective function value decreases sufficiently; for $(x_k)_{k \in \N}$ and $(\mu_k)_{k \in \N}$ obtained by \cref{alg:proposed_backtracking}, we have
\begin{align}
  F(x_{k+1})
  &\leq
  F(x_k)
  - \frac{1 - \theta}{2} \mu_k \norm*{x_{k+1} - x_k}^2.
  \label{eq:descent_k}
\end{align}
The following lemma, called a descent lemma, justifies the backtracking by showing that the inequality in \cref{alg-line-bt:check_cond} of \cref{alg:proposed_backtracking} holds if $\rho$ is sufficiently large.
The lemma also demonstrates one of the benefits of defining $\mu$ and $S_{k, \mu}$ by \cref{eq:set_mu,eq:def_Sk}.
To prove the lemma, we use some inequalities given in \cref{sec:lemmas}.
\begin{lemma}
  \label{lem:descent}
  Suppose that \cref{asm:hc} holds, and set $\mu$ as in \cref{eq:set_mu} with
  \begin{align}
    \rho
    \geq
    \frac{L_c \sqrt{2 L_h}}{1 - \theta}.
    \label{eq:rho_lowerbound}
  \end{align}
  Then, the following holds:
  \begin{align}
    F(x)
    \leq
    F(x_k)
    - \frac{1 - \theta}{2} \mu \norm*{x - x_k}^2,
    \quad
    \forall x \in S_{k, \mu}.
    \label{eq:descent}
  \end{align}
\end{lemma}
\begin{proof}
  To simplify notation, let
  \begin{align}
    u \coloneqq x - x_k,\quad
    w \coloneqq c(x_k) + \nabla c(x_k) u.
  \end{align}
  By using $\bar F_{k, \mu}(x_k) = F(x_k)$, we decompose $F(x) - F(x_k)$ as
  \begin{alignat}{2}
    F(x) - F(x_k)
    =
    \prn*{\bar F_{k, \mu}(x) - \bar F_{k, \mu}(x_k)}
    + \prn*{F(x) - \bar F_{k, \mu}(x)}.
    \label{eq:FxFxk_diff_decomposition}
  \end{alignat}
  Each term can be bounded as follows:
  \begin{alignat}{2}
    \bar F_{k, \mu}(x) - \bar F_{k, \mu}(x_k)
    &\leq
    \bar{\omega}_{k, \mu}(x) \norm*{x - x_k} - \frac{\mu}{2} \norm{x - x_k}^2
    &\quad&\by{\cref{eq:proj-grad_norm_lowerbound_sub}}\\
    &\leq
    \theta \mu \norm*{x - x_k}^2 - \frac{\mu}{2} \norm{x - x_k}^2=
    - \frac{1 - 2 \theta}{2} \mu \norm{u}^2
    &\quad&\by{$x \in S_{k, \mu}$}
    \label{eq:hwhcxk_upperbound}
  \end{alignat}
  and
  \begin{alignat}{2}
    F(x) - \bar F_{k, \mu}(x)
    &=
    h(c(x)) - h(w)
    - \frac{\mu}{2} \norm*{u}^2\\
    &\leq
    \inner{\nabla h(w)}{c(x) - w} + \frac{L_h}{2} \norm*{c(x) - w}^2
    - \frac{\mu}{2} \norm*{u}^2
    &\quad&\by{\cref{eq:lem_three-points}}\\
    &\leq
    \norm*{\nabla h(w)} \norm*{c(x) - w}
    + \frac{L_h}{2} \norm*{c(x) - w}^2
    - \frac{\mu}{2} \norm*{u}^2.
    \label{eq:hcxhw_upperbound}
  \end{alignat}
  Furthermore, we have
  \begin{align}
    \norm*{c(x) - w}
    &=
    \norm*{ c(x) - c(x_k) - \nabla c(x_k) (x - x_k) }
    \leq
    \frac{L_c}{2} \norm*{x - x_k}^2
    = 
    \frac{L_c}{2} \norm*{u}^2,
    \label{eq:cxw_norm_upperbound}
  \end{align}
  where we use \cref{eq:property_nabla-c_lip}, and
  \begin{alignat}{2}
    &\mathInd
    \norm*{\nabla h(w)}
    - \frac{(1 - \theta) \mu}{2 L_c}\\
    &\leq
    \frac{L_c}{2 (1 - \theta) \mu} \norm*{\nabla h(w)}^2
    &\quad&\text{(since $2 \sqrt{ab} - a \leq b$ for $a,b \geq 0$)}\\
    &\leq
    \frac{L_c L_h}{(1 - \theta) \mu} \prn*{h(w) - h^*}
    &\quad&\by{\cref{eq:lem_smoothness_obj_grad-norm_bound}}\\
    &=
    \frac{L_c L_h}{(1 - \theta) \mu} \prn*{\bar F_{k, \mu}(x) - g(x) - h^* - \frac{\mu}{2}\norm*{u}^2}
    &\quad&\by{\cref{eq:def_Fbar}}\\
    &\leq
    \frac{L_c L_h}{(1 - \theta) \mu} \prn*{F(x_k) - g(x) -  h^* - (1 - \theta) \mu \norm{u}^2}
    &\quad&\by{\cref{eq:hwhcxk_upperbound} and $\bar F_{k, \mu}(x_k) = F(x_k)$}\\
    &=
    \frac{L_c L_h}{(1 - \theta) \mu} \prn*{F(x_k) - g(x) - h^*}
    - L_c L_h \norm{u}^2\\
    &\leq
    \frac{L_c L_h}{(1 - \theta) \mu} \prn*{F(x_k) - \prn*{g^* + h^*}}
    - L_c L_h \norm{u}^2
    &\quad&\by{\revise{\cref{eq:requirement_gast_hast}}}\\
    &\leq
    \frac{(1 - \theta) \mu}{2 L_c}
    - L_c L_h \norm{u}^2
    &\quad&\by{\cref{eq:set_mu,eq:rho_lowerbound}}.
    \label{eq:norm_gradhw_upperbound}
  \end{alignat}
  Plugging \cref{eq:cxw_norm_upperbound,eq:norm_gradhw_upperbound} into \cref{eq:hcxhw_upperbound}, we obtain
  \begin{align}
    F(x) - \bar F_{k, \mu}(x)
    &\leq
    \prn*{
      \frac{(1 - \theta) \mu}{L_c}
      - L_c L_h \norm{u}^2
    }
    \prn*{\frac{L_c}{2} \norm*{u}^2}
    + \frac{L_h}{2} \prn*{ \frac{L_c}{2} \norm*{u}^2 }^2
    - \frac{\mu}{2} \norm*{u}^2\\
    &=
    - \frac{\theta \mu}{2} \norm*{u}^2
    - \frac{3 L_c^2 L_h}{8} \norm*{u}^4
    \leq
    - \frac{\theta \mu}{2} \norm*{u}^2.
    \label{eq:FxFbarx_diff_upperbound}
  \end{align}
  The result follows from \cref{eq:FxFxk_diff_decomposition,eq:hwhcxk_upperbound,eq:FxFbarx_diff_upperbound}.
\end{proof}
\cref{lem:descent} immediately yields the following proposition.
\begin{proposition}
  \label{prop:rhomax}
  Suppose that \cref{asm:hc} holds.
  Let 
  \begin{align}
    \rho_{\max}
    \coloneqq
    \begin{dcases*}
      \alpha \frac{L_c \sqrt{2 L_h}}{1 - \theta} & if $\displaystyle \rho_{\min} < \frac{L_c \sqrt{2 L_h}}{1 - \theta}$,\\
      \rho_{\min} & otherwise.
    \end{dcases*}
    \label{eq:def_rhomax}
  \end{align}
  In \cref{alg:proposed_backtracking}, the parameter $\rho$ always satisfies 
  \begin{align}
    \rho_{\min} \leq \rho \leq \rho_{\max}.
    \label{eq:rho_minmax_bound}
  \end{align}
  Furthermore, \cref{alg-line-bt:unsuccessful} is executed at most
  \begin{align}
    \max \set*{
      \ceil*{
        \log_\alpha \prn*{
          \frac{L_c \sqrt{2 L_h}}{(1 - \theta) \rho_{\min}}
        }
      }, \, 
      0
    }
  \end{align}
  times, where $\ceil{\cdot}$ denotes the ceiling function.
\end{proposition}

  Lemmas that evaluate the objective function decrease in one iteration, such as \cref{lem:descent}, are called descent lemmas or sufficient decrease lemmas and are often used in the complexity analysis of optimization methods.
  Descent lemmas similar to \cref{lem:descent} have been shown to analyze generalized LM methods (e.g., \citep[Eq.~(25)]{drusvyatskiy2018error} and \citep[Eq.~(5.2)]{drusvyatskiy2019efficiency}) under the assumption of Lipschitz continuous $h$ and $\nabla c$.
  They require $\mu_k$ to be set above a certain constant, prohibiting $(\mu_k)_{k \in \N}$ from converging to $0$.
  Meanwhile, existing LM methods that enjoy local quadratic convergence under the quadratic growth condition make $(\mu_k)_{k \in \N}$ converge to $0$ \citep{yamashita2001rate,fan2005quadratic,fan2006convergence,dan2002convergence,fischer2010inexactness,bellavia2014strong,bellavia2010convergence,kanzow2004levenberg,behling2012unified,facchinei2013family}.
  This conflict implies that it is difficult or impossible to guarantee local quadratic convergence under the quadratic growth, as long as we use the existing descent lemmas.

  Our descent lemma overcomes the difficulty.
  \cref{lem:descent} allows $(\mu_k)_{k \in \N}$ to converge to $0$ when $(\Delta_k)_{k \in \N}$ defined in \cref{eq:def_Deltak_Fast} converges to $0$, which enables us to achieve both small oracle complexity and local quadratic convergence.
  This is accomplished by exploiting the Lipschitz continuity and convexity of $\nabla h$, while the existing descent lemmas mentioned above use the Lipschitz continuity of $h$.

\subsection{Accelerated method for subproblem}
In each iteration of \cref{alg:proposed_backtracking}, we have to find an approximate solution $x \in S_{k, \mu}$ to subproblem~\cref{eq:subproblem}.
Such a point can be obtained by the accelerated proximal gradient (APG) method, \cref{alg:apg}, which is \citep[Algorithm~31]{daspremont2021acceleration} tailored to our setting.

\begin{algorithm}[t]
  \caption{Accelerated proximal gradient for subproblem~\cref{eq:subproblem}}
  \label{alg:apg}
  \KwIn{%
    $k \in \N$, 
    $\mu > 0$, 
    $0 < \theta < 1$,
    $0 < \bar \beta < 1 < \bar \alpha$
  }
  \KwOut{A point $x \in S_{k, \mu}$}
  $\bar x_0 \gets x_k$, \ 
  $z_0 \gets x_k$\;
  $\eta \gets \bar \alpha \mu$, \ 
  $b_0 \gets 0$\;
  \label{alg-line-apg:initialize_etab0}
  \For{$t = 0,1,\dots$}{
    $b_{t+1} \gets \frac{1 + 2 \eta b_t + \sqrt{1 + 4 \eta b_t (1 + \mu b_t) }}{2 (\eta - \mu)}$, \ 
    $\tau \gets \frac{(b_{t+1} - b_t) (1 + \mu b_t)}{b_{t+1} (1 + \mu b_t) + \mu b_t (b_{t+1} - b_t)}$\;
    \label{alg-line-apg:set_Btau}
    $y_t \gets \bar x_t + \tau (z_t - \bar x_t)$\;
    \label{alg-line-apg:set_y}
    $\bar x_{t+1} \gets \prox_{\revise{\eta^{-1}} g} \prn*{y_t - \frac{1}{\eta} \nabla \bar H_{k, \mu}(y_t)}$\;
    \label{alg-line-apg:set_xbar}
    \If{$\bar H_{k, \mu}(\bar x_{t+1}) \leq \bar H_{k, \mu}(y_t) + \inner{\nabla \bar H_{k, \mu}(y_t)}{\bar x_{t+1} - y_t} + \frac{\eta}{2} \norm*{\bar x_{t+1} - y_t}^2$ does not hold}{
      \label{alg-line-apg:backtracking}
      $\eta \gets \bar \alpha \eta$\;
      \label{alg-line-apg:increase_eta}
      \Goto{\cref{alg-line-apg:set_Btau}}
    }
    \If{$\norm*{\nabla \bar H_{k, \mu}(\bar x_{t+1}) - \nabla \bar H_{k, \mu}(y_t) - \eta (\bar x_{t+1} - y_t)} \leq \theta \mu \norm*{\bar x_{t+1} - \bar x_0}$}{
      \label{alg-line-apg:termination_cond}
      \Return{$\bar x_{t+1}$}
    }
    $\phi \gets \frac{b_{t+1} - b_t}{1 + \mu b_{t+1}}$\;
    $z_{t+1} \gets (1 - \mu \phi) z_t + \mu \phi y_t + \eta \phi (\bar x_{t+1} - y_t)$\;
    $\eta \gets \bar \beta \eta$\;
  }
\end{algorithm}

Let us verify that \cref{alg:apg} outputs a solution $x \in S_{k, \mu}$ when it terminates.
From \cref{alg-line-apg:set_xbar}, we have
\begin{align}
  - \prn*{ \nabla \bar H_{k, \mu}(y_t) + \eta (\bar x_{t+1} - y_t) } \in \partial g(\bar x_{t+1}),
  \label{eq:in_partialgx}
\end{align}
which implies that
\begin{align}
  \bar{\omega}_{k, \mu}(\bar x_{t+1})
  &= 
  \min_{p \in \partial g(\bar x_{t+1})} \norm*{p + \nabla \bar H_{k, \mu}(\bar x_{t+1})}\\
  &\leq
  \norm*{\nabla \bar H_{k, \mu}(\bar x_{t+1}) - \nabla \bar H_{k, \mu}(y_t) - \eta (\bar x_{t+1} - y_t)}.
  \label{eq:baromega_upperbound}
\end{align}
This bound shows that when the termination condition in \cref{alg-line-apg:termination_cond} is satisfied, the inequality
\begin{align}
  \bar{\omega}_{k, \mu}(\bar x_{t+1})
  \leq
  \theta \mu \norm*{\bar x_{t+1} - \bar x_0}
  =
  \theta \mu \norm*{\bar x_{t+1} - x_k}
\end{align}
holds and hence $\bar x_{t+1} \in S_{k, \mu}$.

Next, we will show the iteration and oracle complexity bounds of \cref{alg:apg}.
We assume the boundedness of $\normop*{\nabla c(x_k)}$ to derive the complexity, which is also a standard assumption for LM methods \citep{ueda2010global,zhao2016global,bergou2020convergence,marumo2023majorization,drusvyatskiy2019efficiency}.
\begin{assumption}
  \label{asm:nablac_bound}
  Let $(x_k)_{k \in \N}$ be generated by \cref{alg:proposed_backtracking}.
  For some $\sigma > 0$,
  \begin{align}
    \normop*{\nabla c(x_k)} \leq \sigma,\quad
    \forall k \in \N.
  \end{align}
\end{assumption}
The complexity is expressed using the Lipschitz constant of $\nabla \bar H_{k, \mu}$ shown in the following lemma.
\begin{lemma}
  \label{lem:smooth_subproblem}
  Suppose that \cref{asm:hc,asm:nablac_bound} hold.
  Then, $\nabla \bar H_{k, \mu}$ is $\bar L$-Lipschitz continuous, where
  \begin{align}
    \bar L
    \coloneqq
    \mu + L_h \sigma^2.
    \label{eq:def_Lbar}
  \end{align}
\end{lemma}
\begin{proof}
  See \cref{sec:proof_smooth_subproblem}.
\end{proof}
\begin{lemma}[Iteration complexity of \cref{alg:apg}]
  \label{lem:apg_iteration}
  Suppose that \cref{asm:hc,asm:nablac_bound} hold.
  Then, \cref{alg:apg} terminates within
  \begin{align}
    2 \sqrt{\bar \alpha \bar \kappa} \log \prn*{
      2 \sqrt{\bar \alpha \bar \kappa}
      \prn*{
        1 + \frac{2 (1 + \bar \alpha) \bar \kappa}{\theta}
      }
    }
    =
    O\prn*{
      1 + \sqrt{\bar \kappa} \log \bar \kappa
    }
    \label{eq:inner_iteration_complexity}
  \end{align}
  iterations, where
  \begin{align}
    \bar \kappa
    \coloneqq
    \frac{\bar L}{\mu}
    =
    1 + \frac{L_h \sigma^2}{\mu}
    \label{eq:def_kappa}
  \end{align}
  is the condition number of $\bar H_{k, \mu}$.\footnote{
    The $O$-notation in \cref{eq:inner_iteration_complexity} shows the behavior in the limit with respect to $\bar \kappa$.
    Since $\sqrt{\bar \kappa} \log \bar \kappa \to 0$ if $\bar \kappa \to 1$, the term $O(1)$ in \cref{eq:inner_iteration_complexity} cannot be omitted.
  }
\end{lemma}
\begin{proof}
  See \cref{sec:proof_apg_subproblem}.
\end{proof}

\cref{alg:apg} fails to update the solution, i.e., reaches \cref{alg-line-apg:increase_eta}, at most
\begin{align}
  \ceil*{
    \log_{\bar \alpha} \prn*{
      \frac{\bar L}{\bar \alpha \mu}
    }
  }
  = O \prn*{ \log \bar \kappa }
\end{align}
times.
Putting together this and the iteration complexity in \cref{lem:apg_iteration}, we obtain the following oracle complexity bound of \cref{alg:apg}.
The oracle complexity bound for the subproblem will be used to derive the overall oracle complexity of \cref{alg:proposed_backtracking} combined with \cref{alg:apg} in the next section.
\begin{proposition}[Oracle complexity of \cref{alg:apg}]
  \label{prop:complexity_per_iteration}
  Suppose that \cref{asm:hc,asm:nablac_bound} hold.
  \cref{alg:apg} outputs a solution $x \in S_{k, \mu}$ after calling the oracles
  \begin{align}
    O \prn*{ 1 + \sqrt{\bar \kappa} \log \bar \kappa }
  \end{align}
  times, where $\bar \kappa$ is defined in \cref{eq:def_kappa} and depends on $\mu$.
\end{proposition}

Checking the termination condition in \cref{alg-line-apg:termination_cond} of \cref{alg:apg} requires an additional gradient computation of $\nabla \bar H_{k, \mu}(\bar x_{t+1})$, while
the gradient $\nabla \bar H_{k, \mu}(y_t)$ is already computed in \cref{alg-line-apg:set_xbar}.
To reduce this cost, we can check the termination condition every $T$ iterations.
Although this modification does not affect the order of oracle complexity of \cref{alg:apg}, it is helpful in practice.

While this section describes APG to compute a point in $S_{k, \mu}$, we could use another suitable algorithm that takes advantage of problem structures.
For example, for least-squares problem without regularization (i.e., when $g(\cdot) = 0$ and $h(\cdot) = \frac{1}{2}\norm*{\cdot}^2$), the conjugate gradient (CG) method is probably the best choice for solving the subproblem.
The CG method achieves linear convergence for strongly convex quadratic functions, and its convergence rate is better than that of APG (see e.g., \citep[Eq.~(5.36)]{nocedal2006numerical}).

\section{Iteration and oracle complexity}
\label{sec:global_convergence}
We will derive the iteration complexity of \cref{alg:proposed_backtracking} and the oracle complexity of \cref{alg:proposed_backtracking} combined with \cref{alg:apg}.

\subsection{Notation}
Let $(x_k)_{k \in \N}$ and $(\mu_k)_{k \in \N}$ be generated by \cref{alg:proposed_backtracking}, and define $\Delta_k$ by \cref{eq:def_Deltak_Fast}.
Then, it follows from \cref{eq:descent_k} that
\begin{align}
  \Delta_{k+1} \leq \Delta_k,
  \label{eq:Delta_decreasing}
\end{align}
and it follows from \cref{alg-line-bt:set_lambda} of \cref{alg:proposed_backtracking} and \cref{prop:rhomax} that
\begin{align}
  \rho_{\min} \sqrt{\Delta_k}
  \leq 
  \mu_k
  \leq
  \rho_{\max} \sqrt{\Delta_k}.
  \label{eq:muk_lowerupperbounds}
\end{align}
We assume that $\Delta_k > 0$ for all $k \in \N$ without loss of generality because $\Delta_k = 0$ for some $k$ means that the algorithm finds the optimal solution in a finite number of steps, and the complexity is $O(1)$.

\subsection{Convergence and complexity results}
First, we show the following lemma.
\begin{lemma}
  \label{lem:gradMF_norm_upperbound}
  Suppose that \cref{asm:hc,asm:nablac_bound} hold.
  For the sequences $(x_k)_{k \in \N}$ and $(\mu_k)_{k \in \N}$ generated by \cref{alg:proposed_backtracking}, the following holds for all $k \in \N$ and \revise{$x \in \R^d$ such that $\bar F_{k, \mu_k}(x) \leq F(x_k)$}:
  \begin{align}
    \omega(x)
    &\leq
    \bar{\omega}_{k, \mu_k}(x)
    + \prn*{ \mu_k + L_c \sqrt{2 L_h \Delta_k} } \norm*{x - x_k}
    + \frac{L_c L_h \sigma}{2} \norm*{x - x_k}^2
    \revise{{}+ \frac{L_c^2 L_h}{2} \norm*{x - x_k}^3,}
    \label{eq:gradMF_norm_upperbound_tighter}
  \end{align}
  where $\omega(x)$ and $\bar \omega_{k, \mu_k}(x)$ are defined in \cref{eq:def_error}.
  \revise{
  Furthermore, the following holds for all $k \in \N$:
  \begin{align}
    \omega(x_{k+1})
    &\leq
    \prn*{ (1 + \theta) \mu_k + L_c \sqrt{2 L_h \Delta_k} } \norm*{x_{k+1} - x_k}
    + \frac{L_c L_h \sigma}{2} \norm*{x_{k+1} - x_k}^2.
    \label{eq:gradMF_norm_upperbound}
  \end{align}%
  }%
\end{lemma}
\begin{proof}
  See \cref{sec:proof_gradMF_norm}.
\end{proof}

\Cref{eq:gradMF_norm_upperbound} shows that if $x_{k+1}$ is sufficiently close to $x_k$, then $x_{k+1}$ is an $\epsilon$-stationary point.
Combining this bound with \cref{eq:descent_k} yields the following lemma.

\begin{lemma}
  \label{lem:optimality_upperbound}
  Suppose that \cref{asm:hc,asm:nablac_bound} hold.
  Let $(x_k)_{k \in \N}$ and $(\mu_k)_{k \in \N}$ be generated by \cref{alg:proposed_backtracking}.
  Then, for each $k \in \N$, at least one of the following holds:  
  \begin{align}
    \omega(x_{k+1})^2
    \leq
    \frac{8}{1 - \theta}
    \prn*{
      (1 + \theta) \sqrt{\rho_{\max}}
      + \frac{L_c \sqrt{2 L_h}}{\sqrt{\rho_{\min}}}
    }^2
    \sqrt{\Delta_k}
    \prn*{
      \Delta_k - \Delta_{k+1}
    },
    \label{eq:optimality_upperbound1}
  \end{align}
  or
  \begin{align}
    \omega(x_{k+1})
    \leq
    \frac{4 L_c L_h \sigma}{(1 - \theta) \rho_{\min}}
    \prn*{\sqrt{\Delta_k} - \sqrt{\Delta_{k+1}}}.
    \label{eq:optimality_upperbound2}
  \end{align}
\end{lemma}

\begin{proof}
  From definition \cref{eq:def_Deltak_Fast} of $\Delta_k$, we have
  \begin{align}
    F(x_k) - F(x_{k+1})
    = \Delta_k - \Delta_{k+1}.
    \label{eq:diff_F_equal_diff_Delta}
  \end{align}
  If the first term of the right-hand side of \cref{eq:gradMF_norm_upperbound} is larger than the second term, we have
  \begin{alignat}{2}
    &\mathInd
    \omega(x_{k+1})^2\\
    &\leq
    4 \prn*{ (1 + \theta) \mu_k + L_c \sqrt{2 L_h \Delta_k} } \norm*{x_{k+1} - x_k}^2\\
    &=
    4 \prn*{
      (1 + \theta) \sqrt{\frac{\mu_k}{\sqrt{\Delta_k}}}
      + L_c \sqrt{2 L_h} \sqrt{\frac{\sqrt{\Delta_k}}{\mu_k}}
    }^2
    \sqrt{\Delta_k}
    \mu_k
    \norm*{x_{k+1} - x_k}^2\\
    &\leq
    4 \prn*{
      (1 + \theta) \sqrt{\rho_{\max}}
      + \frac{L_c \sqrt{2 L_h}}{\sqrt{\rho_{\min}}}
    }^2
    \sqrt{\Delta_k}
    \mu_k
    \norm*{x_{k+1} - x_k}^2
    &\quad&\by{\cref{eq:muk_lowerupperbounds}}\\
    &\leq
    4 \prn*{
      (1 + \theta) \sqrt{\rho_{\max}}
      + \frac{L_c \sqrt{2 L_h}}{\sqrt{\rho_{\min}}}
    }^2
    \sqrt{\Delta_k}
    \frac{2}{1 - \theta}
    \prn*{
      \Delta_k - \Delta_{k+1}
    }
    &\quad&\by{\cref{eq:descent_k,eq:diff_F_equal_diff_Delta}},
  \end{alignat}
  which yields \cref{eq:optimality_upperbound1}.
  Otherwise, we obtain \cref{eq:optimality_upperbound2} as follows:
  \begin{alignat}{2}
    \omega(x_{k+1})
    &\leq
    L_c L_h \sigma \norm*{x_{k+1} - x_k}^2\\
    &\leq
    \frac{2 L_c L_h \sigma}{(1 - \theta) \mu_k} \prn*{\Delta_k - \Delta_{k+1}}
    &\quad&\by{\cref{eq:descent_k,eq:diff_F_equal_diff_Delta}}\\
    &=
    \frac{2 L_c L_h \sigma}{(1 - \theta) \mu_k}
    \prn*{\sqrt{\Delta_k} + \sqrt{\Delta_{k+1}}}
    \prn*{\sqrt{\Delta_k} - \sqrt{\Delta_{k+1}}}\\
    &\leq
    \frac{2 L_c L_h \sigma}{(1 - \theta) \mu_k}
    2 \sqrt{\Delta_k}
    \prn*{\sqrt{\Delta_k} - \sqrt{\Delta_{k+1}}}
    &\quad&\by{\cref{eq:Delta_decreasing}}\\
    &\leq
    \frac{4 L_c L_h \sigma}{(1 - \theta) \rho_{\min}}
    \prn*{\sqrt{\Delta_k} - \sqrt{\Delta_{k+1}}}
    &\quad&\by{\cref{eq:muk_lowerupperbounds}},
  \end{alignat}
  which completes the proof.
\end{proof}

Since the sequence $(\Delta_k)_{k \in \N}$ is nonincreasing and lower-bounded, it converges to some value, and therefore the right-hand side of \cref{eq:optimality_upperbound1,eq:optimality_upperbound2} converges to $0$.
Thus, we obtain the global convergence property of \cref{alg:proposed_backtracking}.
\begin{proposition}[Global convergence]
  \label{prop:global_convergence}
  Suppose that \cref{asm:hc,asm:nablac_bound} hold.
  For the sequence $(x_k)_{k \in \N}$ generated by \cref{alg:proposed_backtracking}, the following holds:
  \begin{align}
    \lim_{k \to \infty} \omega(x_k) = 0.
    \label{eq:global_convergence}
  \end{align}
\end{proposition}

The following theorem shows the iteration and oracle complexity of our LM method.
\begin{theorem}[Iteration and oracle complexity]
  \label{thm:complexity_global}
  Suppose that \cref{asm:hc,asm:nablac_bound} hold.
  \begin{enuminthm}
    \item
    \label{thm:iteration_complexity_global}
    \cref{alg:proposed_backtracking} finds a point $x \in \dom g$ satisfying $\omega(x) \leq \epsilon$ within
    \revise{
    \begin{align}
      O \prn*{
        \frac{\sqrt{\Delta_0}}{\epsilon^2}
        (F(x_0) - F^*)
        \prn*{
          \rho_{\max}
          + \frac{L_c^2 L_h}{\rho_{\min}}
        }
      }
      \label{eq:iteration_complexity_general}
    \end{align}%
    }%
    iterations, \revise{where $\rho_{\max}$ is defined by \cref{eq:def_rhomax}.}
    Such a point is included in the sequence $(x_k)_{k \in \N}$ generated by \cref{alg:proposed_backtracking}.
    \revise{
    Furthermore, if we set the input parameter $\rho_{\min}$ of \cref{alg:proposed_backtracking} as $\rho_{\min} = \Theta(L_c \sqrt{L_h})$, the iteration complexity bound simplifies to 
    }
    \begin{align}
      O \prn*{
        \frac{L_c \sqrt{L_h \Delta_0} }{\epsilon^2}
        \prn*{ F(x_0) - F^* }
      }.
      \label{eq:iteration_complexity_simple}
    \end{align}
    \item
    \label{thm:oracle_complexity_global}
    \revise{For sufficiently small $\epsilon > 0$,} \cref{alg:proposed_backtracking} with \cref{alg:apg} finds a point $x \in \dom g$ satisfying $\omega(x) \leq \epsilon$ by calling the oracles \revise{
    \begin{align}
      O \Bigg(
        &
        \frac{\sqrt{\Delta_0}}{\epsilon^2}
        \prn*{ F(x_0) - F^* }
        \prn*{
          \rho_{\max}
          + \frac{L_c^2 L_h}{\rho_{\min}}
        }
        \prn*{
          1 + \log \prn*{ \frac{L_c \sqrt{L_h}}{\rho_{\min}} }
        }\\
        &\quad \times 
        \prn*{
          1 + \sqrt{1 + \frac{L_h \sigma^2}{\rho_{\min} \sqrt{\Delta_0}}}
          \log \prn*{
            1 + \frac{L_h \sigma^2}{\rho_{\min} \sqrt{\Delta_0}}
          }
        }
      \Bigg)
      \label{eq:oracle_complexity_general}
    \end{align}
    times.}
    Such a point is included in sequences $(\bar x_t)_{t \in \N}$ generated by \cref{alg:apg}.
    \revise{
    Furthermore, if we set $\rho_{\min}$ as $\rho_{\min} = \Theta(L_c \sqrt{L_h})$, the oracle complexity bound simplifies to
    }
    \begin{align}
      O \prn*{
        \frac{ L_c \sqrt{L_h \Delta_0}}{\epsilon^2}
        \prn*{ F(x_0) - F^* }
        \prn*{
          1 + 
          \sqrt{\kappa}\log \kappa
        }
      },
      \ \ \text{where}\ \ 
      \kappa
      \coloneqq
      1 + 
      \frac{\sqrt{L_h} \sigma^2}{L_c \sqrt{\Delta_0}}.
      \label{eq:oracle_complexity_simple}
    \end{align}
  \end{enuminthm}
\end{theorem}
\begin{proof}
  See \cref{sec:proof_iteration_oracle}.
\end{proof}


\revise{
  More precisely speaking, in order to guarantee $\omega(\bar x_1) \leq \epsilon$ in \cref{thm:oracle_complexity_global}, $\epsilon$ needs to be so small that $\epsilon \lesssim \frac{\rho_{\min}}{(L_h \sigma^2)^2}$,
  where the notation $\lesssim$ hides constant factors.
  We should note that the right-hand side $\frac{\rho_{\min}}{(L_h \sigma^2)^2}$ could be small depending on the values of $L_h$, $\sigma$, and $\rho_{\min}$.
  See the proof, \cref{eq:lesssim} in particular, given in \cref{sec:proof_iteration_oracle} for details.
}

\begin{remark}
  Setting $\mu$ to $\rho \sqrt{ F(x_k) - (g^* + h^*)} = \rho \sqrt{\Delta_k}$ in \cref{alg-line-bt:set_lambda} of \cref{alg:proposed_backtracking} is necessary to derive local convergence in the next section.
  If we give up the local convergence guarantee, the oracle complexity can be obtained more easily by setting $\mu$ to $\rho \sqrt{\Delta_0}$ instead of $\rho \sqrt{\Delta_k}$.
  Then, the oracle complexity per iteration is bounded as $O(1 + \sqrt{\kappa} \log \kappa)$ by \cref{prop:complexity_per_iteration}, and multiplying this bound by the iteration complexity in \cref{thm:iteration_complexity_global} immediately yields the same oracle complexity as \cref{thm:oracle_complexity_global}.
\end{remark}

\section{Local convergence under H\"olderian growth condition}
\label{sec:local_convergence}
We will see that \cref{alg:proposed_backtracking} achieves faster convergence under the H\"olderian growth condition.

\subsection{Notation and assumptions}
\label{sec:local_notation_assumption}
Let $(x_k)_{k \in \N}$ be generated by \cref{alg:proposed_backtracking}.
Since $(F(x_k))_{k \in \N}$ is nonincreasing and lower-bounded, the sequence converges to some value.
Let
\begin{align}
  F_\infty
  &\coloneqq
  \lim_{k \to \infty} F(x_k),\\
  \Delta_\infty
  &\coloneqq
  \lim_{k \to \infty} \Delta_k
  =
  F_\infty - (g^* + h^*)
  \geq 0,
  \label{eq:def_Deltainf}\\
  X^*
  &\coloneqq
  \Set*{x \in \dom g}{ \omega(x) = 0,\ F(x) \leq F_\infty },
  \label{eq:def_Xast}\\
  \dist(x, X^*)
  &\coloneqq
  \min_{y \in X^*} \norm*{y - x}.
  \label{eq:def_distxXast}
\end{align}
The set $X^*$ is a closed subset of the set of stationary points of problem \cref{eq:problem_main}.\footnote{
  The set $X^*$ is closed because it is the intersection of two closed sets.
  The set $\Set*{x \in \dom g}{F(x) \leq F_\infty}$ is closed since $F$ is continuous, and we can also check that the set $\Set*{x \in \dom g}{\omega(x) = 0}$ is closed since the subdifferential $\partial g(x)$ is closed and $\nabla H(x)$ is continuous.
}
Hence, \cref{eq:def_distxXast} is properly defined in that a minimizer exists.

To derive local convergence, we assume the following in addition to \cref{asm:hc,asm:nablac_bound}.
\begin{assumption}
  \label{asm:local}
  \leavevmode
  \begin{enuminasm}
    \item
    \label{asm:nonempty_Xast}
    $X^*$ is nonempty.
    \item
    \label{asm:holderian_growth}
    Let $(x_k)_{k \in \N}$ be generated by \cref{alg:proposed_backtracking}.
    For some $1 \leq r < 3$, $\gamma > 0$, and $K \in \N$,
    \begin{align}
      \frac{\gamma}{r} \dist(x_k, X^*)^r \leq F(x_k) - F_\infty,\quad
      \forall k \geq K.
    \end{align}
  \end{enuminasm}
\end{assumption}
\cref{asm:nonempty_Xast} holds when problem~\cref{eq:problem_main} has an optimal solution; the set $X^*$ contains the optimal solution.

\cref{asm:holderian_growth} is essential, and this type of condition is called a \emph{H\"olderian growth} condition or a \emph{H\"olderian error-bound} condition.
\revise{
The condition with $r = 2$ is called \emph{quadratic growth} and widely used in optimization, regardless of whether the problem is convex or nonconvex and least-squares or not (e.g., \citep{drusvyatskiy2018error,karimi2016linear,necoara2019linear}).
Furthermore, the condition is not restrictive; \citep[Theorem~2]{karimi2016linear} describes the relationship between various conditions, such as strong convexity and the Polyak--\L ojasiewicz condition, of which the quadratic growth is the weakest.
Similar results can be found in \citep[Theorem~4 and Fig.~1]{necoara2019linear}.
}

The H\"olderian growth with $r \geq 2$ has been used to analyze LM methods with $h(\cdot) = \frac{1}{2} \norm*{\cdot}^2$ \citep{yamashita2001rate,ahookhosh2019local,bao2019modified,wang2021convergence}, and the H\"olderian growth with $r = 1$ has been used for general Lipschitz $h$ \citep{burke1995gauss,li2002convergence,drusvyatskiy2018error}.
We analyze our LM method for general smooth~$h$ under the H\"olderian growth with $1 \leq r < 3$ in a unified manner.

\revise{
\begin{remark}
  \label{rem:compare_bergou2020}
  Let us focus on the case where $g(\cdot) = 0$ and $h(\cdot) = \frac{1}{2} \norm*{\cdot}^2$.
  \citet{bergou2020convergence} showed linear convergence of an LM method under a different condition than the quadratic growth when $\Delta_\infty$ is small.
  Their assumption is of the form
  \begin{align}
    \frac{\gamma}{2} \dist(x_k, X^*)^2
    \leq
    \frac{1}{2} \norm*{c(x_k) - c(x^*_k)}^2
  \end{align}
  for some $x^*_k \in \argmin_{x \in X^*} \norm*{x - x_k}$, while the quadratic growth is
  \begin{align}
    \frac{\gamma}{2} \dist(x_k, X^*)^2
    \leq
    \frac{1}{2} \prn*{ \norm*{c(x_k)}^2 - \norm{c(x^*_k)}^2 }.
  \end{align}
  Because $\norm*{c(x_k) - c(x^*_k)}^2 \leq \norm*{c(x_k)}^2 - \norm{c(x^*_k)}^2$ in some cases and not in others, we cannot conclude which condition is stronger.
  However, using \cref{asm:holderian_growth}, which extends the quadratic growth, helps allow comparison with problem settings other than least squares.
\end{remark}
}

\subsection{Convergence results}
As confirmed by \cref{eq:gradMF_norm_upperbound}, $x_{k+1}$ is an $\epsilon$-stationary point when $\norm*{x_{k+1} - x_k}$ is sufficiently small.
The following lemma shows that small $\dist(x_k, X^*)$ implies small $\norm*{x_{k+1} - x_k}$.
\begin{lemma}
  \label{lem:update_leq_dist}
  Suppose that \cref{asm:hc,asm:nablac_bound,asm:local} hold. 
  Let $(x_k)_{k \in \N}$ be generated by \cref{alg:proposed_backtracking}, and let 
  \begin{align}
    C_1
    &\coloneqq
    \frac{L_c \sqrt{2 L_h}}{(1 - \theta) \rho_{\min}}
    + \frac{1}{\sqrt{1 - \theta^2}}.
    \label{eq:def_C1}
  \end{align}
  If $k \geq K$ and
  \begin{align}
    \dist(x_k, X^*)
    &\leq
    \prn*{
      \frac{8 \sqrt{2 \gamma}}{L_c \sqrt{r L_h}}
    }^{2 / (4 - r)},
    \label{eq:asm_vk_norm_upperbound}
  \end{align}
  then the following holds:
  \begin{align}
    \norm*{x_{k+1} - x_k}
    \leq
    C_1 \dist(x_k, X^*).
    \label{eq:uk_norm_bound_simple}
  \end{align}
\end{lemma}
\begin{proof}
  In \cref{sec:proof_update_leq_dist}.
\end{proof}

Using \cref{eq:gradMF_norm_upperbound,lem:update_leq_dist}, we obtain the following proposition.
\begin{proposition}
  \label{prop:local_convergence}
  Suppose that \cref{asm:hc,asm:nablac_bound,asm:local} hold.
  Let $(x_k)_{k \in \N}$ be generated by \cref{alg:proposed_backtracking}, and let 
  \begin{align}
    \delta_k \coloneqq F(x_k) - F_\infty
    \label{eq:def_deltak}
  \end{align}
  and 
  \begin{align}
    C_2
    &\coloneqq
    (1 + \theta) \rho_{\max} + L_c \sqrt{2 L_h},\quad
    C_3
    \coloneqq
    2 C_1 C_2 
    \prn*{
      \frac{r}{\gamma}
    }^{2/r},\quad
    C_4
    \coloneqq
    2 C_1^2
    L_c L_h \sigma
    \prn*{
      \frac{r}{\gamma}
    }^{3/r},
    \label{eq:def_C23}
  \end{align}
  where $\rho_{\max}$ and $C_1$ are defined in \cref{eq:def_rhomax,eq:def_C1}.
  If $k \geq K$, \cref{eq:asm_vk_norm_upperbound}, and
  \begin{align}
    L_c \sqrt{2 L_h (\delta_{k+1} + \Delta_\infty)}
    \prn*{\frac{r \delta_{k+1}}{\gamma}}^{2/r}
    \leq
    \delta_{k+1}
    \label{eq:asm_vk_norm_upperbound2}
  \end{align}
  hold, then $\delta_{k+1} = 0$ or
  \begin{align}
    \delta_{k+1}^{1-1/r}
    \leq
    \prn*{
      C_3
      \sqrt{\delta_k + \Delta_\infty}
      +
      C_4
      \delta_k^{1/r}
    }
    \delta_k^{1/r}.
    \label{eq:local_convergence_rate}
  \end{align}
\end{proposition}
\begin{proof}
  In \cref{sec:proof_prop_localconv}.
\end{proof}

The above proposition yields the convergence order of \cref{alg:proposed_backtracking}.
\begin{theorem}[Order of local convergence]
  \label{thm:local_obj}
  Suppose that \cref{asm:hc,asm:nablac_bound,asm:local} hold.
  Let $(x_k)_{k \in \N}$ be generated by \cref{alg:proposed_backtracking}.
  \begin{enumerate}
    \item
    \label{thm:local_r1}
    If $r = 1$, then $F(x_k) = F_\infty$ for some $k \in \N$.
    \item
    \label{thm:local_r12_delta0}
    If $1 < r < 2$ and $\Delta_\infty = 0$, then $(F(x_k))_{k \in \N}$ converges with order $\frac{r+2}{2(r-1)}$.
    \item
    \label{thm:local_r12}
    If $1 < r < 2$ and $\Delta_\infty > 0$, then $(F(x_k))_{k \in \N}$ converges with order $\frac{1}{r-1}$.
    \item
    \label{thm:local_r23_delta0}
    If $2 \leq r < 3$ and $\Delta_\infty = 0$, then $(F(x_k))_{k \in \N}$ converges with order $\frac{2}{r-1}$.
    \item
    \label{thm:local_r2_deltasmall}
    If $r = 2$ and $0 < \Delta_\infty < 1 / C_3^2$, then $(F(x_k))_{k \in \N}$ converges with order $1$ (i.e., linearly), where $C_3$ is defined in \cref{eq:def_C23}.
  \end{enumerate}
\end{theorem}
\begin{proof}
  The sequence $(\delta_k)_{k \in \N}$ converges to $0$ by definition \cref{eq:def_deltak}.
  In addition, the sequences $(\dist(x_k, X^*))_{k \in \N}$ also converges to $0$ from \cref{asm:holderian_growth}.
  Therefore, condition~\cref{eq:asm_vk_norm_upperbound} holds for all sufficiently large $k \in \N$.
  \paragraph{Case~\ref{thm:local_r1}: $r = 1$.}
  Condition~\cref{eq:asm_vk_norm_upperbound2} holds for sufficiently small $\delta_{k+1}$ and hence for sufficiently large $k$.
  Thus, \cref{prop:local_convergence} implies that $\delta_{k+1} = 0$ or \cref{eq:local_convergence_rate} holds for all large $k$.
  However, \cref{eq:local_convergence_rate} does not hold if $\delta_{k+1} > 0$, and hence we have $\delta_{k+1} = 0$.
  This completes the proof of Case~\ref{thm:local_r1}.

  Below, we assume that $\delta_k > 0$ for all $k \in \N$.
  If $\delta_k = 0$ for some $k$, Cases~\ref{thm:local_r12_delta0}--\ref{thm:local_r2_deltasmall} are obvious.

  \paragraph{Case~\ref{thm:local_r12_delta0}: $1 < r < 2$ and $\Delta_\infty = 0$.}
  Condition~\cref{eq:asm_vk_norm_upperbound2} holds for all sufficiently large $k$.
  From \cref{prop:local_convergence}, we have \cref{eq:local_convergence_rate} for all large $k$ and hence
  \begin{align}
    \frac{\delta_{k+1}^{1 - 1/r}}{\delta_k^{1/r + 1/2}}
    \leq
    C_3
    +
    C_4
    \delta_k^{1/r - 1/2}
    \to
    C_3
  \end{align}
  as $k \to \infty$.
  Hence, the order of convergence is $\frac{1/r + 1/2}{1 - 1/r} = \frac{r+2}{2(r-1)}$.

  \paragraph{Case~\ref{thm:local_r12}: $1 < r < 2$ and $\Delta_\infty > 0$.}
  As with Case~\ref{thm:local_r12_delta0}, we have \cref{eq:local_convergence_rate} for all large $k$ and hence
  \begin{align}
    \frac{\delta_{k+1}^{1 - 1/r}}{\delta_k^{1/r}}
    \leq
    C_3
    \sqrt{\delta_k + \Delta_\infty}
    +
    C_4
    \delta_k^{1/r}
    \to
    C_3
    \sqrt{\Delta_\infty}
  \end{align}
  as $k \to \infty$.
  Hence, the order of convergence is $\frac{1/r}{1 - 1/r} = \frac{1}{r-1}$.

  \paragraph{Case~\ref{thm:local_r23_delta0}: $2 \leq r < 3$ and $\Delta_\infty = 0$.}
  As with Case~\ref{thm:local_r12_delta0}, we have \cref{eq:local_convergence_rate} for all large $k$ and hence
  \begin{align}
    \frac{\delta_{k+1}^{1 - 1/r}}{\delta_k^{2/r}}
    \leq
    C_3
    \delta_k^{1/2 - 1/r}
    +
    C_4
    \to
    \begin{dcases*}
      C_3 + C_4 & if $r = 2$,\\
      C_4 & if $2 < r < 3$
    \end{dcases*}
  \end{align}
  as $k \to \infty$.
  Hence, the order of convergence is $\frac{2/r}{1 - 1/r} = \frac{2}{r-1}$.

  \paragraph{Case~\ref{thm:local_r2_deltasmall}: $r = 2$ and $0 < \Delta_\infty < 1 / C_3^2$.}
  Since
  \begin{alignat}{2}
    C_1 
    &=
    \frac{L_c \sqrt{2 L_h}}{(1 - \theta) \rho_{\min}}
    + \frac{1}{\sqrt{1 - \theta^2}}
    \geq
    \frac{1}{\sqrt{1 - \theta^2}}
    \geq
    1,\\
    C_2
    &=
    (1 + \theta) \rho_{\max} + L_c \sqrt{2 L_h}
    \geq
    L_c \sqrt{2 L_h},
  \end{alignat}
  we have
  \begin{align}
    \Delta_\infty
    <
    \frac{1}{C_3^2}
    =
    \prn*{
      \frac{\gamma}{4 C_1 C_2}
    }^2
    \leq
    \frac{\gamma^2}{32 L_c^2 L_h}.
  \end{align}
  Thus, condition~\cref{eq:asm_vk_norm_upperbound2} holds if
  \begin{align}
    \delta_{k+1}
    \leq
    \prn*{
      \frac{\gamma}{2 L_c \sqrt{2 L_h}}
    }^2
    - 
    \frac{\gamma^2}{32 L_c^2 L_h}
    =
    \frac{3 \gamma^2}{32 L_c^2 L_h},
  \end{align}
  or if $k$ is sufficiently large.
  From \cref{prop:local_convergence}, we have \cref{eq:local_convergence_rate} for all large $k$ and hence
  \begin{align}
    \frac{\delta_{k+1}^{1/2}}{\delta_k^{1/2}}
    \leq
    C_3 \sqrt{\delta_{k+1} + \Delta_\infty}
    +
    C_4
    \delta_k^{1/2}
    \to
    C_3 \sqrt{\Delta_\infty}
    <
    1
  \end{align}
  as $k \to \infty$.
  Hence, the order of convergence is $\frac{1/2}{1/2} = 1$.
\end{proof}

The assumption of Case~\ref{thm:local_r1} of \cref{thm:local_obj} holds when, for example, $d = n = 1$ and
\begin{align}
  g(x)
  =
  \begin{dcases*}
    0 & if $|x| \leq 1$,\\
    +\infty & otherwise,
  \end{dcases*}
  \quad
  h(y) = y^2,\quad
  c(x) = x^2 - 2.
\end{align}
Existing analyses \citep{burke1995gauss,li2002convergence,drusvyatskiy2018error} of LM methods for a general function $h$ provide local quadratic convergence under the Lipschitz continuity of $h$, while we obtain faster convergence, convergence in finite iterations, under the Lipschitz continuity of $\nabla h$.
We do not have a clear answer to the question of whether there are instances that satisfy the assumptions in Cases~\ref{thm:local_r12_delta0} or \ref{thm:local_r12} but not Case~\ref{thm:local_r1}.

Case~\ref{thm:local_r23_delta0} with $r = 2$ yields quadratic convergence.
This result extends local quadratic convergence of LM methods for $h(\cdot) = \frac{1}{2} \norm{\cdot}^2$, which has been extensively studied \citep{yamashita2001rate,fan2005quadratic,fan2006convergence,dan2002convergence,fischer2010inexactness,bellavia2014strong,bellavia2010convergence,kanzow2004levenberg,behling2012unified,facchinei2013family}, to general $h$.
Similarly, our result in Case~\ref{thm:local_r2_deltasmall} extends the existing linear convergence result for $h(\cdot) = \frac{1}{2} \norm{\cdot}^2$ \citep{bergou2020convergence}.

We have shown the convergence of the sequence $(F(x_k))_{k \in \N}$, and  we see from \cref{asm:holderian_growth} that $(\dist(x_k, X^*))_{k \in \N}$ converges equally or faster.
Furthermore, we show that the solution sequence $(x_k)_{k \in \N}$ also converges to a stationary point in $X^*$.

\begin{theorem}
  Suppose that \cref{asm:hc,asm:nablac_bound,asm:local} hold.
  Let $(x_k)_{k \in \N}$ be generated by \cref{alg:proposed_backtracking}.
  If one of the following holds:
  \begin{itemize}
    \item 
    $1 \leq r < 2$,
    \item 
    $r = 2$ and $\Delta_\infty < 1 / C_3^2$, or
    \item 
    $2 < r < 3$ and $\Delta_\infty = 0$,
  \end{itemize}
  then the sequence $(x_k)_{k \in \N}$ converges to a stationary point $x^* \in X^*$.
\end{theorem}
\begin{proof}
  \cref{asm:holderian_growth} implies that the sequences $(\dist(x_k, X^*))_{k \in \N}$ converges to $0$ and \cref{thm:local_obj} shows that the sequence $(F(x_k))_{k \in \N}$ converges with order at least $1$.
  Therefore, we have some constants $K' \in \N$ and $0 \leq C_5 < 1$ such that \cref{eq:uk_norm_bound_simple} and
  \begin{align}
    F(x_{k+1}) - F_\infty
    \leq
    C_5 
    \prn*{ F(x_k) - F_\infty}
    \label{eq:diff_Fxki_upperbound}
  \end{align}
  hold for all $k \geq K'$.
  Using this bound, we obtain for all $l \geq k \geq K'$,
  \begin{alignat}{2}
    \norm*{x_l - x_k}
    &\leq
    \sum_{i=0}^\infty
    \norm*{x_{k+i+1} - x_{k+i}}
    &\quad&\by{the triangle inequality}\\
    &\leq
    \sum_{i=0}^\infty
    C_1 \dist(x_{k+i}, X^*)
    &\quad&\by{\cref{eq:uk_norm_bound_simple}}\\
    &\leq
    \sum_{i=0}^\infty
    C_1
    \prn*{\frac{r}{\gamma} \prn*{F(x_{k+i}) - F_\infty}}^{1/r}
    &\quad&\by{\cref{asm:holderian_growth}}\\
    &\leq
    \sum_{i=0}^\infty
    C_1
    \prn*{\frac{r}{\gamma} \prn*{F(x_k) - F_\infty}}^{1/r}
    C_5^{i/r}
    &\quad&\by{using \cref{eq:diff_Fxki_upperbound} recursively}\\
    &=
    C_1
    \prn*{\frac{r}{\gamma} \prn*{F(x_k) - F_\infty}}^{1/r}
    \frac{1}{1 - C_5^{1/r}}.
  \end{alignat}
  This upper bound converges to $0$ as $k \to \infty$, and hence $(x_k)_{k \in \N}$ is a Cauchy sequence and converges to a point $x^* \in X^*$.
\end{proof}





\section{Numerical experiments}
\label{sec:experiments}
We compare the proposed algorithm with existing algorithms on three instances of problem~\cref{eq:problem_main}.
We implemented all methods in Python with 
JAX~\citep{jax2018github}
and Flax~\citep{flax2020github}
and executed them on a computer with Apple M1 Chip (8 cores, 3.2 GHz) and 16 GB RAM.

\subsection{Problem setting}
The four instances used in our experiments are described below.

\paragraph{1. Multidimensional Rosenbrock function.}
The first optimization problem is the following:
\begin{align}
  \min_{x = (x_{(1)},\dots,x_{(d)}) \in \R^d} \ 
  \sum_{i=1}^{d-1} \prn*{
    \prn*{x_{(i)} - 1}^2 + 100 \prn*{x_{(i+1)} - x_{(i)}^2}^2
  }.
\end{align}
The objective function is a multidimensional extension of the Rosenbrock function~\citep{rosenbrock1960automatic}.
We set the dimension as $d = 10^4$ and initialized $x$ to $(0.5, \dots, 0.5) \in \R^d$.

\paragraph{2. Supervised classification with MNIST handwritten digit dataset.}
The second optimization problem is to construct a deep linear neural network model for classification problems:
\begin{align}
  \min_{w \in \R^d} \ 
  -
  \frac{1}{N}
  \sum_{i = 1}^N
  \inner{y_i}{\LogSoftmax(\phi_w(x_i))}.
\end{align}
Here, $x_i \in \R^{784}$ is the $i$-th image data, and $y_i \in \set{0, 1}^{10}$ is the corresponding one-hot vector that indicates the class label.
The function $\phi_w: \R^{784} \to \R^{10}$ is a three-layer fully connected neural network with parameter $w \in \R^d$ that receives image data as input and outputs a $10$-dimensional vector.
Two hidden layers have $128$ and $32$ nodes, respectively, and a sigmoid activation function, and thus the dimension of the parameter $w$ is $d = 104{,}938$, including bias parameters.\footnote{
  $d = (784 \times 128 + 128 \times 32 + 32 \times 10) + (128 + 32 + 10) = 104{,}938$.
}
The $i$-th entry of $\LogSoftmax(z)$ is defined by
\begin{align}
  (\LogSoftmax(z))_i
  \coloneqq
  \log \prn*{
    \frac{\exp(z_i)}{\sum_j \exp(z_j)}
  }.
\end{align}
The MNIST dataset contains 60,000 training data, of which $N = 6000$ were randomly chosen for use.
We initialized the parameter $w$ by using \texttt{flax.linen.Module.init} method.\footnote{
  See the documentation, \url{https://flax.readthedocs.io/en/latest/flax.linen.html}, for the detail of \texttt{flax.linen.Module.init}.
}

\paragraph{3. Nonnegative matrix factorization (NMF) with MovieLens 100K dataset.}
The third optimization problem is the following:
\begin{align}
  \min_{U \in \R^{p \times r},\, V \in \R^{q \times r}}
  &\quad
  \frac{1}{N}
  \sum_{(i, j, s) \in \Omega}
  \prn*{\inner{u_i}{v_j} - s}^2
  + \lambda \prn*{
    \norm{U}_{\mathrm F}^2
    + \norm{V}_{\mathrm F}^2
  }\\
  \text{subject to}
  &\quad
  U \geq O,\ V \geq O.
\end{align}
Here, $u_i \in \R^r$ and $v_i \in \R^r$ denotes the $i$-th row of $U$ and $V$, respectively, and the inequalities in the constraints are elementwise.
The set $\Omega$ consists of training data of size $N$; each element $(i, j, s) \in \Omega$ means that user $i \in \set{1,\dots,p}$ rated movie $j \in \set{1,\dots,q}$ as $s \in \set{1,\dots,5}$.
The MovieLens dataset contains 100,000 ratings from $p = 943$ users on $q = 1682$ movies and these ratings are split into $N = 80{,}000$ training data and 20,000 test data.\footnote{
  The dataset provides five train-test splits \texttt{u1} to \texttt{u5}, of which we used \texttt{u1}.
}
We set the dimension of the latent features to $r = 500$ and the regularization parameter to $\lambda = 10^{-10}$.
The dimension of the parameter to be estimated is $pr + qr = 1{,}312{,}500$.
We initialized each entry of $U$ and $V$ by \revise{the uniform distribution on $[0, 10^{-3}]$}.

\revise{
\paragraph{4. Inverse problem on wave equation.}
The fourth is a data assimilation problem for the following nonlinear partial differential equation (PDE) from \citep{bellavia2018levenberg}:
\begin{alignat}{3}
  &
  \frac{\partial^2 u}{\partial t^2} (z, t)
  =
  \frac{\partial^2 u}{\partial z^2} (z, t)
  - e^{u(z, t)},
  &\quad& \forall z, t \in [0, 1],\\
  &
  u(0, t) = u(1, t) = 0,
  &\quad& \forall t \in [0, 1],\\
  &
  \frac{\partial u}{\partial t} (z, 0) = 0,
  &\quad& \forall z \in [0, 1].
\end{alignat}
Let us consider estimating the solution $u(z, 0)$ at time $t = 0$, given some observed data of a function $u$ that follows the PDE.
To solve this problem, we discretize the function $u$ using a $K \times T$ mesh as $u_{i,j} \coloneqq u(i/K, j/T)$, where $0 \leq i \leq K$ and $0 \leq j \leq T$.
Given $(u_{i,0})_{0 \leq i \leq K}$, we compute an estimate $\hat u_{i,j}$ for $u_{i,j}$ via a discrete system for the PDE.
Then, the data assimilation problem can be formulated as the optimization problem:
\begin{align}
  \min_{(u_{i,0})_{0 \leq i \leq K}} \ 
  \frac{1}{N} \sum_{(i, j, v) \in \Omega}
  \prn*{ \hat u_{i,j} - v }^2.
\end{align}
Above, the set $\Omega$ consists of observed data of size $N$; each element $(i, j, v) \in \Omega$ means that the observed value of $u_{i,j}$ is $v$.
Note that $\hat u_{i,j}$ is a function of $(u_{i,0})_{0 \leq i \leq K}$.
We set the parameters as $K = 2^6$, $T = 2^8$, and $N = 2^7$, and set the ground truth of $u(z, 0)$ as
\begin{align}
  u(z, 0)
  =
  \sin(6 \pi z) + 
  \begin{dcases*}
    1 - \cos(20 \pi z) & if $\frac{4}{10} \leq z \leq \frac{5}{10}$, \\
    0 & otherwise.
  \end{dcases*}
\end{align}
The observed data were generated by computing $\hat u_{i,j}$ with the above $u(z, 0)$ and adding random noise that follows the Gaussian distribution $\mathcal N(0, \sigma^2)$ with $\sigma = 10^{-2}$.
We used the Euler method to discretize the PDE.
}




\subsection{Algorithms}
We implemented the following four algorithms.
\begin{itemize}
  \item 
  \textbf{\Proposed (\cref{alg:proposed_backtracking} with \cref{alg:apg}).}
  We set the parameters as $\theta = 0.5$, $\alpha = \bar \alpha = 2$, $\bar \beta = 0.95$, and $\rho_{\min} \in \set{10^{-4}, 10^{-3},\dots, 10^{0}}$, and denote this algorithm by $\Proposed(\rho_{\min})$.
  \revise{Because we can check that $\inf_{x \in \R^d} g(x) = \inf_{y \in \R^n} h(y) = 0$ in all four problem setups presented in the previous section, we set $g^* = h^* = 0$ unless otherwise mentioned.}
  \item 
  \textbf{\DP \citep[Algorithm 6]{drusvyatskiy2019efficiency}.}
  This is an LM method that solves the subproblem with an accelerated method like ours but fixes the damping parameter $\mu$ to a constant, which is an input parameter.
  The method assumes the value of $L_h \sigma^2$ is known, where $L_h$ and $\sigma$ are the same as those in \cref{asm:h_smooth,asm:nablac_bound}.
  Since this value is not actually known, we treat it as another input parameter, $L$.
  We set the parameters as $\mu \in \set{10^{-4}, 10^{-3},\dots, 10^2}$ and $L \in \set{10^{-1}, 10^0,\dots,10^5}$, and denote this algorithm by $\DP(\mu, L)$.
  \item
  \revise{
  \textbf{\ABO \citep[Algorithm 3]{aravkin2022proximal}.}
  This is a trust-region method that employs quasi-Newton models for subproblems.
  We used L-BFGS approximations with memory 5 and solved the subproblem via the proximal gradient method in accordance with the original paper.
  We set the step-size and termination condition for the proximal gradient method as in \citep[Corollary~4.3 and Section~7]{aravkin2022proximal}, respectively.
  The input parameters of \citep[Algorithm 3]{aravkin2022proximal} were set to $\eta_1 = 1/4$, $\eta_2 = 3/4$, $\gamma_1 = \gamma_2 = 1/2$, $\gamma_3 = \gamma_4 = 2$, $\alpha = 1$, $\beta = 2$, and $\Delta_0 = 1$.
  }
  \item 
  \textbf{\PG (proximal gradient).}
  The step-size is automatically selected by backtracking, similarly to $\rho$ in \cref{alg:proposed_backtracking}.
  This algorithm has parameters $\alpha$ and $L_{\min}$ corresponding to $\alpha$ and $\rho_{\min}$ in \cref{alg:proposed_backtracking}.
  We set them as $\alpha = 2$ and $L_{\min} \in \set{10^{-4}, 10^{-3},\dots, 10^{0}}$, and denote this algorithm by $\PG(L_{\min})$.
\end{itemize}
\revise{
See \cref{sec:implementation} for the detail of the implementation.
}

\revise{
We note that \Proposed and \DP assume the smooth term of the objective function to be written as $h(c(\cdot))$ and call oracles on $c$ and $h$, while \ABO and \PG make no such assumptions and can handle more general problems.
}


\subsection{Results}
\cref{fig:rosenbrock_large,fig:mnist,fig:movielens,fig:wave_equation} show the results of the experiments.
We ran each algorithm with several parameter settings as described in the previous section; only the results with the best performing parameters are plotted.
If the best one cannot be determined, multiple lines are plotted.
\revise{In the setup shown in \cref{fig:movielens}, \ABO could not complete the first iteration within the time limit (200 seconds), so the plot was omitted.}
In all \revise{four} instances, the proposed method decreases the objective function value faster than or as fast as \revise{the compared methods.}

\revise{
The main difference between \DP and our method is how to set the damping parameter $\mu$; \DP fixes $\mu$ as a constant while we adaptively choose it at each iteration.
This adaptability is considered one reason for the fast convergence of the proposed method.
\cref{fig:rosenbrock_large} shows that, for \DP, $\mu = 10^1$ is faster in the early stage while $\mu = 10^0$ is faster at the end, which also supports the importance of the adaptability of $\mu$.
}

\revise{
\ABO and \Proposed differ in many respects, the most significant difference being the form of the subproblems.
\ABO updates (a compact form of) approximated Hessian as in the quasi-Newton method, while \Proposed constructs the LM subproblem and uses a Jacobian-vector oracle.
The latter method seems to be more advantageous, especially for large-scale problems.
In fact, \cref{fig:wave_equation} shows that ABO performs well for a small problem with 63 variables, but it performs worse for a large-scale problem with 1312500 variables, as in \cref{fig:movielens}.
Another possible reason our method outperformed \ABO is that our termination condition (\cref{alg-line-apg:termination_cond} of \cref{alg:apg}) for the subproblem was carefully designed based on theoretical analysis.
Generally speaking, solving subproblems too accurately or inaccurately will degrade the overall performance of algorithms; our condition enables us to solve the subproblem with just the right amount of accuracy.
We should mention that \ABO can be applied even if $g$ is nonconvex or if the smooth term of the objective function is not a composite function.
}%

\revise{%
While \PG is a first-order method and thus cannot hope for local superlinear convergence, the proposed method achieves the theoretical guarantee of locally fast convergence.
This local convergence of our method was also observed in the experiments.
The figures on the right side of \cref{fig:rosenbrock_large,fig:movielens} suggest that \Proposed converges superlinearly.
}


\begin{figure}
  \centering%
  \subcaptionbox{Rosenbrock\label{fig:rosenbrock_large}}
  [\linewidth]{%
    \includegraphics[width=0.49\linewidth]{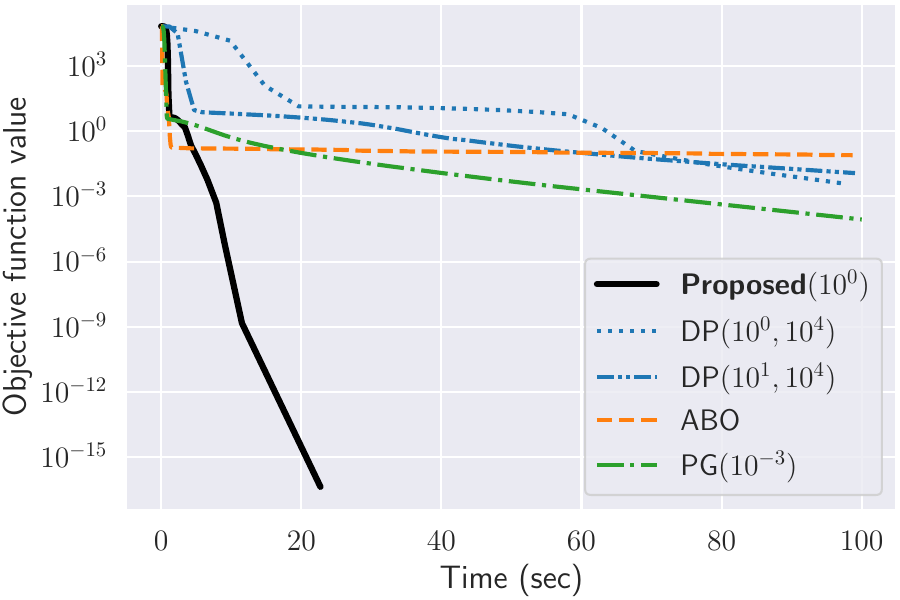}\hfill%
    \includegraphics[width=0.49\linewidth]{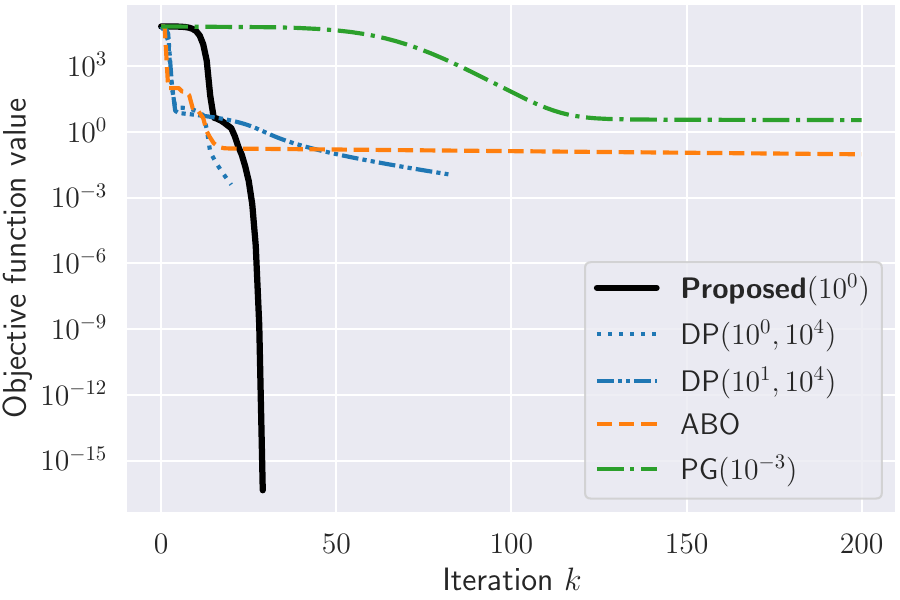}%
  }\par\medskip%
  \subcaptionbox{Classification\label{fig:mnist}}
  [\linewidth]{%
    \includegraphics[width=0.49\linewidth]{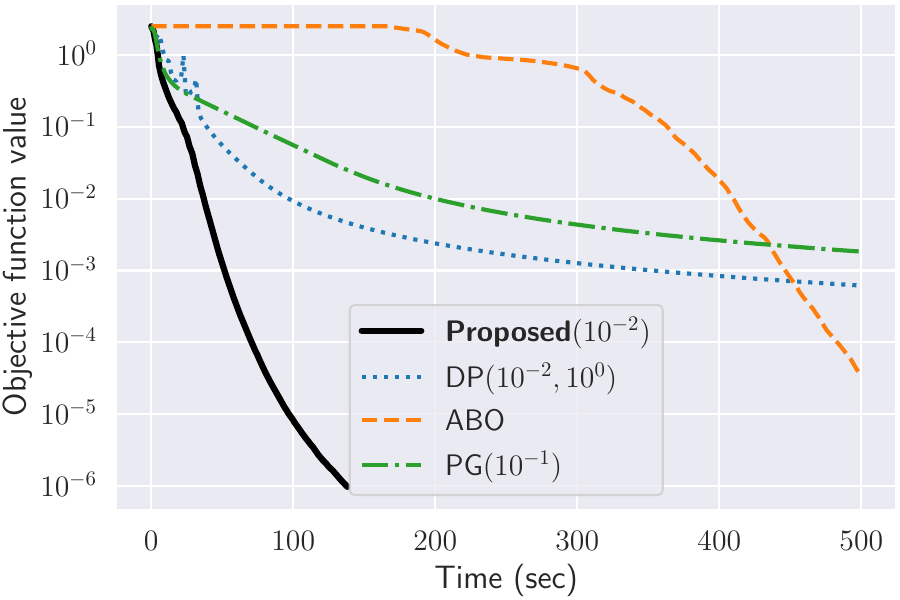}\hfill%
    \includegraphics[width=0.49\linewidth]{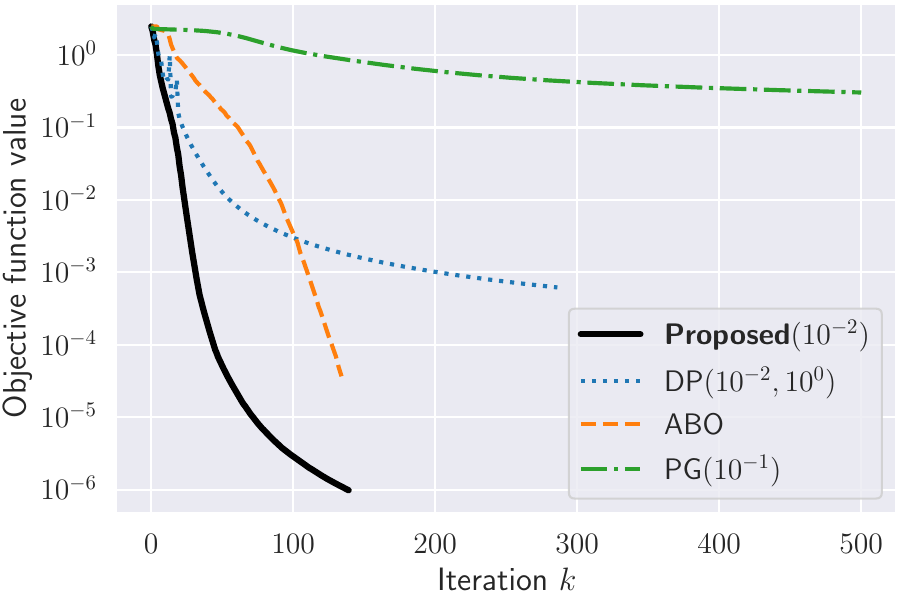}%
  }\par\medskip%
  \subcaptionbox{NMF\label{fig:movielens}}
  [\linewidth]{%
    \includegraphics[width=0.49\linewidth]{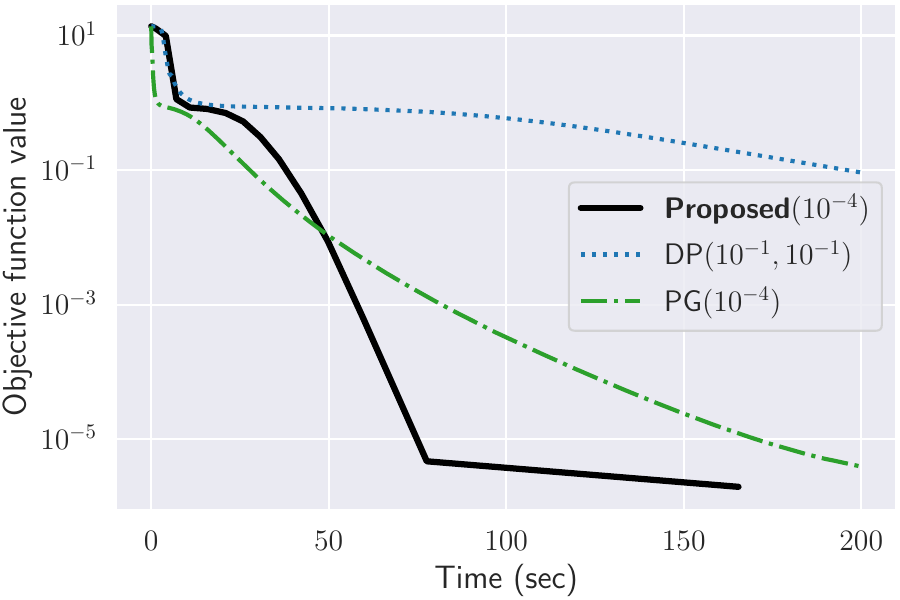}\hfill%
    \includegraphics[width=0.49\linewidth]{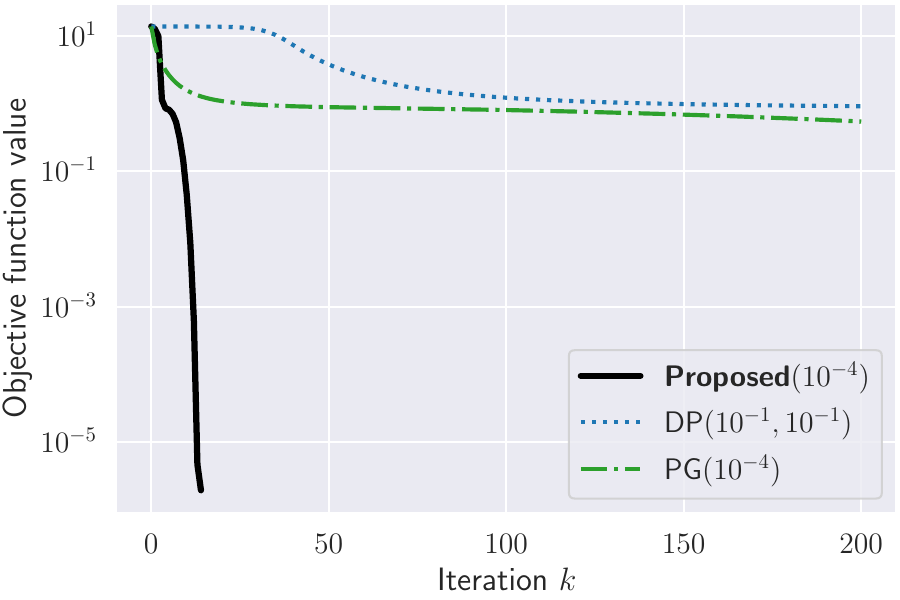}%
  }\par\medskip%
  \subcaptionbox{Wave equation\label{fig:wave_equation}}
  [\linewidth]{%
    \includegraphics[width=0.49\linewidth]{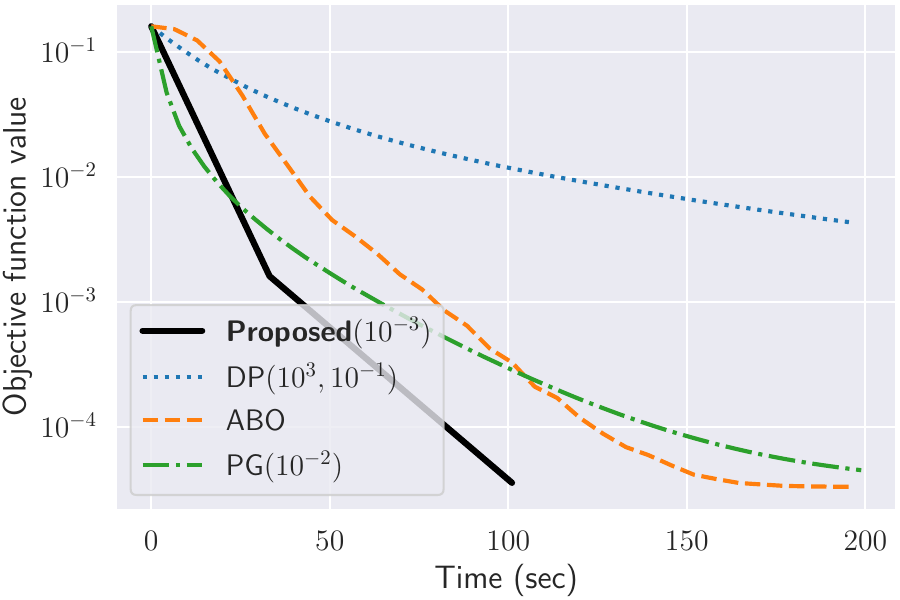}\hfill%
    \includegraphics[width=0.49\linewidth]{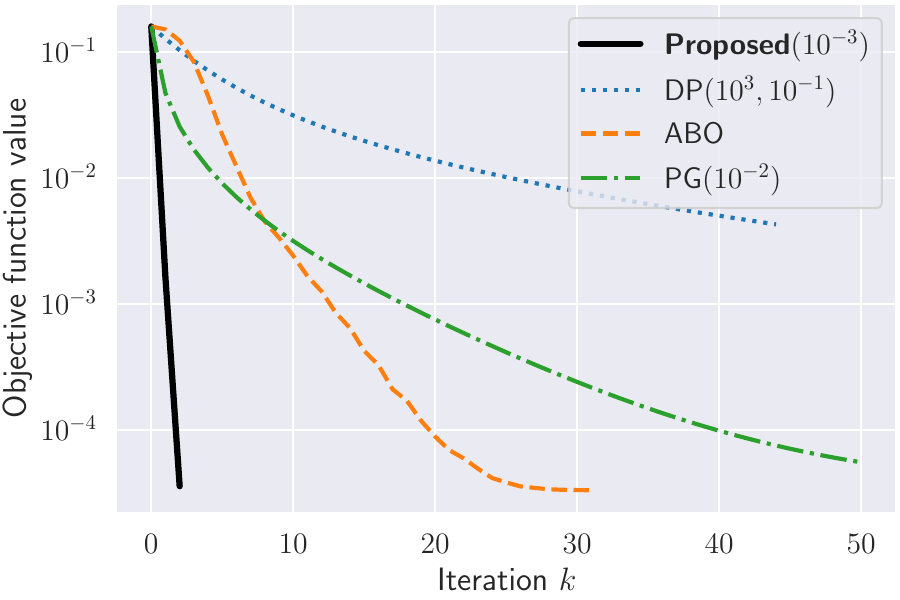}%
  }\par%
  \caption{
    Numerical results of each instance and method.
  }
  \label{fig:exp_methods}
\end{figure}



\begin{figure}[t]
  \centering%
  \subcaptionbox{Rosenbrock}
  [0.49\linewidth]{%
    \includegraphics[width=\linewidth]{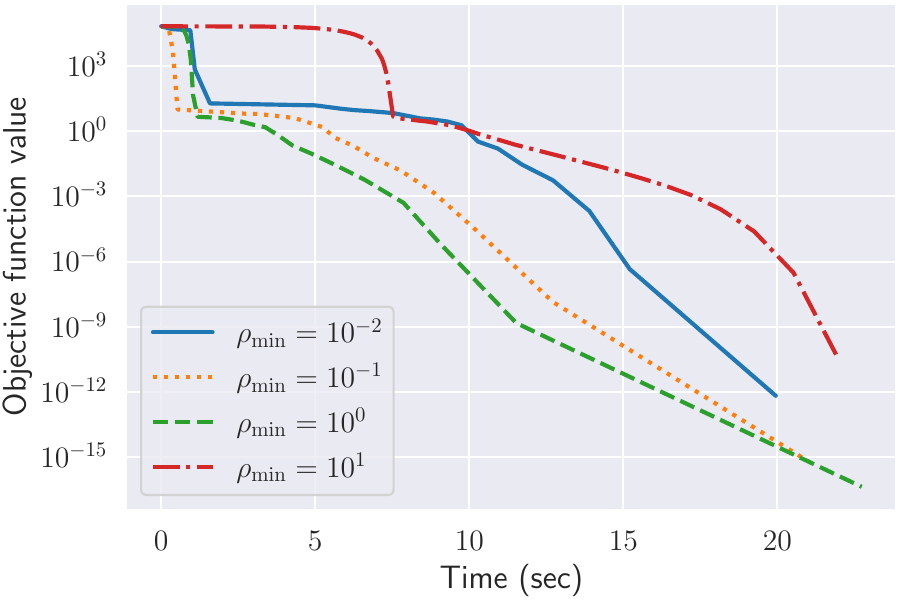}%
  }\hfill%
  \subcaptionbox{Classification}
  [0.49\linewidth]{%
    \includegraphics[width=\linewidth]{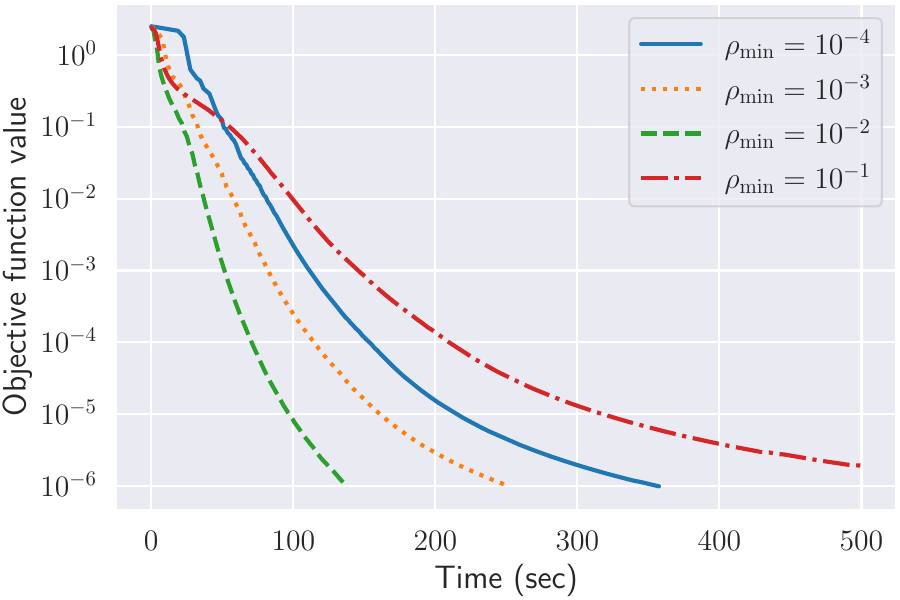}%
  }\par\medskip%
  \subcaptionbox{NMF}
  [0.49\linewidth]{%
    \includegraphics[width=\linewidth]{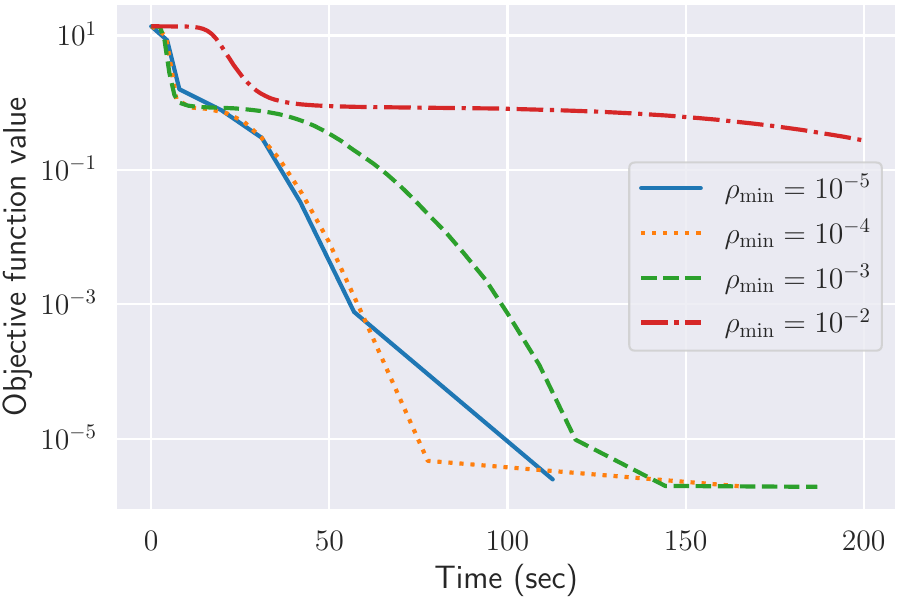}%
  }\hfill%
  \subcaptionbox{Wave equation}
  [0.49\linewidth]{%
    \includegraphics[width=\linewidth]{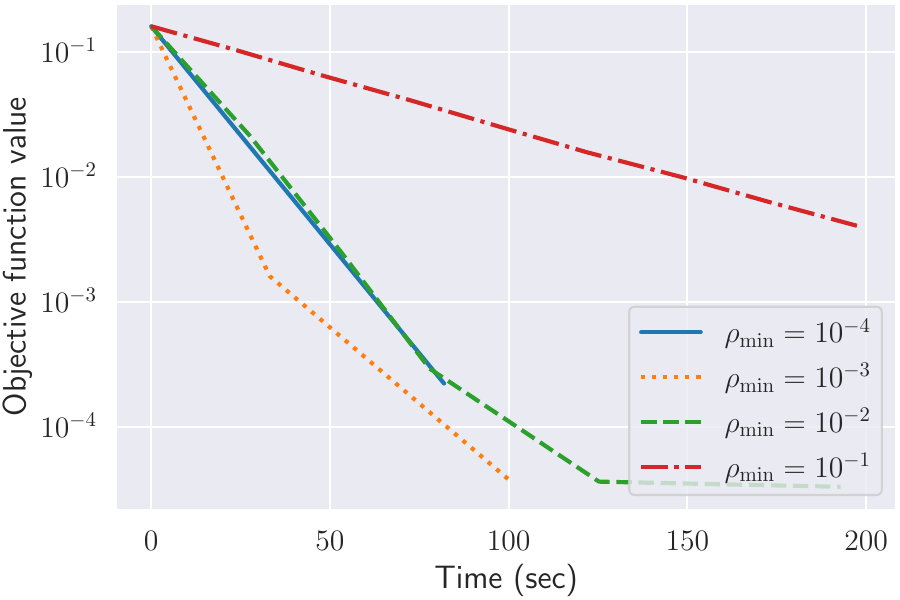}%
  }\par%
  \caption{
    Performance of the proposed algorithm with different values of $\rho_{\mathrm{min}}$.
  }
  \label{fig:exp_rhomin}
\end{figure}

\revise{
\paragraph{Sensitivity to $\rho_{\min}$.}
The proposed algorithm has an input parameter $\rho_{\min}$.
\Cref{fig:exp_rhomin} shows the performance of the proposed algorithm with several settings of $\rho_{\min}$.
The results suggest that setting $\rho_{\min}$ too large can deteriorate performance to a greater extent than setting $\rho_{\min}$ too small.
This is because \cref{alg:proposed_backtracking} can increase $\rho$ but not decrease it.
A practical (but not theoretically grounded) way to reduce the sensitivity of \cref{alg:proposed_backtracking} to $\rho_{\min}$ is to insert a step to decrease $\rho$ in \cref{alg-line-bt:update_xk_muk}.
}

\revise{
\paragraph{Sensitivity to $g^* + h^*$.}
The proposed algorithm has other input parameters, $g^*$ and $h^*$, that should be set as in \cref{eq:requirement_gast_hast}.
So far, we have set $g^* = h^* = 0$ since we know that $\inf_{x \in \R^d} g(x) = \inf_{y \in \R^n} h(y) = 0$ in all four experimental settings.
Here, let us hypothetically consider the case where we do not know the infimum of $g$ or $h$.
A practical choice in such cases is to compute numerically the minimum of the convex functions $g$ and $h$.
Numerical solutions often give an upper bound, but subtracting a small value from the upper bound is expected to result in a lower bound.
Assuming such a case, we tested the proposed algorithm with different values of $g^* + h^*$, and the results are presented in \Cref{fig:exp_ghmin}.\footnote{
  Since \Cref{alg:proposed_backtracking} uses $g^*$ and $h^*$ only in the form $g^* + h^*$, we need only to focus on the value of $g^* + h^*$.
}
The results suggest that the algorithm's performance is not significantly affected as long as $g^* + h^*$ is close to $0$ (i.e., $\inf_{x \in \R^d} g(x) + \inf_{y \in \R^n} h(y)$).
However, if their differences become too large, the performance deteriorates.
Note that we should consider the scale of the function value to compare the magnitude of the value of $g^* + h^*$.
For example, in the NMF setting, the function $h$ has the form
\begin{align}
  h(u, v, w)
  = 
  \frac{1}{N} \norm*{w}^2
  + \lambda \prn*{ \norm*{u}^2 + \norm*{v}^2 }
\end{align}
with $N = 8 \times 10^5$ and $\lambda = 10^{-10}$, which suggests that $-10^0$ is too small as $g^* + h^*\, (< 0)$ considering the scale of the function. }

\begin{figure}[t]
  \centering%
  \subcaptionbox{Rosenbrock}
  [0.49\linewidth]{%
    \includegraphics[width=\linewidth]{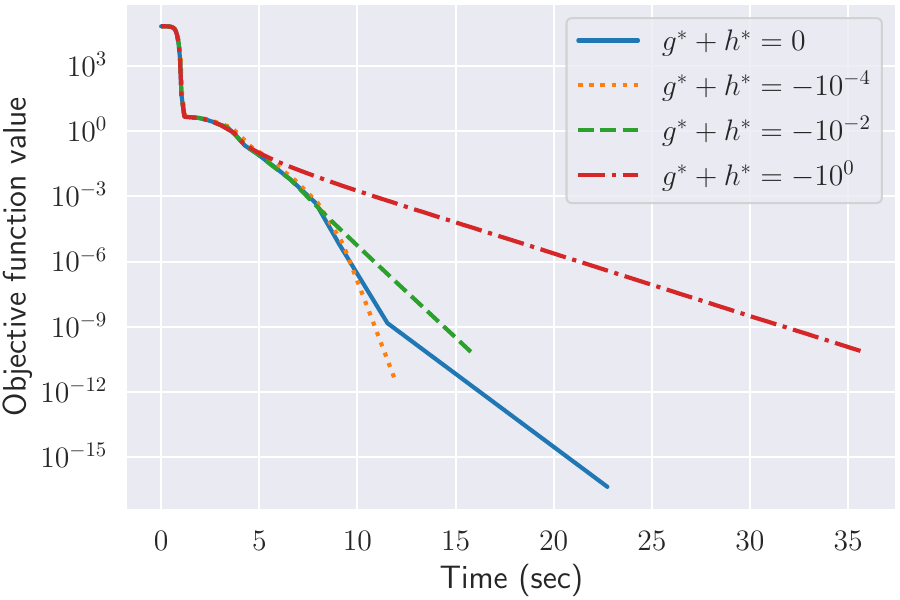}%
  }\hfill%
  \subcaptionbox{Classification}
  [0.49\linewidth]{%
    \includegraphics[width=\linewidth]{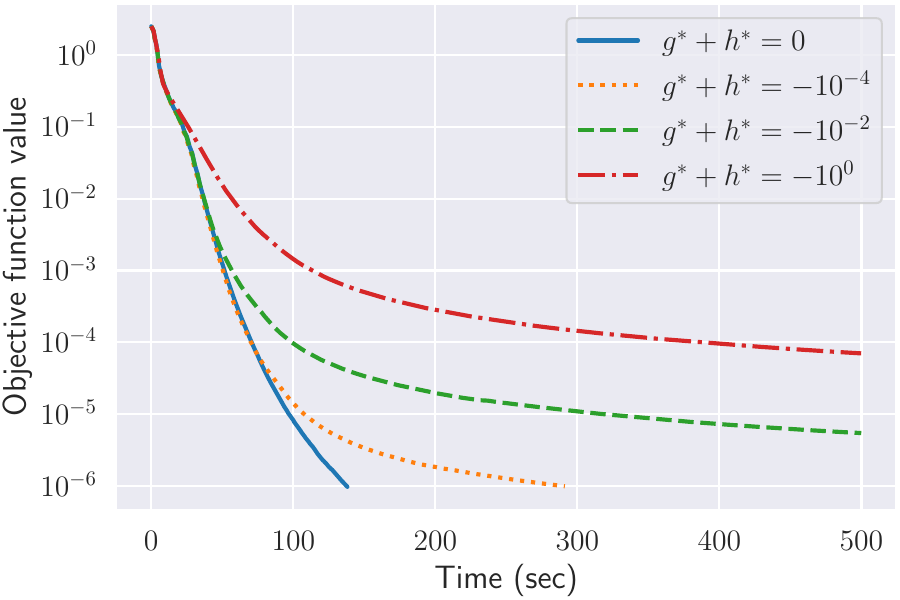}%
  }\par\medskip%
  \subcaptionbox{NMF}
  [0.49\linewidth]{%
    \includegraphics[width=\linewidth]{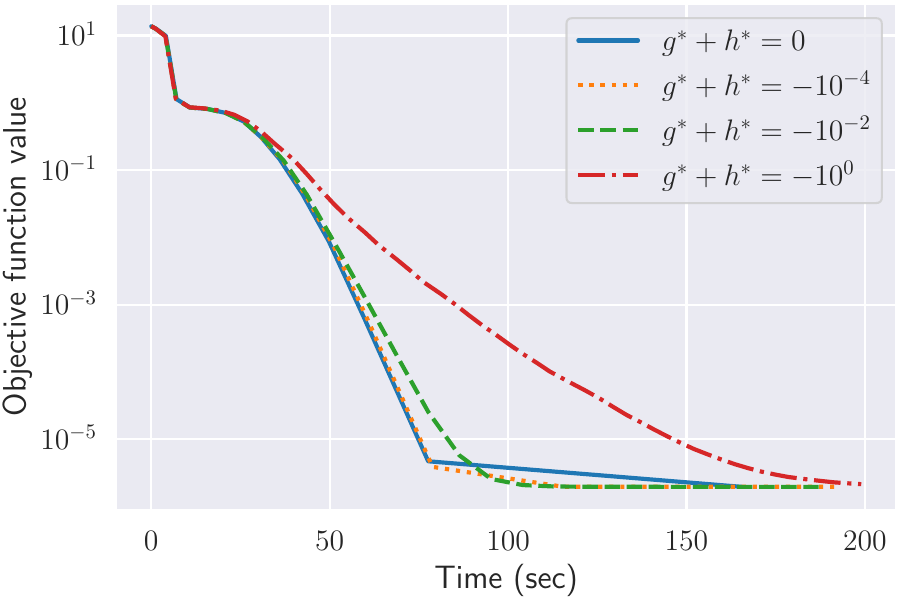}%
  }\hfill%
  \subcaptionbox{Wave equation}
  [0.49\linewidth]{%
    \includegraphics[width=\linewidth]{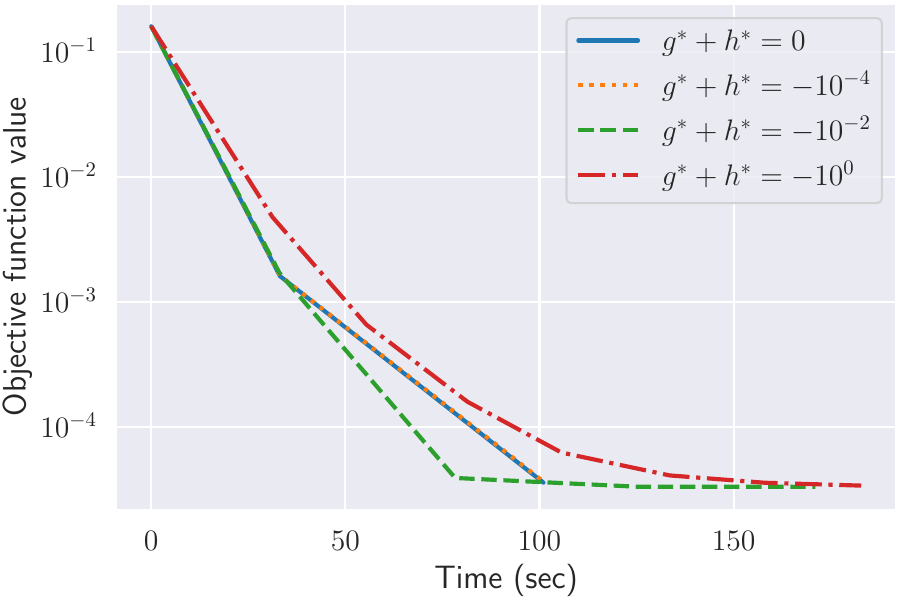}%
  }\par%
  \caption{
    Performance of the proposed algorithm with different values of $g^* + h^*$.
  }
  \label{fig:exp_ghmin}
\end{figure}

\section{Conclusion}
\label{sec:conclusion}
We proposed a new LM method (\cref{alg:proposed_backtracking}) for minimizing the sum of a convex function and a smooth composite function, problem \cref{eq:problem_main}.
The proposed method enjoys the following three theoretical guarantees: iteration complexity bound, oracle complexity bound when combined with an accelerated method (\cref{alg:apg}), and local convergence under the H\"olderian growth condition.
As shown in \cref{table:complexity}, our iteration complexity bound improves existing bounds by $\Theta(\kappa)$, and our oracle complexity bound improves by $\tilde \Theta(\sqrt{\kappa})$ for non-Lipschitz $h$, where $\kappa$ is a constant like the condition number.
Our local convergence results include local quadratic convergence under the quadratic growth condition; this is the first to extend the classical result \citep{yamashita2001rate} for least-squares problems, $h(\cdot) = \frac{1}{2} \norm*{\cdot}^2$, to a general function $h$.
In addition, this is the first LM method with both an oracle complexity bound and local quadratic convergence for smooth problems to the best of our knowledge.

Our numerical experiments show that our method is effective for large-scale problems with $10^4$ to $10^6$ variables, but there is potential for further speed-up, mainly when the functions $h$ and $c$ can be decomposed as follows:
\begin{align}
  h(y_1,\dots, y_N)
  =
  \frac{1}{N} \sum_{i=1}^N h_i (y_i),\quad
  c(x)
  = (c_1(x), \dots, c_N(x)).
\end{align}
Such a finite-sum setting often appears in machine learning, and stochastic methods \citep{lan2020first} effectively solve the problem when $N$ is large.
An interesting direction for future work would be to investigate whether oracle complexity bound and local convergence can still be achieved with stochastic LM methods.

\begin{acknowledgements}
  \revise{We are deeply grateful to the anonymous reviewers, who carefully read the manuscript and provided helpful and constructive comments.}
\end{acknowledgements}

\section*{Code availability}
The code used for this article is found in \url{https://github.com/n-marumo/accelerated-lm}.

\section*{Conflict of interest}
The authors declare that they have no conflict of interest.

\bibliographystyle{abbrvnat}
\bibliography{myrefs}   

\begin{thebibliography}{42}
\providecommand{\natexlab}[1]{#1}
\providecommand{\url}[1]{\texttt{#1}}
\expandafter\ifx\csname urlstyle\endcsname\relax
  \providecommand{\doi}[1]{doi: #1}\else
  \providecommand{\doi}{doi: \begingroup \urlstyle{rm}\Url}\fi

\bibitem[Ahookhosh et~al.(2019)Ahookhosh, Arag{\'o}n~Artacho, Fleming, and Vuong]{ahookhosh2019local}
M.~Ahookhosh, F.~J. Arag{\'o}n~Artacho, R.~M.~T. Fleming, and P.~T. Vuong.
\newblock Local convergence of the {Levenberg--Marquardt} method under {H\"older} metric subregularity.
\newblock \emph{Advances in Computational Mathematics}, 45\penalty0 (5):\penalty0 2771--2806, 2019.
\newblock \doi{10.1007/s10444-019-09708-7}.
\newblock URL \url{https://doi.org/10.1007/s10444-019-09708-7}.

\bibitem[Aravkin et~al.(2022)Aravkin, Baraldi, and Orban]{aravkin2022proximal}
A.~Y. Aravkin, R.~Baraldi, and D.~Orban.
\newblock A proximal quasi-{N}ewton trust-region method for nonsmooth regularized optimization.
\newblock \emph{SIAM Journal on Optimization}, 32\penalty0 (2):\penalty0 900--929, 2022.
\newblock \doi{10.1137/21M1409536}.
\newblock URL \url{https://doi.org/10.1137/21M1409536}.

\bibitem[Aravkin et~al.(2023)Aravkin, Baraldi, and Orban]{aravkin2023levenberg}
A.~Y. Aravkin, R.~Baraldi, and D.~Orban.
\newblock A {L}evenberg--{M}arquardt method for nonsmooth regularized least squares.
\newblock \emph{arXiv preprint arXiv:2301.02347}, 2023.

\bibitem[Bao et~al.(2019)Bao, Yu, Wang, Hu, and Yao]{bao2019modified}
J.~Bao, C.~K.~W. Yu, J.~Wang, Y.~Hu, and J.-C. Yao.
\newblock Modified inexact {Levenberg--Marquardt} methods for solving nonlinear least squares problems.
\newblock \emph{Computational Optimization and Applications}, 74\penalty0 (2):\penalty0 547--582, 2019.
\newblock \doi{10.1007/s10589-019-00111-y}.
\newblock URL \url{https://doi.org/10.1007/s10589-019-00111-y}.

\bibitem[Beck(2017)]{beck2017first}
A.~Beck.
\newblock \emph{First-Order Methods in Optimization}.
\newblock Society for Industrial and Applied Mathematics, Philadelphia, PA, 2017.
\newblock \doi{10.1137/1.9781611974997}.
\newblock URL \url{https://epubs.siam.org/doi/abs/10.1137/1.9781611974997}.

\bibitem[Behling and Fischer(2012)]{behling2012unified}
R.~Behling and A.~Fischer.
\newblock A unified local convergence analysis of inexact constrained {Levenberg--Marquardt} methods.
\newblock \emph{Optimization Letters}, 6\penalty0 (5):\penalty0 927--940, 2012.
\newblock \doi{10.1007/s11590-011-0321-3}.
\newblock URL \url{https://doi.org/10.1007/s11590-011-0321-3}.

\bibitem[Bellavia and Morini(2014)]{bellavia2014strong}
S.~Bellavia and B.~Morini.
\newblock {Strong local convergence properties of adaptive regularized methods for nonlinear least squares}.
\newblock \emph{IMA Journal of Numerical Analysis}, 35\penalty0 (2):\penalty0 947--968, 05 2014.
\newblock ISSN 0272-4979.
\newblock \doi{10.1093/imanum/dru021}.
\newblock URL \url{https://doi.org/10.1093/imanum/dru021}.

\bibitem[Bellavia et~al.(2010)Bellavia, Cartis, Gould, Morini, and Toint]{bellavia2010convergence}
S.~Bellavia, C.~Cartis, N.~I.~M. Gould, B.~Morini, and P.~L. Toint.
\newblock Convergence of a regularized {Euclidean} residual algorithm for nonlinear least-squares.
\newblock \emph{SIAM Journal on Numerical Analysis}, 48\penalty0 (1):\penalty0 1--29, 2010.
\newblock \doi{10.1137/080732432}.
\newblock URL \url{https://doi.org/10.1137/080732432}.

\bibitem[Bellavia et~al.(2018)Bellavia, Gratton, and Riccietti]{bellavia2018levenberg}
S.~Bellavia, S.~Gratton, and E.~Riccietti.
\newblock A {Levenberg--Marquardt} method for large nonlinear least-squares problems with dynamic accuracy in functions and gradients.
\newblock \emph{Numerische Mathematik}, 140\penalty0 (3):\penalty0 791--825, 2018.
\newblock \doi{10.1007/s00211-018-0977-z}.
\newblock URL \url{https://doi.org/10.1007/s00211-018-0977-z}.

\bibitem[Bergou et~al.(2020)Bergou, Diouane, and Kungurtsev]{bergou2020convergence}
E.~H. Bergou, Y.~Diouane, and V.~Kungurtsev.
\newblock Convergence and complexity analysis of a {Levenberg--Marquardt} algorithm for inverse problems.
\newblock \emph{Journal of Optimization Theory and Applications}, 185\penalty0 (3):\penalty0 927--944, 2020.
\newblock \doi{10.1007/s10957-020-01666-1}.
\newblock URL \url{https://doi.org/10.1007/s10957-020-01666-1}.

\bibitem[Bradbury et~al.(2018)Bradbury, Frostig, Hawkins, Johnson, Leary, Maclaurin, Necula, Paszke, Vander{P}las, Wanderman-{M}ilne, and Zhang]{jax2018github}
J.~Bradbury, R.~Frostig, P.~Hawkins, M.~J. Johnson, C.~Leary, D.~Maclaurin, G.~Necula, A.~Paszke, J.~Vander{P}las, S.~Wanderman-{M}ilne, and Q.~Zhang.
\newblock {JAX}: composable transformations of {P}ython+{N}um{P}y programs, 2018.
\newblock URL \url{https://github.com/google/jax}.

\bibitem[Burke and Ferris(1995)]{burke1995gauss}
J.~V. Burke and M.~C. Ferris.
\newblock A {Gauss--Newton} method for convex composite optimization.
\newblock \emph{Mathematical Programming}, 71\penalty0 (2):\penalty0 179--194, 1995.
\newblock \doi{10.1007/BF01585997}.
\newblock URL \url{https://doi.org/10.1007/BF01585997}.

\bibitem[Carmon et~al.(2020)Carmon, Duchi, Hinder, and Sidford]{carmon2020lower}
Y.~Carmon, J.~C. Duchi, O.~Hinder, and A.~Sidford.
\newblock Lower bounds for finding stationary points {I}.
\newblock \emph{Mathematical Programming}, 184\penalty0 (1):\penalty0 71--120, 2020.
\newblock \doi{10.1007/s10107-019-01406-y}.
\newblock URL \url{https://doi.org/10.1007/s10107-019-01406-y}.

\bibitem[Cartis et~al.(2010)Cartis, Gould, and Toint]{cartis2010complexity}
C.~Cartis, N.~I.~M. Gould, and P.~L. Toint.
\newblock On the complexity of steepest descent, {Newton's} and regularized {Newton's} methods for nonconvex unconstrained optimization problems.
\newblock \emph{SIAM Journal on Optimization}, 20\penalty0 (6):\penalty0 2833--2852, 2010.
\newblock \doi{10.1137/090774100}.
\newblock URL \url{https://doi.org/10.1137/090774100}.

\bibitem[Cartis et~al.(2011)Cartis, Gould, and Toint]{cartis2011evaluation}
C.~Cartis, N.~I.~M. Gould, and P.~L. Toint.
\newblock On the evaluation complexity of composite function minimization with applications to nonconvex nonlinear programming.
\newblock \emph{SIAM Journal on Optimization}, 21\penalty0 (4):\penalty0 1721--1739, 2011.
\newblock \doi{10.1137/11082381X}.
\newblock URL \url{https://doi.org/10.1137/11082381X}.

\bibitem[Dan et~al.(2002)Dan, Yamashita, and Fukushima]{dan2002convergence}
H.~Dan, N.~Yamashita, and M.~Fukushima.
\newblock Convergence properties of the inexact {Levenberg--Marquardt} method under local error bound conditions.
\newblock \emph{Optimization Methods and Software}, 17\penalty0 (4):\penalty0 605--626, 2002.
\newblock \doi{10.1080/1055678021000049345}.
\newblock URL \url{https://doi.org/10.1080/1055678021000049345}.

\bibitem[d'Aspremont et~al.(2021)d'Aspremont, Scieur, and Taylor]{daspremont2021acceleration}
A.~d'Aspremont, D.~Scieur, and A.~Taylor.
\newblock \emph{Acceleration Methods}.
\newblock Foundations and Trends{\textregistered} in Optimization. Now Publishers, 2021.
\newblock \doi{10.1561/2400000036}.
\newblock URL \url{http://dx.doi.org/10.1561/2400000036}.

\bibitem[Drusvyatskiy and Lewis(2018)]{drusvyatskiy2018error}
D.~Drusvyatskiy and A.~S. Lewis.
\newblock Error bounds, quadratic growth, and linear convergence of proximal methods.
\newblock \emph{Mathematics of Operations Research}, 43\penalty0 (3):\penalty0 919--948, 2018.
\newblock \doi{10.1287/moor.2017.0889}.
\newblock URL \url{https://doi.org/10.1287/moor.2017.0889}.

\bibitem[Drusvyatskiy and Paquette(2019)]{drusvyatskiy2019efficiency}
D.~Drusvyatskiy and C.~Paquette.
\newblock Efficiency of minimizing compositions of convex functions and smooth maps.
\newblock \emph{Mathematical Programming}, 178\penalty0 (1):\penalty0 503--558, 2019.
\newblock \doi{10.1007/s10107-018-1311-3}.
\newblock URL \url{https://doi.org/10.1007/s10107-018-1311-3}.

\bibitem[Facchinei et~al.(2013)Facchinei, Fischer, and Herrich]{facchinei2013family}
F.~Facchinei, A.~Fischer, and M.~Herrich.
\newblock A family of {Newton} methods for nonsmooth constrained systems with nonisolated solutions.
\newblock \emph{Mathematical Methods of Operations Research}, 77\penalty0 (3):\penalty0 433--443, 2013.
\newblock \doi{10.1007/s00186-012-0419-0}.
\newblock URL \url{https://doi.org/10.1007/s00186-012-0419-0}.

\bibitem[Fan(2006)]{fan2006convergence}
J.~Fan.
\newblock Convergence rate of the trust region method for nonlinear equations under local error bound condition.
\newblock \emph{Computational Optimization and Applications}, 34\penalty0 (2):\penalty0 215--227, 2006.
\newblock \doi{10.1007/s10589-005-3078-8}.
\newblock URL \url{https://doi.org/10.1007/s10589-005-3078-8}.

\bibitem[Fan and Yuan(2005)]{fan2005quadratic}
J.~Fan and Y.~Yuan.
\newblock On the quadratic convergence of the {Levenberg--Marquardt} method without nonsingularity assumption.
\newblock \emph{Computing}, 74\penalty0 (1):\penalty0 23--39, 2005.
\newblock \doi{10.1007/s00607-004-0083-1}.
\newblock URL \url{https://doi.org/10.1007/s00607-004-0083-1}.

\bibitem[Fischer et~al.(2010)Fischer, Shukla, and Wang]{fischer2010inexactness}
A.~Fischer, P.~Shukla, and M.~Wang.
\newblock On the inexactness level of robust {Levenberg--Marquardt} methods.
\newblock \emph{Optimization}, 59\penalty0 (2):\penalty0 273--287, 2010.
\newblock \doi{10.1080/02331930801951256}.
\newblock URL \url{https://doi.org/10.1080/02331930801951256}.

\bibitem[Heek et~al.(2020)Heek, Levskaya, Oliver, Ritter, Rondepierre, Steiner, and van {Z}ee]{flax2020github}
J.~Heek, A.~Levskaya, A.~Oliver, M.~Ritter, B.~Rondepierre, A.~Steiner, and M.~van {Z}ee.
\newblock {F}lax: A neural network library and ecosystem for {JAX}, 2020.
\newblock URL \url{https://github.com/google/flax}.

\bibitem[Kanzow et~al.(2004)Kanzow, Yamashita, and Fukushima]{kanzow2004levenberg}
C.~Kanzow, N.~Yamashita, and M.~Fukushima.
\newblock {Levenberg--Marquardt} methods with strong local convergence properties for solving nonlinear equations with convex constraints.
\newblock \emph{Journal of Computational and Applied Mathematics}, 172\penalty0 (2):\penalty0 375--397, 2004.
\newblock ISSN 0377-0427.
\newblock \doi{https://doi.org/10.1016/j.cam.2004.02.013}.
\newblock URL \url{https://doi.org/10.1016/j.cam.2004.02.013}.

\bibitem[Karimi et~al.(2016)Karimi, Nutini, and Schmidt]{karimi2016linear}
H.~Karimi, J.~Nutini, and M.~Schmidt.
\newblock Linear convergence of gradient and proximal-gradient methods under the {Polyak--\L ojasiewicz} condition.
\newblock In \emph{European Conference on Machine Learning and Knowledge Discovery in Databases}, pages 795--811. Springer, 2016.

\bibitem[Lan(2020)]{lan2020first}
G.~Lan.
\newblock \emph{First-order and Stochastic Optimization Methods for Machine Learning}.
\newblock Springer, Cham, 2020.
\newblock URL \url{https://doi.org/10.1007/978-3-030-39568-1}.

\bibitem[Levenberg(1944)]{levenberg1944method}
K.~Levenberg.
\newblock A method for the solution of certain non-linear problems in least squares.
\newblock \emph{Quarterly of Applied Aathematics}, 2\penalty0 (2):\penalty0 164--168, 1944.

\bibitem[Lewis and Wright(2016)]{lewis2016proximal}
A.~S. Lewis and S.~J. Wright.
\newblock A proximal method for composite minimization.
\newblock \emph{Mathematical Programming}, 158\penalty0 (1):\penalty0 501--546, 2016.
\newblock \doi{10.1007/s10107-015-0943-9}.
\newblock URL \url{https://doi.org/10.1007/s10107-015-0943-9}.

\bibitem[Li and Wang(2002)]{li2002convergence}
C.~Li and X.~Wang.
\newblock On convergence of the {Gauss--Newton} method for convex composite optimization.
\newblock \emph{Mathematical Programming}, 91\penalty0 (2):\penalty0 349--356, 2002.
\newblock \doi{10.1007/s101070100249}.
\newblock URL \url{https://doi.org/10.1007/s101070100249}.

\bibitem[Marquardt(1963)]{marquardt1963algorithm}
D.~W. Marquardt.
\newblock An algorithm for least-squares estimation of nonlinear parameters.
\newblock \emph{Journal of the Society for Industrial and Applied Mathematics}, 11\penalty0 (2):\penalty0 431--441, 1963.
\newblock \doi{10.1137/0111030}.
\newblock URL \url{https://doi.org/10.1137/0111030}.

\bibitem[Marumo et~al.(2023)Marumo, Okuno, and Takeda]{marumo2023majorization}
N.~Marumo, T.~Okuno, and A.~Takeda.
\newblock Majorization-minimization-based {Levenberg--Marquardt} method for constrained nonlinear least squares.
\newblock \emph{Computational Optimization and Applications}, 84\penalty0 (3):\penalty0 833--874, 2023.
\newblock \doi{10.1007/s10589-022-00447-y}.
\newblock URL \url{https://doi.org/10.1007/s10589-022-00447-y}.

\bibitem[Necoara et~al.(2019)Necoara, Nesterov, and Glineur]{necoara2019linear}
I.~Necoara, Y.~Nesterov, and F.~Glineur.
\newblock Linear convergence of first order methods for non-strongly convex optimization.
\newblock \emph{Mathematical Programming}, 175\penalty0 (1):\penalty0 69--107, 2019.
\newblock \doi{10.1007/s10107-018-1232-1}.
\newblock URL \url{https://doi.org/10.1007/s10107-018-1232-1}.

\bibitem[Nesterov(2007)]{nesterov2007modified}
Y.~Nesterov.
\newblock Modified {Gauss--Newton} scheme with worst case guarantees for global performance.
\newblock \emph{Optimization Methods and Software}, 22\penalty0 (3):\penalty0 469--483, 2007.
\newblock \doi{10.1080/08927020600643812}.
\newblock URL \url{https://doi.org/10.1080/08927020600643812}.

\bibitem[Nesterov(2018)]{nesterov2018lectures}
Y.~Nesterov.
\newblock \emph{Lectures on Convex Optimization}, volume 137.
\newblock Springer, Cham, 2018.
\newblock URL \url{https://doi.org/10.1007/978-3-319-91578-4}.

\bibitem[Nocedal and Wright(2006)]{nocedal2006numerical}
J.~Nocedal and S.~J. Wright.
\newblock \emph{Numerical Optimization}.
\newblock Springer, 2nd edition, 2006.
\newblock URL \url{https://doi.org/10.1007/978-0-387-40065-5}.

\bibitem[Rosenbrock(1960)]{rosenbrock1960automatic}
H.~H. Rosenbrock.
\newblock An automatic method for finding the greatest or least value of a function.
\newblock \emph{The Computer Journal}, 3\penalty0 (3):\penalty0 175--184, 01 1960.
\newblock ISSN 0010-4620.
\newblock \doi{10.1093/comjnl/3.3.175}.
\newblock URL \url{https://doi.org/10.1093/comjnl/3.3.175}.

\bibitem[Ueda and Yamashita(2010)]{ueda2010global}
K.~Ueda and N.~Yamashita.
\newblock On a global complexity bound of the {Levenberg--Marquardt} method.
\newblock \emph{Journal of Optimization Theory and Applications}, 147\penalty0 (3):\penalty0 443--453, 2010.
\newblock \doi{10.1007/s10957-010-9731-0}.
\newblock URL \url{https://doi.org/10.1007/s10957-010-9731-0}.

\bibitem[Wang and Fan(2021)]{wang2021convergence}
H.~Wang and J.~Fan.
\newblock Convergence properties of inexact {Levenberg--Marquardt} method under {H{\"o}lderian} local error bound.
\newblock \emph{Journal of Industrial \& Management Optimization}, 17\penalty0 (4):\penalty0 2265, 2021.
\newblock URL \url{https://doi.org/10.3934/jimo.2020068}.

\bibitem[Wiltschko and Johnson()]{jax2022autodiff}
A.~Wiltschko and M.~Johnson.
\newblock {The Autodiff Cookbook --- JAX documentation}.
\newblock URL \url{https://jax.readthedocs.io/en/latest/notebooks/autodiff_cookbook.html}.

\bibitem[Yamashita and Fukushima(2001)]{yamashita2001rate}
N.~Yamashita and M.~Fukushima.
\newblock On the rate of convergence of the {Levenberg--Marquardt} method.
\newblock In G.~Alefeld and X.~Chen, editors, \emph{Topics in Numerical Analysis}, pages 239--249, Vienna, 2001. Springer Vienna.
\newblock ISBN 978-3-7091-6217-0.
\newblock URL \url{https://doi.org/10.1007/978-3-7091-6217-0_18}.

\bibitem[Zhao and Fan(2016)]{zhao2016global}
R.~Zhao and J.~Fan.
\newblock Global complexity bound of the {Levenberg--Marquardt} method.
\newblock \emph{Optimization Methods and Software}, 31\penalty0 (4):\penalty0 805--814, 2016.
\newblock \doi{10.1080/10556788.2016.1179737}.
\newblock URL \url{https://doi.org/10.1080/10556788.2016.1179737}.

\end{thebibliography}

\appendix

\section{Lemmas for analysis}
\label{sec:lemmas}
We show several lemmas used for our analysis.
First, the following two lemmas may be standard in optimization.
\begin{lemma}
  Suppose that \cref{asm:hc} holds.
  Then, the following hold:
  \begin{alignat}{2}
    h(y) - h(x)
    &\leq
    \inner{\nabla h(z)}{y - x} + \frac{L_h}{2} \norm*{y - z}^2,
    &\quad&
    \forall x, y, z \in \R^n,
    \label{eq:lem_three-points}\\
    h(x) - h^*
    &\geq
    \frac{1}{2L_h} \norm*{\nabla h(x)}^2,
    &\quad&
    \forall x \in \R^n.
    \label{eq:lem_smoothness_obj_grad-norm_bound}
  \end{alignat}
\end{lemma}
\begin{proof}
  From \citep[Eq.~(2.1.9)]{nesterov2018lectures}, we have
  \begin{align}
    h(y) - h(z)
    \leq
    \inner{\nabla h(z)}{y - z} + \frac{L_h}{2} \norm*{y - z}^2,\quad
    \forall y, z \in \R^n.
    \label{eq:nesterov_219}
  \end{align}
  On the other hand, from \cref{asm:h_convex}, we have
  \begin{align}
    h(z) - h(x)
    \leq
    \inner{\nabla h(z)}{z - x},\quad
    \forall x, z \in \R^n.
  \end{align}
  Putting these inequalities together yields \cref{eq:lem_three-points}.
  \revise{Using \cref{eq:requirement_gast_hast} and} taking $y = z - \frac{1}{L_h} \nabla h(z)$ for \cref{eq:nesterov_219} yields
  \begin{align}
    h^* - h(z)
    \leq
    h(y) - h(z)
    \leq
    - \frac{1}{2 L_h} \norm*{\nabla h(z)}^2,\quad
    \forall z \in \R^n,
  \end{align}
  which is equivalent to \cref{eq:lem_smoothness_obj_grad-norm_bound}.
\end{proof}
\begin{lemma}
  \label{lem:norn_c_lin_diff_upperbound}
  Suppose that \cref{asm:hc} holds.
  Then, the following holds:
  \begin{align}
    \norm*{c(y) - c(x) - \nabla c(x) (y - x)}
    \leq
    \frac{L_c}{2} \norm*{y - x}^2,\quad
    \forall x, y \in \dom g.
    \label{eq:property_nabla-c_lip}
  \end{align}
\end{lemma}
\begin{proof}
  We can prove this lemma in the same way as \citep[Lemma 1.2.3]{nesterov2018lectures}.
\end{proof}

Next, the following lemma bounds the function value using the first-order optimality measures defined in \cref{eq:def_error}.
\begin{lemma}
  \label{lem:proj-grad_norm_lowerbound}
  Suppose that \cref{asm:hc} holds.
  Then, the following hold:
  \begin{alignat}{2}
    g(y) - g(x) + \inner{\nabla H(x)}{y - x}
    &\geq - \omega(x) \norm*{y - x},
    &\quad&
    \forall x, y \in \dom g,
    \label{eq:proj-grad_norm_lowerbound}\\
    \bar F_{k, \mu}(y) - \bar F_{k, \mu}(x)
    - \frac{\mu}{2} \norm*{y - x}^2
    &\geq
    - \bar{\omega}_{k, \mu}(x) \norm*{y - x},
    &\quad&
    \forall x, y \in \dom g.
    \label{eq:proj-grad_norm_lowerbound_sub}
  \end{alignat}
\end{lemma}
\begin{proof}
  For all $p \in \partial g(x)$, we have
  \begin{alignat}{2}
    g(y) - g(x) + \inner{\nabla H(x)}{y - x}
    \geq
    \inner{p + \nabla H(x)}{y - x}
    \geq
    - \norm*{p + \nabla H(x)} \norm*{y - x},
  \end{alignat}
  and hence
  \begin{alignat}{2}
    g(y) - g(x) + \inner{\nabla H(x)}{y - x}
    \geq
    - \min_{p \in \partial g(x)} \norm*{p + \nabla H(x)} \norm*{y - x}
    =
    - \omega(x) \norm*{y - x},
  \end{alignat}
  which completes the proof of \cref{eq:proj-grad_norm_lowerbound}.
  Similarly to \cref{eq:proj-grad_norm_lowerbound}, we have
  \begin{align}
    g(y) - g(x) + \inner{\nabla \bar H_{k, \mu}(x)}{y - x}
    \geq
    - \bar{\omega}_{k, \mu}(x) \norm*{y - x},
  \end{align}
  and on the other hand, we have
  \begin{align}
    \bar H_{k, \mu}(y) - \bar H_{k, \mu}(x)
    \geq
    \inner{\nabla \bar H_{k, \mu}(x)}{y - x}
    + \frac{\mu}{2} \norm*{y - x}^2
  \end{align}
  since $\bar H_{k, \mu}$ is strongly convex.
  Combining these bounds yields
  \begin{align}
    \bar F_{k, \mu}(y) - \bar F_{k, \mu}(x)
    &=
    g(y) - g(x) + \bar H_{k, \mu}(y) - \bar H_{k, \mu}(x)\\
    &\geq
    \frac{\mu}{2} \norm*{y - x}^2
    - \bar{\omega}_{k, \mu}(x) \norm*{y - x},
  \end{align}
  which completes the proof of \cref{eq:proj-grad_norm_lowerbound_sub}.
\end{proof}

The following lemma is used to bound the oracle complexity of our algorithm.
\begin{lemma}
  Let 
  \begin{align}
    P(z)
    \coloneqq 
    1 + \sqrt{z} \log z.
    \label{eq:def_A_P}
  \end{align}
  Then,
  \begin{enuminlem}
    \item
    \label{lem:increasing_Pz}
    $P(z)$ is increasing with respect to $z \geq 1$;
    \item
    \label{lem:concave_Pz}
    $P(z)$ is concave with respect to $z \geq 1$;
    \item
    \label{lem:increasing_zP}
    $\frac{1}{z} P(1 + z)$ is decreasing with respect to $z > 0$.
  \end{enuminlem}
\end{lemma}
\begin{proof}
  For $z \geq 1$, we have
  \begin{align}
    \frac{d}{dz} P(z)
    &=
    \frac{\log z + 2}{2 \sqrt{z}}
    > 0,\\
    \frac{d^2}{dz^2} P(z)
    &=
    - \frac{\log z}{4 z^{3/2}}
    < 0,
  \end{align}
  which implies \cref{lem:increasing_Pz,lem:concave_Pz}.
  To prove \cref{lem:increasing_zP}, we use the following inequality:
  \begin{align}
    \log (1 + z) \geq \frac{2z}{z+2},\quad
    \forall z \geq 0,
    \label{eq:logz_lowerbound}
  \end{align}
  which can be proved by using
  \begin{align}
    \frac{d}{dz}
    \prn*{
      \log (1 + z) - \frac{2z}{z+2}
    }
    = 
    \frac{z^2}{(z + 1) (z + 2)^2}
    \geq 0.
  \end{align}
  For $z > 0$, we have
  \begin{alignat}{2}
    \frac{d}{dz} \prn*{ \frac{1}{z} P \prn*{ 1 + z } }
    &=
    \frac{2 z - (z + 2) \log(1 + z)}{2 z^2 \sqrt{1 + z}} - \frac{1}{z^2}\\
    &\leq
    - \frac{1}{z^2}
    <
    0
    &\quad&\by{\cref{eq:logz_lowerbound}},
  \end{alignat}
  which implies \cref{lem:increasing_zP}.
\end{proof}

\section{Proofs for Section~\ref{sec:algorithm}}
\subsection{Proof of Lemma~\ref{lem:smooth_subproblem}}
\label{sec:proof_smooth_subproblem}
\begin{proof}
  As
  \begin{align}
    \nabla \bar H_{k, \mu} (x)
    =
    \nabla c(x_k)^\top \nabla h \prn[\big]{
      c(x_k) + \nabla c(x_k) (x - x_k)
    }
    + \mu (x - x_k),
  \end{align}
  we have
  \begin{align}
    &\mathInd
    \norm*{
      \nabla \bar H_{k, \mu} (y) - \nabla \bar H_{k, \mu} (x)
    }\\
    &=
    \norm*{
      \nabla c(x_k)^\top \prn[\Big]{
        \nabla h \prn[\big]{
          c(x_k) + \nabla c(x_k) (y - x_k)
        }
        - 
        \nabla h \prn[\big]{
          c(x_k) + \nabla c(x_k) (x - x_k)
        }
      }
      + \mu (y - x)
    }\\
    &\leq
    \normop*{\nabla c(x_k)}
    \norm*{
      \nabla h \prn[\big]{
        c(x_k) + \nabla c(x_k) (y - x_k)
      }
      - 
      \nabla h \prn[\big]{
        c(x_k) + \nabla c(x_k) (x - x_k)
      }
    }
    + \mu \norm{y - x}\\
    &\leq
    L_h
    \normop*{\nabla c(x_k)}
    \norm*{
      \nabla c(x_k) (y - x)
    }
    + \mu \norm{y - x}
    \qquad\text{(by \cref{asm:h_smooth})}\\
    &\leq
    L_h
    \normop*{\nabla c(x_k)}^2
    \norm{y - x}
    + \mu \norm{y - x}\\
    &\leq
    (L_h \sigma^2 + \mu)
    \norm{y - x}
    \qquad\text{(by \cref{asm:nablac_bound})},
  \end{align}
  which completes the proof.
\end{proof}

\subsection{Proof of Lemma~\ref{lem:apg_iteration}}
\label{sec:proof_apg_subproblem}
Before the proof, we show some standard inequalities on \cref{alg:apg}, which will be also used for \cref{thm:oracle_complexity_global}.

The parameter $\eta$ is initialized to $\bar \alpha \mu \ (\leq \bar \alpha \bar L)$ in \cref{alg-line-apg:initialize_etab0}, and the inequality in \cref{alg-line-apg:backtracking} must hold if $\eta \geq \bar L$.
This implies that the parameter $\eta$ in \cref{alg:apg} always satisfies
\begin{align}
  \eta \leq \bar \alpha \bar L.
  \label{eq:eta_upperbound}
\end{align}
Using this bound, we can bound the left-hand side of the termination condition in \cref{alg-line-apg:termination_cond} as
\begin{alignat}{2}
  &\mathInd
  \norm*{\nabla \bar H_{k, \mu}(\bar x_{t+1}) - \nabla \bar H_{k, \mu}(y_t) - \eta (\bar x_{t+1} - y_t)}\\
  &\leq
  \norm*{\nabla \bar H_{k, \mu}(\bar x_{t+1}) - \nabla \bar H_{k, \mu}(y_t)}
  + \eta \norm*{\bar x_{t+1} - y_t}\\
  &\leq
  \bar L \norm*{\bar x_{t+1} - y_t}
  + \eta \norm*{\bar x_{t+1} - y_t}
  &\quad&\by{\cref{lem:smooth_subproblem}}\\
  &\leq
  (1 + \bar \alpha) \bar L \norm*{\bar x_{t+1} - y_t}
  &\quad&\by{\cref{eq:eta_upperbound}}.
  \label{eq:norm_bxtyt_upperbound_norm_xbt1yt}
\end{alignat}

Now, we prove \cref{lem:apg_iteration}.

\begin{proof}[Proof of \cref{lem:apg_iteration}]
  Let $\bar x^*$ denote the optimal solution to subproblem~\cref{eq:subproblem}.
  At \cref{alg-line-apg:termination_cond} of \cref{alg:apg}, we have
  \begin{alignat}{2}
    &\mathInd
    \norm*{\nabla \bar H_{k, \mu}(\bar x_{t+1}) - \nabla \bar H_{k, \mu}(y_t) - \eta (\bar x_{t+1} - y_t)}\\
    &\leq
    (1 + \bar \alpha) \bar L \norm*{\bar x_{t+1} - y_t}
    &\quad&\by{\cref{eq:norm_bxtyt_upperbound_norm_xbt1yt}}\\
    &=
    (1 + \bar \alpha) \bar L \norm*{\bar x_{t+1} - (1 - \tau) \bar x_t - \tau z_t}
    &\quad&\by{\cref{alg-line-apg:set_y}}\\
    &\leq
    (1 + \bar \alpha) \bar L \prn*{
      \norm*{\bar x_{t+1} - \bar x^*}
      + \norm*{(1 - \tau) (\bar x_t - \bar x^*)}
      + \norm*{\tau (z_t - \bar x^*)}
    }\\
    &\leq
    (1 + \bar \alpha) \bar L
    \prn*{
      \norm*{\bar x_{t+1} - \bar x^*}
      + \norm*{\bar x_t - \bar x^*}
      + \norm*{z_t - \bar x^*}
    }
    &\quad&\by{$0 \leq \tau \leq 1$}.
    \label{eq:norm_bxtyt_upperbound}
  \end{alignat}
  Here, we can check $0 \leq \tau \leq 1$ from the definitions of $b_{t+1}$ and $\tau$.
  To bound \cref{eq:norm_bxtyt_upperbound}, we will show several inequalities.

  With a slight modification of \citep[Corollary~B.3]{daspremont2021acceleration}, we obtain
  \begin{align}
    \bar F_{k, \mu}(\bar x_t) - \bar F_{k, \mu}(\bar x^*)
    + \frac{\mu}{2} \norm{z_t - \bar x^*}^2
    &\leq
    \bar \alpha \bar \kappa \mu 
    \prn*{
      1 - \frac{1}{\sqrt{\bar \alpha \bar \kappa}}
    }^t
    \norm*{\bar x_0 - \bar x^*}^2
    \label{eq:apg_bound}
  \end{align}
  for all $t \geq 1$.
  On the other hand, from the strong convexity of $\bar F_{k, \mu}$, we have
  \begin{align}  
    \bar F_{k, \mu}(\bar x_t) - \bar F_{k, \mu}(\bar x^*)
    \geq
    \frac{\mu}{2} \norm*{\bar x_t - \bar x^*}^2.
    \label{eq:strong_convexity_Hbar_bxt}
  \end{align}
  Combining these two bounds yields
  \begin{alignat}{2}
    &\mathInd
    \norm*{\bar x_t - \bar x^*} + \norm*{z_t - \bar x^*}\\
    &\leq
    \sqrt{
      2 \norm*{\bar x_t - \bar x^*}^2
      + 2 \norm*{z_t - \bar x^*}^2
    }
    &\quad&\text{(since $a + b \leq \sqrt{2 (a^2 + b^2)}$ for all $a, b \in \R$)}\\
    &\leq
    \sqrt{
      \frac{4}{\mu} \prn*{ \bar F_{k, \mu}(\bar x_t) - \bar F_{k, \mu}(\bar x^*) + \frac{\mu}{2} \norm*{z_t - \bar x^*}^2 }
    }
    &\quad&\by{\cref{eq:strong_convexity_Hbar_bxt}}\\
    &\leq
    p_t
    \norm*{\bar x_0 - \bar x^*}
    &\quad&\by{\cref{eq:apg_bound}},
    \label{eq:bxtbx_ztbx_norm_upperbound}
  \end{alignat}
  where
  \begin{align}
    p_t
    \coloneqq    
    2 \sqrt{\bar \alpha \bar \kappa}
    \prn*{
      1 - \frac{1}{\sqrt{\bar \alpha \bar \kappa}}
    }^{t/2}
    \leq
    2 \sqrt{\bar \alpha \bar \kappa}
    \exp
    \prn*{
      - \frac{t}{2 \sqrt{\bar \alpha \bar \kappa}}
    }.
  \end{align}
  By setting $t$ to
  \begin{align}
    t
    \geq
    2 \sqrt{\bar \alpha \bar \kappa}
    \log \prn*{
      2 \sqrt{\bar \alpha \bar \kappa}
      \prn*{
        1 + \frac{2 (1 + \bar \alpha) \bar \kappa}{\theta}
      }
    }
  \end{align}
  as in \cref{eq:inner_iteration_complexity}, we have 
  \begin{align}
    p_t
    \leq
    \prn*{
      1 + \frac{2 (1 + \bar \alpha) \bar \kappa}{\theta}
    }^{-1}
    < 1.
    \label{eq:pt_upperbound}
  \end{align}
  From \cref{eq:bxtbx_ztbx_norm_upperbound}, we also have
  \begin{align}
    \norm*{\bar x_{t+1} - \bar x^*}
    \leq 
    p_{t+1} \norm*{\bar x_0 - \bar x^*}
    \leq 
    p_t \norm*{\bar x_0 - \bar x^*}.
    \label{eq:bxtbx_norm_upperbound}
  \end{align}
  This inequality implies
  \begin{align}
    \norm*{\bar x_0 - \bar x^*}
    &\leq
    \norm*{\bar x_{t+1} - \bar x_0}
    + \norm*{\bar x_{t+1} - \bar x^*}
    \leq
    \norm*{\bar x_{t+1} - \bar x_0}
    + p_t \norm*{\bar x_0 - \bar x^*},
  \end{align}
  and hence
  \begin{align}
    \norm*{\bar x_0 - \bar x^*}
    &\leq
    \frac{1}{1 - p_t}
    \norm*{\bar x_{t+1} - \bar x_0}.
    \label{eq:norm_xkbxs_upperbound}
  \end{align}

  Now, we can bound \cref{eq:norm_bxtyt_upperbound} as follows:
  \begin{alignat}{2}
    &\mathInd
    \norm*{\nabla \bar H_{k, \mu}(\bar x_{t+1}) - \nabla \bar H_{k, \mu}(y_t) - \eta (\bar x_{t+1} - y_t)}\\
    &\leq
    (1 + \bar \alpha) \bar L
    \prn*{
      \norm*{\bar x_{t+1} - \bar x^*}
      + \norm*{\bar x_t - \bar x^*}
      + \norm*{z_t - \bar x^*}
    }\\
    &\leq
    2 p_t
    (1 + \bar \alpha) \bar L
    \norm*{\bar x_0 - \bar x^*}
    &\quad&\by{\cref{eq:bxtbx_ztbx_norm_upperbound,eq:bxtbx_norm_upperbound}}\\
    &\leq
    \frac{2p_t}{1 - p_t}
    (1 + \bar \alpha) \bar L
    \norm*{\bar x_{t+1} - \bar x_0}
    &\quad&\by{\cref{eq:norm_xkbxs_upperbound}}\\
    &\leq
    \theta \mu
    \norm*{\bar x_{t+1} - \bar x_0}
    &\quad&\by{\cref{eq:pt_upperbound} and definition \cref{eq:def_kappa} of $\bar \kappa$},
  \end{alignat}
  which completes the proof.
\end{proof}

\section{Proofs for Section~\ref{sec:global_convergence}}
\label{sec:proof_global}
Under \cref{asm:hc}, we have
\begin{align}
  \norm*{\nabla h (c(x_k))}
  \leq
  \sqrt{2 L_h \Delta_k},\quad
  \forall k \in \N
  \label{eq:norm_gradhcxk_upperbound}
\end{align}
as
\begin{alignat}{2}
  \frac{1}{2 L_h} \norm*{\nabla h (c(x_k))}^2
  &\leq
  h(c(x_k)) - h^*
  &\quad&\by{\cref{eq:lem_smoothness_obj_grad-norm_bound}}\\
  &\leq
  g(x_k) + h(c(x_k)) - (g^* + h^*)
  &\quad&\by{\revise{\cref{eq:requirement_gast_hast}}}\\
  &=
  \Delta_k
  &\quad&\by{definition \cref{eq:def_Deltak_Fast} of $\Delta_k$}.
\end{alignat}
\revise{
Similarly, we have 
\begin{align}
  \norm*{\nabla h(c(x_k) + \nabla c(x_k) (x - x_k))}
  \leq
  \sqrt{2 L_h \prn*{\bar F_{k, \mu_k}(x) - (g^* + h^*)} },\quad
  \forall k \in \N,\ 
  x \in \R^d
  \label{eq:norm_gradhw_upperbound_2}
\end{align}
as
\begin{alignat}{2}
  \frac{1}{2 L_h} \norm*{\nabla h(c(x_k) + \nabla c(x_k) (x - x_k))}^2
  &\leq
  h(c(x_k) + \nabla c(x_k) (x - x_k)) - h^*\\
  &\leq
  g(x_k) + h(c(x_k) + \nabla c(x_k) (x - x_k)) - (g^* + h^*)\\
  &\leq
  \bar F_{k, \mu_k}(x) - (g^* + h^*).
\end{alignat}
}
We will use \revise{inequalities~\cref{eq:norm_gradhcxk_upperbound,eq:norm_gradhw_upperbound_2}} to prove \cref{lem:gradMF_norm_upperbound,prop:local_convergence}.

\subsection{Proof of Lemma~\ref{lem:gradMF_norm_upperbound}}
\label{sec:proof_gradMF_norm}
\begin{proof}
  To simplify notation, let
  \begin{align}
    u \coloneqq x - x_k,\quad
    w \coloneqq c(x_k) + \nabla c(x_k) u.
  \end{align}
  From definition \cref{eq:def_error} of $\bar{\omega}_{k, \mu}$, we can take $p \in \partial g(x)$ such that
  \begin{align}
    \norm*{p + \nabla \bar H_{k, \mu_k}(x)}
    =
    \bar{\omega}_{k, \mu_k}(x).
    \label{eq:nablaMg_norm_bound}
  \end{align}
  Then, from definition \cref{eq:def_error} of $\omega$, we also have
  \begin{alignat}{2}
    \omega(x)
    \leq
    \norm*{p + \nabla H(x)}
    &\leq
    \norm*{p + \nabla \bar H_{k, \mu_k}(x)} + \norm*{\nabla \bar H_{k, \mu_k}(x) - \nabla H(x)}\\
    &=
    \bar{\omega}_{k, \mu_k}(x)
    + \norm*{\nabla \bar H_{k, \mu_k}(x) - \nabla H(x)}.
    \label{eq:DeltaXF_upperbound}
  \end{alignat}
  We will bound $\norm*{\nabla \bar H_{k, \mu_k}(x) - \nabla H(x)}$ in \cref{eq:DeltaXF_upperbound}.
  As
  \begin{align}
    &\mathInd
    \nabla \bar H_{k, \mu_k}(x) - \nabla H(x)\\
    &=
    \nabla c(x_k)^\top \nabla h(w)
    - \nabla c(x)^\top \nabla h (c(x))
    + \mu_k u\\
    &=
    \nabla c(x_k)^\top \prn*{ \nabla h(w) - \nabla h (c(x)) }
    - \prn*{ \nabla c(x) - \nabla c(x_k) }^\top \nabla h (c(x))
    + \mu_k u,
  \end{align}
  we have
  \begin{align}
    &\mathInd
    \norm*{\nabla \bar H_{k, \mu_k}(x) - \nabla H(x)}\\
    &\leq
    \normop*{\nabla c(x_k)} \norm*{ \nabla h(w) - \nabla h (c(x)) }
    + \normop*{ \nabla c(x) - \nabla c(x_k) } \norm*{ \nabla h (c(x)) }
    + \mu_k \norm*{u}.
    \label{eq:mkf_norm_upperbound}
  \end{align}
  Each term of \cref{eq:mkf_norm_upperbound} can be bounded as follows:
  first,
  \begin{alignat}{2}
    \normop*{ \nabla c(x_k) }
    &\leq
    \sigma
    &\quad&\by{\cref{asm:nablac_bound}};
  \end{alignat}
  second,
  \begin{alignat}{2}
    \norm*{ \nabla h(w) - \nabla h (c(x)) }
    &\leq
    L_h \norm*{ w - c(x) }
    &\quad&\by{\cref{asm:h_smooth}}\\
    &=
    L_h \norm*{ c(x_k) + \nabla c(x_k) (x - x_k) - c(x) }\\
    &\leq
    \frac{L_cL_h}{2} \norm*{u}^2
    &\quad&\by{\cref{eq:property_nabla-c_lip}};
    \label{eq:norm_diff_gradhw_gradhcx_upperbound}
  \end{alignat}
  third,
  \begin{alignat}{2}
    \normop*{ \nabla c(x) - \nabla c(x_k) }
    &\leq
    L_c \norm*{u}
    &\quad&\by{\cref{asm:nabla-c_lip}}.
  \end{alignat}
  Plugging these bounds into \cref{eq:mkf_norm_upperbound}, we obtain
  \begin{align}
    \norm*{\nabla \bar H_{k, \mu_k}(x) - \nabla H(x)}
    &\leq
    \sigma \prn*{ \frac{L_cL_h}{2} \norm*{u}^2 }
    + \prn*{L_c \norm*{u}} \revise{ \norm*{ \nabla h (c(x)) } }
    + \mu_k \norm*{u}.
  \end{align}
  Again, plugging this bound into \cref{eq:DeltaXF_upperbound} yields
  \revise{
  \begin{align}
    \omega(x)
    &\leq
    \bar{\omega}_{k, \mu_k}(x)
    + \frac{L_c L_h \sigma}{2} \norm*{u}^2
    + L_c \norm*{u} \norm*{\nabla h (c(x))}
    + \mu_k \norm*{u}.
    \label{eq:omega_x_upperbound_general}
  \end{align}
  }
  
  \revise{
  We first show \cref{eq:gradMF_norm_upperbound}.
  Setting $x = x_{k+1} \in S_{k, \mu_k}$ in \cref{eq:omega_x_upperbound_general} gives
  \begin{align}
    \omega(x_{k+1})
    &\leq
    \prn[\Big]{
      (1 + \theta) \mu_k
      + L_c \norm*{\nabla h (c(x_{k+1}))}
    } \norm*{x_{k+1} - x_k}
    + \frac{L_c L_h \sigma}{2} \norm*{x_{k+1} - x_k}^2
  \end{align}
  since $\bar \omega_{k, \mu_k}(x_{k+1}) \leq \theta \mu_k \norm*{x_{k+1} - x_k}$ by the definition of $S_{k, \mu_k}$.
  From \cref{eq:norm_gradhcxk_upperbound,eq:Delta_decreasing}, we have $\norm*{\nabla h (c(x_{k+1}))} \leq \sqrt{2 L_h \Delta_{k+1}} \leq \sqrt{2 L_h \Delta_k}$ and hence obtain \cref{eq:gradMF_norm_upperbound}.
  }

  \revise{
  Next, we prove \cref{eq:gradMF_norm_upperbound_tighter}.
  We have
  \begin{alignat}{2}
    \norm*{\nabla h (c(x))}
    &\leq
    \norm*{\nabla h (w)}
    + \norm*{\nabla h (w) - \nabla h (c(x))}\\
    &\leq
    \norm*{\nabla h (w)}
    + \frac{L_cL_h}{2} \norm*{u}^2
    &\quad&\by{\cref{eq:norm_diff_gradhw_gradhcx_upperbound}}\\
    &\leq
    \sqrt{2 L_h \prn*{\bar F_{k, \mu_k}(x) - (g^* + h^*)} }
    + \frac{L_cL_h}{2} \norm*{u}^2
    &\quad&\by{\cref{eq:norm_gradhw_upperbound_2}}\\
    &\leq
    \sqrt{2 L_h \Delta_k}
    + \frac{L_cL_h}{2} \norm*{u}^2
    &\quad&\by{$\bar F_{k, \mu_k}(x) \leq F(x_k)$}.
  \end{alignat}
  Plugging this bound into \cref{eq:omega_x_upperbound_general} yields
  \begin{align}
    \omega(x)
    &\leq
    \bar{\omega}_{k, \mu_k}(x)
    + \frac{L_c L_h \sigma}{2} \norm*{u}^2
    + L_c \norm*{u}
    \prn*{
      \sqrt{2 L_h \Delta_k}
      + \frac{L_cL_h}{2} \norm*{u}^2
    }
    + \mu_k \norm*{u},
  \end{align}
  and rearranging the terms completes the proof of \cref{eq:gradMF_norm_upperbound_tighter}.
  }
\end{proof}

\subsection{Proof of Theorem~\ref{thm:complexity_global}}
\label{sec:proof_iteration_oracle}
To prove \cref{thm:complexity_global}, we introduce some notations.
Let
\begin{align}
  B_1
  &\coloneqq
  \frac{8}{1 - \theta}
  \prn*{
    (1 + \theta) \sqrt{\rho_{\max}}
    + \frac{L_c \sqrt{2 L_h}}{\sqrt{\rho_{\min}}}
  }^2
  =
  \revise{
    O \prn*{
      \rho_{\max}
      + \frac{L_c^2 L_h}{\rho_{\min}}
    }
  },
  \label{eq:def_B1}\\
  B_2
  &\coloneqq
  \frac{4 L_c L_h \sigma}{(1 - \theta) \rho_{\min}},
  \label{eq:def_B2}\\
  B_3
  &\coloneqq
  \frac{L_h \sigma^2}{\rho_{\min}},
  \label{eq:def_B3}\\
  \revise{
  B_4}
  &
  \revise{
  {}\coloneqq
  1 + 
  \max \set*{
    \ceil*{
      \log_\alpha \prn*{
        \frac{L_c \sqrt{2 L_h}}{(1 - \theta) \rho_{\min}}
      }
    }, \, 
    0
  }
  =
  O \prn*{
    1 + \log \prn*{ \frac{L_c \sqrt{L_h}}{\rho_{\min}} }
  },
  }
  \label{eq:def_B4}
\end{align}
\revise{where we have used $(a+b)^2 \leq 2(a^2 + b^2)$ for $B_1$.}
Let $N$ be the least $k$ satisfying $\omega(x_k) \leq \epsilon$, and let $V \coloneqq \set{0,1,\dots,N-1}$.
Then, we have
\begin{align}
  \omega(x_k)
  > \epsilon,\quad
  \forall k \in V.
  \label{eq:error_eps_lowerbound}
\end{align}
Let $S \subseteq V \setminus \set{N-1}$ and $T \subseteq V \setminus \set{N-1}$ be the sets of $k$ such that \cref{eq:optimality_upperbound1,eq:optimality_upperbound2} hold, respectively.
From \cref{lem:optimality_upperbound}, we have
\begin{align}
  S \cup T = V \setminus \set{N-1},
  \label{eq:ST_V}
\end{align}
and hence, $N \leq |S| + |T| + 1$.

Now, we prove \cref{thm:iteration_complexity_global}.
\begin{proof}
  We will derive upper bounds on $|S|$ and $|T|$.
  We have
  \begin{alignat}{2}
    \sum_{k \in S} \frac{1}{\sqrt{\Delta_k}} 
    &\leq
    B_1
    \sum_{k \in S}
    \frac{\Delta_k - \Delta_{k+1}}{\omega(x_{k+1})^2}
    &\quad&\by{\cref{eq:optimality_upperbound1}}\\
    &\leq
    B_1
    \sum_{k \in S}
    \frac{\Delta_k - \Delta_{k+1}}{\epsilon^2}
    &\quad&\by{\cref{eq:error_eps_lowerbound}}\\
    &\leq
    B_1
    \sum_{k \in V}
    \frac{\Delta_k - \Delta_{k+1}}{\epsilon^2}
    &\quad&\by{$S \subseteq V$ and \cref{eq:Delta_decreasing}}\\
    &=
    \frac{B_1}{\epsilon^2}
    \prn*{ F(x_0) - F(x_N) }
    \leq
    \frac{B_1}{\epsilon^2}
    \prn*{ F(x_0) - F^* },
    \label{eq:invsum_lam_upperbound}
  \end{alignat}
  and this bound implies
  \begin{alignat}{2}
    |S|
    &=
    \sum_{k \in S} 1
    \leq
    \sum_{k \in S} \sqrt{ \frac{\Delta_0}{\Delta_k} }
    \leq
    \frac{B_1 \sqrt{\Delta_0}}{\epsilon^2}
    \prn*{ F(x_0) - F^* }.
    \label{eq:S_size_upperbound}
  \end{alignat}
  As with $|S|$, it follows from \cref{eq:optimality_upperbound2,eq:error_eps_lowerbound} that
  \begin{alignat}{2}
    |T|
    \leq
    \sum_{k \in T} 1
    \leq
    \sum_{k \in T}
    \frac{B_2}{\omega(x_{k+1})}
    \prn*{\sqrt{\Delta_k} - \sqrt{\Delta_{k+1}}}
    \leq
    \sum_{k \in T}
    \frac{B_2}{\epsilon}
    \prn*{\sqrt{\Delta_k} - \sqrt{\Delta_{k+1}}}
    \leq
    \frac{B_2 \sqrt{\Delta_0}}{\epsilon}.
    \label{eq:T_size_upperbound}
  \end{alignat}
  Putting \cref{eq:S_size_upperbound,eq:T_size_upperbound} together yields
  \begin{alignat}{2}
    |S| + |T| + 1
    &\leq
    \frac{B_1 \sqrt{\Delta_0}}{\epsilon^2}
    (F(x_0) - F^*)
    + \frac{B_2 \sqrt{\Delta_0}}{\epsilon}
    + 1\\
    &
    \revise{
    =
    O \prn*{
      \frac{B_1 \sqrt{\Delta_0}}{\epsilon^2}
      (F(x_0) - F^*)
    }
    }\\
    &
    \revise{
    =
    O \prn*{
      \frac{\sqrt{\Delta_0}}{\epsilon^2}
      (F(x_0) - F^*)
      \prn*{
        \rho_{\max}
        + \frac{L_c^2 L_h}{\rho_{\min}}
      }
    }
    }
    &\quad&\by[,]{definition \cref{eq:def_B1} of $B_1$}
  \end{alignat}
  which completes the proof \revise{of \cref{eq:iteration_complexity_general}}.
  \revise{
  If we set $\rho_{\min} = \Theta \prn*{L_c \sqrt{L_h}}$, definition~\cref{eq:def_rhomax} of $\rho_{\max}$ gives $\rho_{\max} = \Theta \prn*{L_c \sqrt{L_h}}$.
  We thus obtain \cref{eq:iteration_complexity_simple}.
  }
\end{proof}

Next, we prove \cref{thm:oracle_complexity_global}.
\begin{proof}
  \revise{
  \Cref{prop:rhomax} shows that subproblems are solved $B_4$ times per iteration, where $B_4$ is defined by \cref{eq:def_B4}.}
  On the other hand, the parameter $\mu$ in \cref{alg:proposed_backtracking} satisfies $\mu \geq \rho_{\min} \sqrt{\Delta_k}$ throughout the $k$-th iteration.
  Thus, from \cref{prop:complexity_per_iteration}, the oracle complexity for the $k$-th iteration is
  \begin{align}
    O \prn*{
      \revise{B_4}
      P \prn*{
        1 + \frac{L_h \sigma^2}{\rho_{\min} \sqrt{\Delta_k}}
      }
    }
    =
    O \prn*{
      \revise{B_4}
      P \prn*{
        1 + \frac{B_3}{\sqrt{\Delta_k}}
      }
    },
  \end{align}
  where $P$ and $B_3$ are defined by \cref{eq:def_A_P,eq:def_B3}, and we used \cref{lem:increasing_Pz}.

  First, we will bound the complexity for the first $N-1$ iterations, i.e.,
  \begin{align}
    O \prn*{
      \revise{B_4}
      \sum_{k \in V \setminus \set{N-1}}
      P \prn*{
        1 + \frac{B_3}{\sqrt{\Delta_k}}
      }
    }.
    \label{eq:inner_complexity_bound1}
  \end{align}
  Using \cref{eq:ST_V}, we decompose this bound as
  \begin{alignat}{2}
    \sum_{k \in V \setminus \set{N-1}}
    P \prn*{
      1 + \frac{B_3}{\sqrt{\Delta_k}}
    }
    \leq
    \sum_{k \in S} P \prn*{ 1 + \frac{B_3}{\sqrt{\Delta_k}} }
    + \sum_{k \in T} P \prn*{ 1 + \frac{B_3}{\sqrt{\Delta_k}} }.
    \label{eq:complexity_sum_decomposition}
  \end{alignat}
  Each term can be bounded as follows:
  \begin{alignat}{2}
    &\mathInd
    \sum_{k \in S}
    P \prn*{
      1 + \frac{B_3}{\sqrt{\Delta_k}} 
    }\\
    &\leq
    |S| P \prn*{
      1 + \frac{B_3}{|S|} 
      \sum_{k \in S} \frac{1}{\sqrt{\Delta_k}}
    }
    &\quad&\by{\cref{lem:concave_Pz}\\and Jensen's inequality}\\
    &\leq
    |S| P \prn*{
      1
      + \frac{B_1 B_3}{|S| \epsilon^2}
      \prn*{ F(x_0) - F^* }
    }
    &\quad&\by{\cref{lem:increasing_Pz,eq:invsum_lam_upperbound}}\\
    &\leq
    \frac{B_1 \sqrt{\Delta_0}}{\epsilon^2}
    \prn*{ F(x_0) - F^* }
    P \prn*{ 1 + \frac{B_3}{\sqrt{\Delta_0}} }
    &\quad&\by{\cref{lem:increasing_zP,eq:S_size_upperbound}}
    \label{eq:invsum_S_upperbound}
  \end{alignat}
  and
  \begin{alignat}{2}
    \sum_{k \in T}
    P \prn*{ 1 + \frac{B_3}{\sqrt{\Delta_k}} }
    &\leq
    \sum_{k \in T}
    P \prn*{ 1 + \frac{B_2 B_3}{\omega(x_{k+1})} }
    &\quad&\by{\cref{lem:increasing_Pz,eq:optimality_upperbound2}}\\
    &\leq
    |T|
    P \prn*{ 1 + \frac{B_2 B_3}{\epsilon} }
    &\quad&\by{\cref{lem:increasing_Pz,eq:error_eps_lowerbound}}\\
    &\leq
    \frac{B_2 \sqrt{\Delta_0}}{\epsilon}
    P \prn*{ 1 + \frac{B_2 B_3}{\epsilon} }
    &\quad&\by{\cref{eq:T_size_upperbound}}\\
    &=
    O \prn*{\epsilon^{-3/2} \log \epsilon^{-1}}
    &\quad&\by{definition \cref{eq:def_A_P} of $P$}.
    \label{eq:invsum_T_upperbound}
  \end{alignat}
  \revise{From \cref{eq:complexity_sum_decomposition,eq:invsum_S_upperbound,eq:invsum_T_upperbound}, we obtain the following upper bound on}
  \revise{\cref{eq:inner_complexity_bound1}}:
  \revise{
  \begin{align}
    &\mathInd
    O \prn*{
      \frac{B_1 B_4 \sqrt{\Delta_0}}{\epsilon^2}
      \prn*{ F(x_0) - F^* }
      P \prn*{ 1 + \frac{B_3}{\sqrt{\Delta_0}}}
    }
    + O \prn*{ B_4 \epsilon^{-3/2} \log \epsilon^{-1}}\\
    &=
    O \prn*{
      \frac{B_1 B_4 \sqrt{\Delta_0}}{\epsilon^2}
      \prn*{ F(x_0) - F^* }
      P \prn*{ 1 + \frac{B_3}{\sqrt{\Delta_0}} }
    }\\
    &=
    O \prn*{
      \frac{B_1 B_4 \sqrt{\Delta_0}}{\epsilon^2}
      \prn*{ F(x_0) - F^* }
      \prn*{
        1 + \sqrt{1 + \frac{B_3}{\sqrt{\Delta_0}}}
        \log \prn*{ 1 + \frac{B_3}{\sqrt{\Delta_0}} }
      }
    }\\
    &=
    O \Bigg(
      \frac{\sqrt{\Delta_0}}{\epsilon^2}
      \prn*{ F(x_0) - F^* }
      \prn*{
        \rho_{\max}
        + \frac{L_c^2 L_h}{\rho_{\min}}
      }
      \prn*{
        1 + \log \prn*{ \frac{L_c \sqrt{L_h}}{\rho_{\min}} }
      }\\
      &\qquad \quad \times
      \prn*{
        1 + \sqrt{1 + \frac{L_h \sigma^2}{\rho_{\min} \sqrt{\Delta_0}}}
        \log \prn*{
          1 + \frac{L_h \sigma^2}{\rho_{\min} \sqrt{\Delta_0}}
        }
      }
    \Bigg),
    \label{eq:complexity_except_final}
  \end{align}
  where
  we have used $O \prn[\big]{\epsilon^{-3/2} \log \epsilon^{-1}} = o(\epsilon^{-2})$ and definitions \cref{eq:def_A_P,eq:def_B1,eq:def_B3,eq:def_B4} of $P$, $B_1$, $B_3$, and $B_4$.
  }


  Next, we will bound the complexity for the last iteration, i.e., the iteration with $k = N-1$.
  As we obtained \cref{eq:inner_complexity_bound1} from \cref{prop:complexity_per_iteration}, the complexity for the last iteration is upper-bounded by
  \begin{alignat}{2}
    O
    \prn*{
      \revise{B_4}
      P \prn*{ 1 + \frac{B_3}{\sqrt{\Delta_{N-1}}} }
    }.
  \end{alignat}
  If $\Delta_{N-1} \geq \epsilon^6$, then 
  \begin{alignat}{2}
    \revise{B_4}
    P \prn*{ 1 + \frac{B_3}{\sqrt{\Delta_{N-1}}} }
    =
    O \prn*{
      \epsilon^{-3/2} \log \epsilon^{-1} 
    }
  \end{alignat}
  from definition \cref{eq:def_A_P} of $P$.
  Thus, the overall oracle complexity is equal to \cref{eq:complexity_except_final}.

  Below, assuming $\Delta_{N-1} < \epsilon^6$, we will show that the point $\bar x_1$ computed by \cref{alg:apg} for $k = N-1$ satisfies $\omega(\bar x_1) \leq \epsilon$ when $\epsilon > 0$ is sufficiently small.
  To simplify notation, let $k \coloneqq N-1$.
  We show several inequalities on \cref{alg:apg}:
  first,
  \begin{alignat}{2}
    \bar \omega_{k, \mu_k}(\bar x_1)
    &\leq
    \norm*{\nabla \bar H_{k, \mu}(\bar x_1) - \nabla \bar H_{k, \mu}(y_0) - \eta (\bar x_1 - y_0)}
    &\quad&\by{\cref{eq:baromega_upperbound}}\\
    &\leq
    (1 + \bar \alpha) \revise{(\mu_k + L_h \sigma^2)} \norm*{\bar x_1 - y_0}
    &\quad&\by{\cref{eq:norm_bxtyt_upperbound_norm_xbt1yt,eq:def_Lbar}}\\
    &=
    (1 + \bar \alpha) \revise{(\mu_k + L_h \sigma^2)} \norm*{\bar x_1 - x_k};
    \label{eq:bokbx1_upperbound}
  \end{alignat}
  second,
  \begin{alignat}{2}
    g(\bar x_1)
    - g(y_0)
    &\leq
    - \inner{\nabla \bar H_{k, \mu_k}(y_0) + \eta (\bar x_1 - y_0)}{\bar x_1 - y_0}
    \label{eq:g_diff_upperbound}
  \end{alignat}
  from \cref{eq:in_partialgx} and definition \cref{eq:def_subdif_g} of $\partial g$; third,
  \begin{alignat}{2}
    \bar H_{k, \mu_k}(\bar x_1)
    - \bar H_{k, \mu_k}(y_0)
    \leq
    \inner{\nabla \bar H_{k, \mu_k}(y_0)}{\bar x_1 - y_0} + \frac{\eta}{2} \norm*{\bar x_1 - y_0}^2
    \label{eq:Hbar_diff_upperbound}
  \end{alignat}
  from \cref{alg-line-apg:backtracking}.
  Putting \cref{eq:Hbar_diff_upperbound,eq:g_diff_upperbound} together yields
  \begin{align}
    \bar F_{k, \mu_k}(\bar x_1) - \bar F_{k, \mu_k}(y_0)
    &=
    \prn*{g(\bar x_1) + \bar H_{k, \mu_k}(\bar x_1)}
    - \prn*{g(y_0) + \bar H_{k, \mu_k}(y_0)}\\
    &\leq
    - \frac{\eta}{2} \norm*{\bar x_1 - y_0}^2.
    \label{eq:barF_diff_upperbound}
  \end{align}
  Using this bound, we obtain
  \begin{alignat}{2}
    \norm*{\bar x_1 - x_k}^2
    &\leq
    \frac{2}{\eta}
    \prn*{
      \bar F_{k, \mu_k}(x_k) - \bar F_{k, \mu_k}(\bar x_1)
    }
    &\quad&\by{\cref{eq:barF_diff_upperbound} and $y_0 = x_k$}\\
    &\leq
    \frac{2}{\eta}
    \prn*{
      \bar F_{k, \mu_k}(x_k) - (g^* + h^*)
    }\\
    &=
    \frac{2 \Delta_k}{\eta}
    &\quad&\by{$\bar F_{k, \mu_k}(x_k) = F(x_k)$ and definition \cref{eq:def_Deltak_Fast} of $\Delta_k$}\\
    &\leq
    \frac{2 \sqrt{\Delta_k}}{\rho_{\min}}
    &\quad&\by{$\eta \geq \mu_k$ and \cref{eq:muk_lowerupperbounds}}\\
    &\leq
    \frac{2 \epsilon^3}{\rho_{\min}}
    &\quad&\by{$\Delta_k = \Delta_{N-1} < \epsilon^6$}.
    \label{eq:norm_bx1xk_upperbound}
  \end{alignat}
  Now, we bound $\omega(\bar x_1)$ as follows:
  \begin{alignat}{2}
    \omega(\bar x_1)
    &\leq
    \bar \omega_{k, \mu_k}(\bar x_1)
    + \prn*{ \mu_k + L_c \sqrt{2 L_h \Delta_k} } \norm*{\bar x_1 - x_k}
    + \frac{L_c L_h \sigma}{2} \norm*{\bar x_1 - x_k}^2\\
    &\quad
    \revise{{}+ \frac{L_c^2 L_h}{2} \norm*{\bar x_1 - x_k}^3}
    &\quad&\by{\cref{eq:gradMF_norm_upperbound_tighter}}\\
    &\revise{{}\leq{}
    \prn*{ 
      (2 + \bar \alpha) \revise{(\mu_k + L_h \sigma^2)} + L_c \sqrt{2 L_h \Delta_k}
    } \norm*{\bar x_1 - x_k}
    + \frac{L_c L_h \sigma}{2} \norm*{\bar x_1 - x_k}^2}\\
    &\quad
    \revise{{}+ \frac{L_c^2 L_h}{2} \norm*{\bar x_1 - x_k}^3}
    &\quad&\by{\cref{eq:bokbx1_upperbound}}\\
    &\revise{{}\leq{}
    (2 + \bar \alpha) L_h \sigma^2 \sqrt{\frac{2 \epsilon^3}{\rho_{\min}}}
    + \prn*{ 
      (2 + \bar \alpha) \mu_k + L_c \sqrt{2 L_h \Delta_k}
    } O(\epsilon^{3/2})
    + O(\epsilon^3) }
    &\quad&\by{\cref{eq:norm_bx1xk_upperbound}}\\
    &\revise{{}\leq{}
    (2 + \bar \alpha) L_h \sigma^2 \sqrt{\frac{2 \epsilon^3}{\rho_{\min}}}
    + \prn*{ (2 + \bar \alpha) \rho_{\max} + L_c \sqrt{2 L_h} } \sqrt{\Delta_k} O(\epsilon^{3/2})
    + O(\epsilon^3) }
    &\quad&\by{\cref{eq:muk_lowerupperbounds}}\\
    &\revise{{}\leq{}
    (2 + \bar \alpha) L_h \sigma^2 \sqrt{\frac{2 \epsilon^3}{\rho_{\min}}}
    + O(\epsilon^3) }
    &\quad&\by{$\Delta_k < \epsilon^6$}.
    \label{eq:proof_omegabarx1_upperbound}
  \end{alignat}
  \revise{Therefore, using the last inequality and taking $\epsilon>0$ so small that
  \begin{align}
    L_h \sigma^2 \sqrt{\frac{\epsilon^3}{\rho_{\min}}}
    \lesssim
    \epsilon,
    \quad\text{i.e.,}\quad
    \epsilon
    \lesssim
    \frac{\rho_{\min}}{(L_h \sigma^2)^2},\label{eq:lesssim}
  \end{align}
  we ensure $\omega(\bar x_1) \leq \epsilon$.}
  \revise{
  We now have proved that the overall oracle complexity is equal to \cref{eq:complexity_except_final}, i.e., \cref{eq:oracle_complexity_general}.
  If we set $\rho_{\min} = \Theta \prn*{L_c \sqrt{L_h}}$, this complexity simplifies to
  \begin{align}
    &\mathInd
    O \Bigg(
      \frac{\sqrt{\Delta_0}}{\epsilon^2}
      \prn*{ F(x_0) - F^* }
      \prn*{
        \rho_{\max}
        + \frac{L_c^2 L_h}{\rho_{\min}}
      }
      \prn*{
        1 + \log \prn*{ \frac{L_c \sqrt{L_h}}{\rho_{\min}} }
      }\\
      &\qquad \quad \times
      \prn*{
        1 + \sqrt{1 + \frac{L_h \sigma^2}{\rho_{\min} \sqrt{\Delta_0}}}
        \log \prn*{
          1 + \frac{L_h \sigma^2}{\rho_{\min} \sqrt{\Delta_0}}
        }
      }
    \Bigg)\\
    &=
    O \prn*{
      \frac{L_c \sqrt{L_h \Delta_0}}{\epsilon^2}
      \prn*{ F(x_0) - F^* }
      \prn*{
        1 + \sqrt{1 + \frac{\sqrt{L_h} \sigma^2}{L_c \sqrt{\Delta_0}}}
        \log \prn*{
          1 + \frac{\sqrt{L_h} \sigma^2}{L_c \sqrt{\Delta_0}}
        }
      }
    }
  \end{align}
  since $\rho_{\max} = \Theta \prn*{L_c \sqrt{L_h}}$.
  We thus obtain \cref{eq:oracle_complexity_simple}.
  }
\end{proof}

\section{Proofs for Section~\ref{sec:local_convergence}}
Under \cref{asm:hc}, we have
\begin{align}
  \norm*{\nabla h(c(x^*))} \leq \sqrt{2 L_h \Delta_\infty},\quad
  \forall x^* \in X^*
  \label{eq:grad-h_norm_upperbound}
\end{align}
as
\begin{alignat}{2}
  \frac{1}{2 L_h} \norm*{\nabla h(c(x^*))}^2
  &\leq
  h(c(x^*)) - h^*
  &\quad&\by{\cref{eq:lem_smoothness_obj_grad-norm_bound}}\\
  &\leq
  h(c(x^*)) - h^* + g(x^*) - g^*
  &\quad&\by{\revise{\cref{eq:requirement_gast_hast}}}\\
  &=
  F(x^*) - (g^* + h^*)\\
  &\leq
  F_\infty - (g^* + h^*)
  =
  \Delta_\infty
  &\quad&\by{definition \cref{eq:def_Xast} of $X^*$}.
\end{alignat}
We will use \cref{eq:grad-h_norm_upperbound} to prove \cref{lem:update_leq_dist}.

\subsection{Proof of Lemma~\ref{lem:update_leq_dist}}
\label{sec:proof_update_leq_dist}
\begin{proof}
  We fix a point $x^*_k \in X^*$ closest to $x_k$ arbitrarily.
  To simplify notation, let
  \begin{align}
    u \coloneqq x_{k+1} - x_k,\quad
    v \coloneqq x^*_k - x_k,\quad
    w \coloneqq c(x_k) + \nabla c(x_k) u,\quad
    z \coloneqq c(x_k) + \nabla c(x_k) v.
    \label{eq:def_uvwz}
  \end{align}
  We will prove the result by deriving upper and lower bounds for $\bar F_{k, \mu_k}(x^*_k) - \bar F_{k, \mu_k}(x_{k+1})$.
  We have the following lower bound:
  \begin{alignat}{2}
    &\mathInd
    \bar F_{k, \mu_k}(x^*_k) - \bar F_{k, \mu_k}(x_{k+1})\\
    &\geq
    \frac{\mu_k}{2} \norm*{x^*_k - x_{k+1}}^2
    - \bar{\omega}_{k, \mu}(x_{k+1}) \norm*{x^*_k - x_{k+1}}
    &\quad&\by{\cref{eq:proj-grad_norm_lowerbound_sub}}\\
    &\geq
    \frac{\mu_k}{2} \norm*{x^*_k - x_{k+1}}^2
    - \theta \mu_k \norm*{x_{k+1} - x_k} \norm*{x^*_k - x_{k+1}}
    &\quad&\by{$x_{k+1} \in S_{k, \mu_k}$}\\
    &=
    \frac{\mu_k}{2} \norm*{u - v}^2
    - \theta \mu_k \norm*{u} \norm*{u - v}.
    \label{eq:Fbar_diff_lowerbound}
  \end{alignat}
  To derive the upper bound, we decompose $\bar F_{k, \mu_k}(x^*_k) - \bar F_{k, \mu_k}(x_{k+1})$ as
  \begin{alignat}{2}
    \bar F_{k, \mu_k}(x^*_k) - \bar F_{k, \mu_k}(x_{k+1})
    &=
    \prn*{
      g(x^*_k) - g(x_{k+1}) + \inner*{\nabla h(c(x^*_k))}{z - w}
    }\\
    &\quad
    + \prn*{
      h(z) - h(w) - \inner*{\nabla h(c(x^*_k))}{z - w}
    }
    + \frac{\mu_k}{2} \prn*{ \norm*{v}^2 - \norm*{u}^2 }.
    \label{eq:Fbar_diff_decomposition}
  \end{alignat}
  Each term is bounded as follows: first,
  \begin{alignat}{2}
    &\mathInd
    g(x^*_k) - g(x_{k+1}) + \inner*{\nabla h(c(x^*_k))}{z - w}\\
    &\leq
    \inner*{\nabla H(x^*_k)}{x_{k+1} - x^*_k} + \inner*{\nabla h(c(x^*_k))}{z - w}
    &\quad&\by{\cref{eq:proj-grad_norm_lowerbound} and $\omega(x^*_k) = 0$}\\
    &=
    \inner*{\nabla h(c(x^*_k))}{\prn*{\nabla c(x^*_k) - \nabla c(x_k)} (u - v)}
    &\quad&\by{$\nabla H(x^*_k) = \nabla c(x^*_k)^\top \nabla h(c(x^*_k))$\\and \cref{eq:def_uvwz}}\\
    &\leq
    \norm*{\nabla h(c(x^*_k))}
    \normop*{\nabla c(x^*_k) - \nabla c(x_k)}
    \norm*{u - v}\\
    &\leq
    \sqrt{2 L_h \Delta_\infty} \times L_c \norm*{v} \times \norm*{u - v}
    &\quad&\by{\cref{eq:grad-h_norm_upperbound,asm:nabla-c_lip}}\\
    &\leq
    \frac{L_c \sqrt{2 L_h}}{\rho_{\min}} \mu_k \norm*{v} \norm*{u - v}
    &\quad&\by{$\Delta_\infty \leq \Delta_k$ and \cref{eq:muk_lowerupperbounds}};
    \label{eq:Fbar_diff_decomposition_upperbound1}
  \end{alignat}
  second,
  \begin{alignat}{2}
    h(z) - h(w) - \inner*{\nabla h(c(x^*_k))}{z - w}
    &\leq
    \frac{L_h}{2} \norm*{z - c(x^*_k)}^2
    &\quad&\by{\cref{eq:lem_three-points}}\\
    &=
    \frac{L_h}{2} \norm*{ c(x_k) + \nabla c(x_k) v - c(x^*_k) }^2\\
    &\leq
    \frac{L_h}{2} \prn*{ \frac{L_c}{2} \norm*{v}^2 }^2
    &\quad&\by{\cref{eq:property_nabla-c_lip}}.
    \label{eq:Fbar_diff_decomposition_upperbound2}
  \end{alignat}
  Let
  \begin{align}
    \phi
    \coloneqq
    \frac{L_c \sqrt{2 L_h}}{\rho_{\min}}.
    \label{eq:def_phi}
  \end{align}
  Putting \cref{eq:Fbar_diff_lowerbound,eq:Fbar_diff_decomposition,eq:Fbar_diff_decomposition_upperbound1,eq:Fbar_diff_decomposition_upperbound2} together yields 
  \begin{align}
    \frac{\mu_k}{2} \norm*{u - v}^2
    - \theta \mu_k \norm*{u} \norm*{u - v}
    &\leq
    \bar F_{k, \mu_k}(x^*_k) - \bar F_{k, \mu_k}(x_{k+1})\\
    &\leq
    \phi \mu_k \norm*{v} \norm*{u - v}
    + \frac{L_h}{2} \prn*{ \frac{L_c}{2} \norm*{v}^2 }^2
    + \frac{\mu_k}{2} \prn*{ \norm*{v}^2 - \norm*{u}^2 }.
  \end{align}
  Rearranging the terms and using $2ab - b^2 \leq a^2$ for $a, b \in \R$, we obtain
  \begin{align}
    \norm*{u}^2
    &\leq
    \norm*{v}^2 + \frac{L_c^2 L_h}{4\mu_k} \norm*{v}^4
    + 2 \prn*{
      \theta \norm*{u} + \phi \norm*{v}
    } \norm*{v - u}
    - \norm*{v - u}^2\\
    &\leq
    \norm*{v}^2 + \frac{L_c^2 L_h}{4\mu_k} \norm*{v}^4
    + \prn*{
      \theta \norm*{u} + \phi \norm*{v}
    }^2.
    \label{eq:u_norm_upperbound}
  \end{align}
  
  Here, from $\Delta_\infty \geq 0$ and \cref{asm:holderian_growth}, we have
  \begin{align}
    \Delta_k
    \geq
    \Delta_k
    -
    \Delta_\infty
    = 
    F(x_k) - F_\infty
    \geq
    \frac{\gamma}{r} \norm*{v}^r.
    \label{eq:Deltak_lowerbound}
  \end{align}
  Using this bound, we bound the second term of \cref{eq:u_norm_upperbound} as follows:
  \begin{alignat}{2}
    \frac{L_c^2 L_h}{4 \mu_k} \norm*{v}^2
    &\leq
    \frac{L_c^2 L_h}{4 \rho_{\min} \sqrt{\Delta_k}} \norm*{v}^2
    &\quad&\by{\cref{eq:muk_lowerupperbounds}}\\
    &\leq
    \frac{L_c^2 L_h}{4 \rho_{\min} \sqrt{\gamma / r}} \norm*{v}^{2 - r/2}
    &\quad&\by{\cref{eq:Deltak_lowerbound}}\\
    &\leq
    \frac{L_c^2 L_h}{4 \rho_{\min} \sqrt{\gamma / r}} 
    \times \frac{8 \sqrt{2 \gamma}}{L_c \sqrt{r L_h}}
    =
    2 \phi
    &\quad&\by{\cref{eq:asm_vk_norm_upperbound} and definition \cref{eq:def_phi} of $\phi$},
  \end{alignat}
  and we thus have
  \begin{align}
    \norm*{u}^2
    \leq
    \norm*{v}^2
    + 2 \phi \norm*{v}^2
    + \prn*{
      \theta \norm*{u}
      + \phi \norm*{v}
    }^2.
  \end{align}
  As this inequality is quadratic for $\norm*{u}$, we can solve it as follows:
  \begin{align}
    \norm*{u}
    &\leq
    \frac{1}{1 - \theta^2}
    \prn*{
      \theta \phi
      + \sqrt{
        (\theta \phi)^2
        + (1 - \theta^2) (1 + \phi)^2
      }
    }
    \norm*{v}\\
    &\leq
    \frac{1}{1 - \theta^2}
    \prn*{
      \theta \phi
      + \prn*{
        \phi + \sqrt{1 - \theta^2}
      }
    }
    \norm*{v}
    = C_1 \norm*{v},
  \end{align}
  which completes the proof.
\end{proof}

\subsection{Proof of Proposition~\ref{prop:local_convergence}}
\label{sec:proof_prop_localconv}
\begin{proof}
  We fix a point $x^*_{k+1} \in X^*$ closest to $x_{k+1}$ arbitrarily.
  To simplify notation, let
  \begin{align}
    u
    \coloneqq
    x_{k+1} - x_k,\quad
    v^+
    \coloneqq
    x^*_{k+1} - x_{k+1}.
  \end{align}
  We will prove \cref{eq:local_convergence_rate} assuming $\delta_{k+1} \neq 0$.
  First, since $F(x^*_{k+1}) \leq F_\infty$ and $h$ is convex, we have
  \begin{alignat}{2}
    \delta_{k+1}
    &=
    F(x_{k+1}) - F_\infty\\
    &\leq
    g(x_{k+1}) - g(x^*_{k+1}) + h(c(x_{k+1})) - h(c(x^*_{k+1}))\\
    &\leq
    g(x_{k+1}) - g(x^*_{k+1})
    + \inner*{\nabla h(c(x_{k+1}))}{c(x_{k+1}) - c(x^*_{k+1})}\\
    &=
    g(x_{k+1}) - g(x^*_{k+1})
    - \inner*{\nabla H(x_{k+1})}{v^+}\\
    &\quad
    + \inner*{\nabla h(c(x_{k+1}))}{c(x_{k+1}) + \nabla c(x_{k+1}) v^+ - c(x^*_{k+1})}.
    \label{eq:gnv_upperbound}
  \end{alignat}
  Each term of \cref{eq:gnv_upperbound} is bounded as follows:
  first,
  \begin{alignat}{2}
    &\mathInd
    g(x_{k+1}) - g(x^*_{k+1})
    - \inner*{\nabla H(x_{k+1})}{v^+}\\
    &\leq
    \omega(x_{k+1})
    \norm{v^+}
    &\quad&\by{\cref{lem:proj-grad_norm_lowerbound}}\\
    &\leq
    \prn*{
      (1 + \theta) \mu_k + L_c \sqrt{2 L_h \Delta_k}
      + \frac{L_c L_h \sigma}{2} \norm*{u}
    }
    \norm*{u} \norm*{v^+}
    &\quad&\by{\cref{eq:gradMF_norm_upperbound}}\\
    &\leq
    \prn*{
      \prn*{ (1 + \theta) \rho_{\max} + L_c \sqrt{2 L_h} } \sqrt{\Delta_k}
      + \frac{L_c L_h \sigma}{2} \norm*{u}
    }
    \norm*{u} \norm*{v^+}
    &\quad&\by{\cref{eq:muk_lowerupperbounds}}\\
    &=
    \prn*{
      C_2 \sqrt{\Delta_k}
      + \frac{L_c L_h \sigma}{2} \norm*{u}
    }
    \norm*{u} \norm*{v^+}
    &\quad&\by{definition \cref{eq:def_C23} of $C_2$};
    \label{eq:inner_gradh_upperbound}
  \end{alignat}
  second,
  \begin{alignat}{2}
    &\mathInd
    \inner*{\nabla h(c(x_{k+1}))}{c(x_{k+1}) + \nabla c(x_{k+1}) v^+ - c(x^*_{k+1})}\\
    &\leq
    \norm*{\nabla h(c(x_{k+1}))}
    \norm*{c(x_{k+1}) + \nabla c(x_{k+1}) v^+ - c(x^*_{k+1})}\\
    &\leq
    \sqrt{2 L_h \Delta_{k+1}} \times \frac{L_c}{2} \norm*{v^+}^2
    &\quad&\by{\cref{eq:norm_gradhcxk_upperbound,eq:property_nabla-c_lip}}\\
    &=
    \frac{L_c}{2} \sqrt{2 L_h (\delta_{k+1} + \Delta_\infty)} \norm*{v^+}^2
    &\quad&\by{definitions \cref{eq:def_Deltak_Fast,eq:def_Deltainf,eq:def_deltak}\\of $\Delta_k$, $\Delta_\infty$, and $\delta_k$}\\
    &\leq
    \frac{\delta_{k+1}}{2}
    \prn*{\frac{r \delta_{k+1}}{\gamma}}^{-2/r} \norm*{v^+}^2
    &\quad&\by{\cref{eq:asm_vk_norm_upperbound2}}.
    \label{eq:inner_gradh_upperbound2}
  \end{alignat}
  Plugging \cref{eq:inner_gradh_upperbound,eq:inner_gradh_upperbound2} into \cref{eq:gnv_upperbound} yields
  \begin{align}
    \delta_{k+1}
    \leq
    \prn*{
      C_2 \sqrt{\Delta_k}
      + \frac{L_c L_h \sigma}{2} \norm*{u}
    }
    \norm*{u} \norm*{v^+}
    +
    \frac{\delta_{k+1}}{2}
    \prn*{\frac{r \delta_{k+1}}{\gamma}}^{-2/r}
    \norm*{v^+}^2.
    \label{eq:deltak1_upperbound}
  \end{align}

  Here, from \cref{asm:holderian_growth,lem:update_leq_dist}, we have
  \begin{align}
    \norm*{u}
    &\leq
    C_1 \dist(x_k, X^*)
    \leq
    C_1
    \prn*{
      \frac{r \delta_k}{\gamma}
    }^{1/r},\\
    \norm*{v^+}
    &=
    \dist(x_{k+1}, X^*)
    \leq
    \prn*{
      \frac{r \delta_{k+1}}{\gamma}
    }^{1/r}.
  \end{align}
  Plugging these bounds into \cref{eq:deltak1_upperbound}, we obtain
  \begin{align}
    \delta_{k+1}
    \leq
    \prn*{
      C_2 \sqrt{\Delta_k}
      + \frac{L_c L_h \sigma}{2}
      C_1
      \prn*{
        \frac{r \delta_k}{\gamma}
      }^{1/r}
    }
    C_1
    \prn*{
      \frac{r \delta_k}{\gamma}
    }^{1/r}
    \prn*{
      \frac{r \delta_{k+1}}{\gamma}
    }^{1/r}
    +
    \frac{\delta_{k+1}}{2},
  \end{align}
  and rearranging the terms yields
  \begin{align}
    \delta_{k+1}^{1-1/r}
    &\leq
    \prn*{
      2 C_1 C_2 
      \prn*{
        \frac{r}{\gamma}
      }^{2/r}
      \sqrt{\Delta_k}
      +
      2 C_1^2
      L_c L_h \sigma
      \prn*{
        \frac{r}{\gamma}
      }^{3/r}
      \delta_k^{1/r}
    }
    \delta_k^{1/r}\\
    &=
    \prn*{
      C_3
      \sqrt{\delta_k + \Delta_\infty}
      +
      C_4
      \delta_k^{1/r}
    }
    \delta_k^{1/r},
  \end{align}
  which completes the proof.
\end{proof}

\section{Complexity of proximal gradient method}
\label{sec:complexity_pg}
We describe the complexity of the proximal gradient (PG) method in \cref{table:complexity}.
We define a function $H$ by \cref{eq:def_H_Hbar} and suppose that \cref{asm:hc} holds.

The oracle complexity and iteration complexity of PG methods are equal since the method computes the gradient and the proximal mapping only once per iteration.
The following result on the iteration complexity is standard (see, e.g., \citep[Theorem~10.15]{beck2017first}): the iteration complexity of the proximal gradient method is
\begin{align}
  O \prn*{ L_H \epsilon^{-2} (F(x_0) - F^*) },
  \label{eq:complexity_gd_general}
\end{align}
where $L_H$ is the Lipschitz constant of $\nabla H$.

The following proposition provides upper bounds on the Lipschitz constant $L_H$, which yields the complexity bounds listed in \cref{table:complexity}.
We can check that these upper bounds are tight for some instances but omit the details.
\begin{proposition}
  Suppose that \cref{asm:hc} holds and that for some $\sigma > 0$,
  \begin{align}
    \normop*{\nabla c(x)} \leq \sigma,\quad
    \forall x \in \dom g.
    \label{eq:asm_normop_nablac}
  \end{align}
  Then, the following holds:
  \begin{align}
    \norm*{
      \nabla H(y) - \nabla H(x)
    }
    &\leq
    \prn*{
      L_h \sigma^2 + L_c \sqrt{2 L_h \Delta_0}
    }
    \norm*{y - x},\quad
    \forall x, y \in \dom g
    \ \text{s.t.}\ 
    F(x) \leq F(x_0),
    \label{eq:norm_diff_nablaH_upperbound}
  \end{align}
  where $\Delta_0$ is defined by \cref{eq:def_Deltak_Fast}.
  Furthermore, if $h$ is $K_h$-Lipschitz continuous,
  then the following holds:
  \begin{align}
    \norm*{
      \nabla H(y) - \nabla H(x)
    }
    &\leq
    \prn*{
      L_h \sigma^2 + L_c K_h
    }
    \norm*{y - x},\quad
    \forall x, y \in \dom g.
    \label{eq:norm_diff_nablaH_upperbound2}
  \end{align}
\end{proposition}
\begin{proof}
  As
  \begin{align}
    \nabla H(y) - \nabla H(x)
    &=
    \nabla c(y)^\top \nabla h(c(y))
    - \nabla c(x)^\top \nabla h(c(x))\\
    &=
    \nabla c(y)^\top \prn*{ \nabla h(c(y)) - \nabla h(c(x)) }
    + \prn*{\nabla c(y) - \nabla c(x)}^\top \nabla h(c(x)),
  \end{align}
  we have
  \begin{align}
    \norm*{
      \nabla H(y) - \nabla H(x)
    }
    &\leq
    \normop*{ \nabla c(y) } \norm*{ \nabla h(c(y)) - \nabla h(c(x)) }
    + \normop*{\nabla c(y) - \nabla c(x)} \norm*{\nabla h(c(x))}\\
    &\leq
    L_h \sigma \norm*{ c(y) - c(x) }
    + L_c \norm*{\nabla h(c(x))} \norm*{y - x}.
    \label{eq:norm_diff_nablaH_upperbound3}
  \end{align}
  It follows from the mean value theorem and \cref{eq:asm_normop_nablac} that $\norm*{ c(y) - c(x) } \leq \sigma \norm*{y - x}$, and it follows, similarly to \cref{eq:norm_gradhcxk_upperbound}, that
  \begin{align}
    \norm*{\nabla h(c(x))}
    \leq
    \sqrt{2 L_h \prn*{ F(x) - (g^* + h^*) }}
    \leq
    \sqrt{2 L_h \prn*{ F(x_0) - (g^* + h^*) }}
    =
    \sqrt{2 L_h \Delta_0}.
    \label{eq:norm_nablahcx_upperbound}
  \end{align}
  for $x \in \dom g$ such that $F(x) \leq F(x_0)$.
  Plugging these bounds into \cref{eq:norm_diff_nablaH_upperbound3} yields \cref{eq:norm_diff_nablaH_upperbound}.
  We also obtain \cref{eq:norm_diff_nablaH_upperbound2} by using
  \begin{align}
    \norm*{ \nabla h(z) }
    =
    \lim_{t \to 0}
    \frac{\abs*{ h(z + t \nabla h(z)) - h(z) }}{ t \norm*{ \nabla h(z) } }
    \leq
    K_h
  \end{align}
  with $z = c(x)$, which follows from the Lipchitz continuity of $h$, in \cref{eq:norm_diff_nablaH_upperbound3}.
\end{proof}

Bounds \cref{eq:complexity_gd_general,eq:norm_diff_nablaH_upperbound} yield the complexity $O \prn*{ \sqrt{L_h \Delta_0} \kappa }$ of PG in \cref{table:complexity} as follows:
\begin{align}
  \frac{L_H \epsilon^{-2} (F(x_0) - F^*)}{L_c \epsilon^{-2} (F(x_0) - F^*)}
  =
  \frac{L_H}{L_c}
  \leq
  \frac{L_h \sigma^2 + L_c \sqrt{2 L_h \Delta_0}}{L_c}
  =
  O \prn[\big]{ \sqrt{L_h \Delta_0} \kappa },
\end{align}
where $\kappa$ is defined in \cref{eq:def_kappa_sigma}.

\section{Details of implementation for numerical experiments}
\label{sec:implementation}
In the numerical experiments, \revise{we implemented all the methods using automatic differentiation with JAX~\citep{jax2018github}.}
\revise{The LM methods, \revise{\DP and \Proposed,}} access the oracles several times to compute the Jacobian-vector products $\nabla c(x) u$ and $\nabla c(x)^\top v$ to solve subproblem~\cref{eq:subproblem}.
Note that $u$ and $v$ take on various vectors, but $x$ is fixed at $x_k$.
In such cases, using the \texttt{jax.linearize} method is efficient, which generates the function $u \mapsto \nabla c(x_k) u$, and we can also get the function $v \mapsto \nabla c(x_k)^\top v$ using the \texttt{jax.linear\_transpose} method.
We used these two methods to implement the LM methods.

\revise{
For \ABO method, we implemented a compact representation of BFGS update \citep[p.~181]{nocedal2006numerical}.
In the BFGS computation (\citep[Eq.~(7.29)]{nocedal2006numerical}), we used LU factorization via \texttt{scipy.linalg.lu\_factor}.
}



\end{document}